\documentclass[leqno,12pt]{amsart}
\usepackage[colorlinks,pagebackref,hypertexnames=false]{hyperref}
\hypersetup{
	allcolors = blue
}
\usepackage{amsmath,amsthm,amssymb,mathrsfs,mathtools}
\usepackage[alphabetic,backrefs]{amsrefs}
\usepackage[T1]{fontenc}
\usepackage[english]{babel}
\usepackage[margin=1in]{geometry}
\usepackage[onehalfspacing]{setspace}

\usepackage{xcolor}

\newcounter{dummy}
\usepackage{enumitem}
\makeatletter
\newcommand\myitem[1][]{\item[#1]\refstepcounter{dummy}\def\@currentlabel{#1}}
\makeatother
\usepackage{ae,aecompl}
\usepackage{enumerate}
\usepackage{multicol}
\usepackage[normalem]{ulem}
\usepackage{tabularx}
\numberwithin{equation}{section}							
\usepackage[nameinlink,noabbrev]{cleveref}
\usepackage{autonum}										

\let\originalleft\left
\let\originalright\right
\renewcommand{\left}{\mathopen{}\mathclose\bgroup\originalleft}
\renewcommand{\right}{\aftergroup\egroup\originalright}

\makeatletter
\renewcommand*{\eqref}[1]{\hyperref[{#1}]{\textup{\tagform@{\ref*{#1}}}}}		
\makeatother

\usepackage{cleveref}

\newcommand{\ad}{\mathrm{ad}}

\newcommand{\del}{\partial}

\newcommand{\tr}{\mathop{\mathrm{tr}}\nolimits}

\def\ad{\mathrm{ad}}

\def\tr{\mathrm{tr}}

\def\<{\mathopen{}\left<}
\def\>{\right>\mathclose{}}
\def\({\mathopen{}\left(}
\def\){\right)\mathclose{}}

\newtheorem{theorem}{Theorem}[section]

\newtheorem{corollary}[theorem]{Corollary}

\newtheorem{lemma}[theorem]{Lemma}
\newtheorem{proposition}[theorem]{Proposition}
\theoremstyle{definition} \newtheorem{definition}[theorem]{Definition}

\newtheorem{remark}[theorem]{Remark}
\newtheorem{notation}[theorem]{Notation}

\crefname{theorem}{Theorem}{Theorems}						
\creflabelformat{theorem}{#2{#1}#3}							
\crefname{Mtheorem}{Main Theorem}{Main Theorems}			
\creflabelformat{Mtheorem}{#2{#1}#3}						
\crefname{lemma}{Lemma}{Lemmata}							
\creflabelformat{lemma}{#2{#1}#3}							
\crefname{corollary}{Corollary}{Corollaries}				
\creflabelformat{corollary}{#2{#1}#3}						
\crefname{proposition}{Proposition}{Propositions}			
\creflabelformat{proposition}{#2{#1}#3}						
\crefname{ineq}{inequality}{inequalities}					
\creflabelformat{ineq}{#2{\upshape(#1)}#3}					
\crefname{cond}{condition}{conditions}						
\creflabelformat{cond}{#2{\upshape(#1)}#3}					
\crefname{hypoth}{Hypothesis}{Hypotheses}					
\creflabelformat{hypoth}{#2{#1}#3}							
\crefname{def}{Definition}{Definitions}						
\creflabelformat{def}{#2{#1}#3}								
\crefname{appsec}{Appendix}{Appendices}

\crefname{sec}{Section}{Sections}

\begin{document}

\author{Daniel Fadel} 
\address[Daniel Fadel]{Universidade Estadual de Campinas, Campinas, Brazil / Université de Bretagne Occidentale, Brest, France}
\urladdr{\href{https://sites.google.com/view/daniel-fadel-math-homepage/home}{sites.google.com/view/daniel-fadel-math-homepage/home}}
\email{\href{mailto:fadel.daniel@gmail.com}{fadel.daniel@gmail.com}}


\date{\today}

\title{Asymptotics of finite energy monopoles on AC $3$-manifolds}
\maketitle

\begin{abstract}
    We study the asymptotic behavior of finite energy $\rm SU(2)$ monopoles, and general critical points of the $\rm SU(2)$ Yang--Mills--Higgs energy, on asymptotically conical $3$-manifolds with only one end. Our main results generalize classical results due to Groisser and Taubes in the particular case of the flat $3$-dimensional Euclidean space $\mathbb{R}^3$. Indeed, we prove the integrality of the monopole number, or charge, of finite energy configurations, and derive the classical energy formula establishing monopoles as absolute minima. Moreover, we prove that the covariant derivative of the Higgs field of a critical point of the energy decays quadratically along the end, and that its transverse component with respect to the Higgs field, as well as the corresponding component of the curvature of the underlying connection, actually decay exponentially. Additionally, under the assumption of positive Gaussian curvature on the asymptotic link, we prove that the curvature of any critical point connection decays quadratically. Furthermore, we deduce that any irreducible critical point converges uniformly along the conical end to a limiting configuration at infinity consisting of a reducible Yang--Mills connection and a parallel Higgs field. 
\end{abstract}

\begingroup
\hypersetup{linkcolor=black}
\tableofcontents
\endgroup

\section{Introduction}

\subsection{Background}\label{subsec: background}

Given a complete, noncompact, connected and oriented Riemannian $3$-manifold $(X^3,g)$, and a principal $G$-bundle $P$ over $X^3$, where $G$ is a compact Lie group, we shall consider pairs $(A,\Phi)$ consisting of a smooth connection $A\in\mathscr{A}(P)$ on $P$ and a smooth \emph{Higgs field} $\Phi\in\Gamma(\mathfrak{g}_P)$, \emph{i.e.} a smooth section of the associated adjoint bundle $\mathfrak{g}_P:=P\times_{\mathrm{Ad}}\mathfrak{g}$. Such a pair $(A,\Phi)$ is called a \emph{monopole} if it is a solution to the \emph{Bogomolnyi equation}:
\begin{equation}\label{eq: monopole}
F_A = \ast\nabla_A\Phi.
\end{equation} Here $F_A\in\Omega^2(X,\mathfrak{g}_P)$ denotes the curvature of the connection $A$, while $\nabla_A\Phi\in\Omega^1(X,\mathfrak{g}_P)$ is the covariant derivative of $\Phi$ with respect to the connection induced by $A$ on $\mathfrak{g}_P$, and $\ast$ stands for the Hodge star operator induced by the metric $g$.


Combining \eqref{eq: monopole} with the \emph{Bianchi identity}, ${d}_AF_A=0$, it readily follows that monopoles are solutions to the second order equations\footnote{Here $\Delta_A\Phi := d_A^{\ast}d_A\Phi = \nabla_A^{\ast}\nabla_A\Phi$ since $\Phi\in\Omega^0(X,\mathfrak{g}_P)=\Gamma(\mathfrak{g}_P)$; see \S\ref{sec: notations}.}\begin{subequations}
	\begin{align}
		\Delta_A \Phi &= 0, \label{eq:2nd_Order_Eq_1}\\
		d_A^* F_A	&= [\nabla_A \Phi, \Phi], \label{eq:2nd_Order_Eq_2}
	\end{align}
\end{subequations}
which correspond to the Euler--Lagrange equations of the \emph{Yang--Mills--Higgs energy} functional
\begin{equation}\label{eq: finite_energy}
\mathcal{E}_X(A,\Phi):=\frac{1}{2}\int_X |F_A|^2+|\nabla_A\Phi|^2,   
\end{equation} defined over the \emph{configuration space}
\[
\mathscr{C}(P):=\{(A,\Phi)\in \mathscr{A}(P)\times \Gamma(\mathfrak{g}_P): |F_A|,|\nabla_A\Phi|\in L^2(X)\}.
\] Here the norms $|\cdot{}|$ are induced by $g$ together with a metric on $\mathfrak{g}_P$ arising from a choice of an Ad$_G$-invariant inner product on the compact Lie algebra $\mathfrak{g}$ of $G$. In this paper, we shall mainly restrict ourselves to the structure group $G=\rm SU(2)$, in which case we fix the metric on $\mathfrak{g}_P$ to be the one arising from the inner product $(a,b)\mapsto -2\text{tr}(ab)$ on the Lie algebra $\mathfrak{g}=\mathfrak{su}(2)$. 

Thus, finite energy monopoles are in particular critical points of $\mathcal{E}_X:\mathscr{C}(P)\to [0,\infty)$. Now note that equation \eqref{eq:2nd_Order_Eq_1} implies
\begin{equation}\label{eq: subharmonic}
\Delta\lvert\Phi\rvert^2 = -2\lvert\nabla_A\Phi\rvert^2\leqslant 0.
\end{equation} As a consequence, if $(A,\Phi)$ is a solution to the second order equations \eqref{eq:2nd_Order_Eq_1} and \eqref{eq:2nd_Order_Eq_2} (\emph{e.g.} a monopole) and if $X$ was to be a compact manifold (without boundary), then $|\Phi|$ would be constant, $\nabla_A \Phi = 0$ and $A$ would be a \emph{Yang--Mills connection}, \emph{i.e.} $d_A^{\ast} F_A=0$ (note that $A$ would be flat, $F_A=0$, in the monopole case); in particular, $A$ would be reducible if $\Phi\neq 0$. Since we are interested in \emph{irreducible} critical points of $\mathcal{E}_X$, meaning those satisfying $\nabla_A\Phi\neq 0$, it follows that we must assume $X$ is \emph{noncompact}\footnote{Note that we are restricting ourselves to \emph{smooth} configurations $(A,\Phi)$ on a manifold $X$ \emph{without boundary}.}. 
	
In this article we shall focus our study on asymptotically conical $3$-manifolds. We say that $(X^3,g)$ is \emph{asymptotically conical} (AC) with rate $\nu>0$ if there exist a compact set $K\subset X$, a closed, connected and oriented Riemannian surface $(\Sigma^2,g_{\Sigma})$, and an orientation preserving diffeomorphism
	\[
	\varphi:C(\Sigma):=(1,\infty)_r\times \Sigma\to X\setminus K
	\] such that the cone metric $g_C:= dr^2+r^2g_{\Sigma}$ on $C(\Sigma)$ and its Levi--Civita connection $\nabla_C$ satisfy
	\[
	\lvert\nabla_{C}^j\left(\varphi^{\ast}g - g_C\right)\rvert_C = O(r^{-\nu - j}),\quad\forall j\in\mathbb{N}_0.\quad\text{(as $r\to\infty$)}
	\] Note that we impose, for simplicity, that $X$ has only one end ($\Sigma$ is connected); $X\setminus K$ is called the \emph{(conical) end} of $X$, while $\Sigma$ is called the asymptotic \emph{link}. A \emph{radius function} on $X$ is any smooth extension $\rho:X\to [1,\infty)$ of $r\circ\varphi^{-1}|_{\varphi([2,\infty)\times\Sigma)}$; in particular, for any reference point $o\in K$ we have $\rho(x)\sim (1+d(x,o)^2)^{1/2}$.
    
    The flat $3$-dimensional Euclidean space $(\mathbb{R}^3,g_{\mathbb{R}^3})$ is the model example of an AC $3$-manifold with only one end, where the link $(\Sigma^2,g_{\Sigma})=(\mathbb{S}^2,g_{\mathbb{S}^2})$ is the round $2$-sphere and the rate $\nu=\infty$; in fact, we can write $g_{\mathbb{R}^3}=dr^2 + r^2g_{\mathbb{S}^2}$ on $\mathbb{R}^3\setminus\{0\}\cong\mathbb{R}^{+}\times\mathbb{S}^2$. The literature on monopoles in $\mathbb{R}^3$ is vast, the theory has been developed since the mid 1970s by both physicists and mathematicians; we refer the reader to the standard textbooks \cites{Jaffe1980,Atiyah88} and the references therein for the fundamental classical developments. 
    
    More generally, an AC $3$-manifold with rate $\nu>0$ and asymptotic link $(\Sigma^2,g_{\Sigma})=(\mathbb{S}^2,g_{\mathbb{S}^2})$ the round $2$-sphere is called \emph{asymptotically Euclidean} (AE). The monopole theory on AE manifolds was first investigated by Ernst \cites{Ernst1995,ernst1995ends} and Floer \cites{Floer1995config,Floer1995}. 
    
    We remark that on a general AC $3$-manifold $(X^3,g)$ the Ricci curvature tensor $\mathcal{R}ic_g$ decays (at least) quadratically along the end, \emph{i.e.} $\rho^2|\mathcal{R}ic_g|\leqslant C<\infty$ as $\rho\to\infty$. If $(X^3,g)$ is an AE $3$-manifold $(X^3,g)$ then the Ricci curvature tensor $\mathcal{R}ic_g$ decays faster than quadratically, \emph{i.e.} $\rho^2|\mathcal{R}ic_g|\to 0$ as $\rho\to\infty$ (see Appendix \ref{app: A}). Conversely, an AC $3$-manifold $(X^3,g)$ satisfying the later condition is automatically AE by the works \cites{bando1989construction,tian2005bach}. Also, in this AE case if furthermore the Ricci curvature is nonnegative, \emph{i.e.} $\mathcal{R}ic_g\geqslant 0$, then $(X^3,g)$ is in fact isometric to $(\mathbb{R}^3,g_{\mathbb{R}^3})$. More generally, by combining the results in \cite{zhu1993finiteness,liu20133} one has that if $(X^3,g)$ is AC and $\mathcal{R}ic_g\geqslant 0$ then the manifold $X^3$ is necessarily diffeomorphic to $\mathbb{R}^3$ (see also \cite[Corollary 1.1]{reiris2015ricci}).  
    
    A general AC oriented $3$-manifold $(X^3,g)$ (with only one end) still shares a lot of nice geometric-analytic properties with $(\mathbb{R}^3,g_{\mathbb{R}^3})$. Besides being a manifold of bounded geometry, with quadratically decaying curvature, it has been shown by van Coevering \cite{van2009regularity} that $(X^3,g)$ satisfies the (Euclidean-type) $L^2$-Sobolev inequality, and in particular satisfies the uniform volume growth lower bound $V(x,r):=\text{Vol}_g(B(x,r))\gtrsim r^3$, for all $x\in X$ and $r>0$. Moreover, van Coevering showed that $(X^3,g)$ satisfies a two-sided Gaussian bound on the heat kernel or, equivalently, a uniform parabolic Harnack inequality, so that in particular it satisfies the strong Liouville property -- namely, it admits no nonconstant semibounded harmonic functions. These later properties are key to the analysis of this paper, and in fact we shall be able to prove general results that are valid not only on AC manifolds but also on other general geometries satisfying analogous properties (see Remark \ref{rmk: generalizations}).
    
    The study of monopoles on general AC $3$-manifolds was initiated only recently by the works of Oliveira \cites{oliveira2014thesis,oliveira2016monopoles} and Kottke \cite{kottke2015dimension}; see also \cite{fadel2019limit} and Remark \ref{rmk: large_mass}. In particular, Kottke \cite{kottke2015dimension} computed the virtual dimension of the moduli spaces of $\rm SU(2)$ monopoles on AC $3$-manifolds and Oliveira \cite{oliveira2016monopoles} tackled the problem of existence of such monopoles by proving an AC version of Taubes' original gluing theorem of well-separated multi-monopoles on $\mathbb{R}^3$, giving a construction which covers a smooth open set in the moduli space of $\rm SU(2)$ monopoles on any AC $3$-manifold with vanishing second Betti number.
    
    This paper is dedicated to the study of the asymptotic behavior of finite energy $\rm SU(2)$ monopoles, and more generally of any critical point of the Yang--Mills--Higgs energy, on general AC $3$-manifolds with only one end. In particular, we generalize important classical results due to Groisser and Taubes, namely the smooth version of the main result in \cite{groisser1984integrality}, and Theorems IV.10.3 and IV.10.5 in \cite{Jaffe1980}. 
   
\subsection{Main results}

In the following statements, let $(X^3,g)$ be an AC oriented $3$-manifold with only one end and rate $\nu>0$. Denote the asymptotic link of the conical end by $(\Sigma^2,g_{\Sigma})$ and let $\rho$ be a radius function on $X$. For each $R>0$, we let $B_R:=\{x\in X: \rho(x)<R\}$ and $\overline{B}_R := \{x\in X : \rho(x)\leqslant R\}$. For large enough $R$, each $\overline{B}_R$ is a smooth $3$-manifold with boundary, where $\Sigma_R := \partial \overline{B}_R$ is diffeomorphic to $\Sigma$. 

\begin{theorem}[Finite mass, integrality of charge and energy formula]\label{thm: main_1}
	Let $P\to X$ be a principal $G$-bundle, where $G$ is a compact Lie group, and let $(A,\Phi)\in\mathscr{C}(P)$ be an arbitrary finite energy configuration. Then:
	\begin{itemize}
		\myitem[(i)] \label{itm: i_main_1} There is a unique number $m=m(|\Phi|)\in [0,\infty)$ such that $m - |\Phi|\in L^6(X)$.
	
		\myitem[(ii)] \label{itm: ii_main_1} If $m>0$ and $G=\rm SU(2)$ then the \emph{charge} (or \emph{monopole number}) of $(A,\Phi)$, given by
		\begin{equation}\label{eq: charge_integer}
		k=k(A,\Phi):=\frac{1}{4\pi m}\int_X\langle F_A\wedge\nabla_A\Phi\rangle,
		\end{equation} is an integer, \emph{i.e.} $k\in\mathbb{Z}$.
		
		\myitem[(iii)] \label{itm: iii_main_1} If $(A,\Phi)$ satisfies \eqref{eq:2nd_Order_Eq_1}, \emph{i.e.} if $\Delta_A\Phi=0$, then the Higgs field norm $|\Phi|$ converges uniformly to the constant $m=m(|\Phi|)$ at infinity:
		\begin{equation}\label{eq: finite_mass_main}
			\lim_{\rho\to\infty} |\Phi| = m.
		\end{equation} Moreover, $\|\Phi\|_{L^{\infty}(X)}\leqslant m$; in particular, if $m=0$ then $\Phi=0$.
		
		\myitem[(iv)] \label{itm: iv_main_1} Whenever the uniform convergence \eqref{eq: finite_mass_main} holds we say that $m$ is the \emph{mass} of $(A,\Phi)$. In this case, assuming furthermore that we have $m>0$ and $G=\rm SU(2)$, then the charge \eqref{eq: charge_integer} can be calculated as
	\begin{equation}\label{eq: charge_calculation}
	k = \lim_{R\to\infty}\frac{1}{4\pi}\int_{\Sigma_R}|\Phi|^{-1}\langle\Phi,F_A\rangle.
	\end{equation} Moreover, for every $R\gg_{A,\Phi} 1$, restricting $\Phi/|\Phi|$ to $\Sigma_R\cong\Sigma$ determines a homotopy class of maps $\Sigma^2\to\mathbb{S}^2\subset\mathfrak{su}(2)$, and $k$ is the Brouwer degree of this class. Alternatively, the restrictions of the associated vector bundle $P\times_{\rm SU(2)}\mathbb{C}^2$ over $\Sigma_R$ split as $\mathscr{L}\oplus \mathscr{L}^{-1}$, where $\mathscr{L}$ is a complex line bundle over $\Sigma_{R}\cong\Sigma$, corresponding to one of the eigenspaces of $\Phi$, and the degree of any such $\mathscr{L}$ does not depend on $R$ and equals the charge $k$.
	\end{itemize}  In particular, if $(A,\Phi)$ is a monopole, \emph{i.e.} a solution to equation \eqref{eq: monopole}, and $G=\rm SU(2)$ then the charge $k$ of $(A,\Phi)$ is \emph{a priori} a positive integer $k\in\mathbb{N}$, the Higgs field $\Phi$ must have zeros, \emph{i.e.} $\Phi^{-1}(0)\neq\emptyset$, and the following energy formula holds
\begin{equation}\label{eq: Energy_Formula}
\mathcal{E}_X(A,\Phi) = 4\pi m k.
\end{equation} 
\end{theorem}

\begin{remark}
	Consider $G=\rm SU(2)$ in Theorem \ref{thm: main_1}. Then the energy of any $(A,\Phi)\in\mathscr{C}(P)$ with $m=m(|\Phi|)>0$ and charge $k\in\mathbb{Z}$ is given by
	\begin{equation}\label{eq: general_energy_formula}
		\mathcal{E}_X(A,\Phi)=\pm 4\pi mk + \frac{1}{2}\|F_A \mp \ast\nabla_A\Phi\|_{L^2(X)}^2;
	\end{equation} in particular we have
	\[
	\mathcal{E}_X(A,\Phi)\geqslant 4\pi m|k|.
	\] Therefore, for fixed $m>0$ and $k\in\mathbb{Z}$, the absolute minima of the $\rm SU(2)$ Yang--Mills--Higgs energy $\mathcal{E}_X$ are either solutions to the monopole equation \eqref{eq: monopole} or to the \emph{anti-monopole} equation $F_A = - \ast\nabla_A\Phi$, according to whether $k\geqslant 0$ or $\leqslant 0$ respectively. Since the transformation $(A,\Phi)\mapsto (A,-\Phi)$ gives a one-to-one correspondence between solutions of the monopole equation and solutions of the anti-monopole equation, we concentrate our attention on monopoles.
\end{remark}
 
\begin{remark}\label{rmk: generalizations}
	Theorem \ref{thm: main_1} is a consequence of the main results we prove in Section \ref{sec: mass_charge}, and some of those results are proved in more generality. In particular, we prove that both parts \ref{itm: i_main_1} and \ref{itm: iii_main_1} of Theorem \ref{thm: main_1} also hold \emph{e.g.} on any complete $3$-manifold with nonnegative Ricci curvature and maximal volume growth (see Remark \ref{rmk: nR_case_finite_mass}). Moreover, we prove a very general finite mass result, stated as Theorem \ref{thm: alternative_finite_mass}, implying \ref{itm: iii_main_1} for any smooth configuration $(A,\Phi)\in\mathscr{A}(P)\times\Gamma(\mathfrak{g}_P)$ such that $\nabla_A\Phi\in L^2(X)\cap L^{2(n-1)}(X)$ on any complete nonparabolic $n$-manifold, $n\geqslant 3$, satisfying a uniform parabolic Harnack inequality (PHI). (See \S\ref{subsec: greens_functions} and Theorem \ref{thm: PHI} about (PHI); this property is in particular satisfied by manifolds with nonnegative Ricci curvature and AC manifolds with only one end). In particular, in \S\ref{subsec: e-reg}, using the $\varepsilon$-regularity for critical points of the Yang--Mills--Higgs energy on $3$-manifolds with bounded geometry, we deduce on Corollary \ref{cor: finite_mass_for_crit_pts} that \ref{itm: iii_main_1} holds for any critical point of the energy on any complete nonparabolic $3$-manifold of bounded geometry satisfying (PHI). 
\end{remark} 
 
\begin{theorem}[Asymptotics of finite energy monopoles on AC $3$-manifolds]\label{thm: main_2}
	Let $P\to X$ be a principal $\rm SU(2)$-bundle, and let $(A,\Phi)\in\mathscr{C}(P)$ be a solution to the second order equations \eqref{eq:2nd_Order_Eq_1} and \eqref{eq:2nd_Order_Eq_2}, \emph{i.e.} a critical point of $\mathcal{E}_X:\mathscr{C}(P)\to [0,\infty)$. Denote by $m$ the finite mass of $(A,\Phi)$ given by Theorem \ref{thm: main_1} and suppose that $m>0$. Then there is a constant $R_0=R_0(A,\Phi)>1$ with the following significance.
	\begin{itemize}
		\myitem[(i)] \label{itm: i_main_2} The $\Phi$-transverse components of $F_A$ and $\nabla_A\Phi$ decay exponentially along the end:
		\[
		|[\nabla_A\Phi,\Phi]| + |[F_A,\Phi]| \lesssim_{A,\Phi} e^{-cm(\rho-R_0)}\quad\text{for }\rho\geqslant R_0.
		\]
		\myitem[(ii)] \label{itm: ii_main_2} $m-|\Phi|$ decays linearly along the end; in fact:
		\[
		m-|\Phi|\sim_{A,\Phi}\rho^{-1}\quad\text{for }\rho\geqslant R_0.
		\]
		\myitem[(iii)] \label{itm: iii_main_2} $\nabla_A\Phi$ decays quadratically:
		\[
		|\nabla_A\Phi|\lesssim_{A,\Phi} \rho^{-2}\quad\text{for }\rho\geqslant R_0.
		\]
	\end{itemize}
		Furthermore, if we suppose at least one of the following holds:
		\begin{itemize}
			\myitem[($\dagger$)] \label{itm: a_main_2} $(A,\Phi)$ is a monopole, \emph{i.e.} a solution to equation \eqref{eq: monopole};
			\myitem[($\dagger\dagger$)] \label{itm: b_main_2} $\Sigma$ has positive Gaussian curvature;
		\end{itemize} then\footnote{To be clear: under \emph{any} of the further assumptions \ref{itm: a_main_2} and/or \ref{itm: b_main_2} we get \ref{itm: iv_main_2} and \ref{itm: v_main_2}. Also, \ref{itm: a_main_2} and \ref{itm: b_main_2} are not mutually exclusive.}: 
	\begin{itemize}
		\myitem[(iv)] \label{itm: iv_main_2} The curvature $F_A$ decays quadratically:
		\[
		|F_A|\lesssim_{A,\Phi} \rho^{-2}\quad\text{for }\rho\geqslant R_0.
		\]
		\myitem[(v)] \label{itm: v_main_2} There exists a principal $\rm SU(2)$-bundle $P_{\infty}\to\Sigma$ and a configuration at infinity $(A_{\infty},\Phi_{\infty})\in\mathscr{A}(P_{\infty})\times\Gamma(\mathfrak{su}(2)_{P_{\infty}})$ such that the following hold:
			\begin{itemize}
				\myitem[(v.a)] \label{itm: v.a_main_2} $(A,\Phi)|_{\Sigma_R}\to (A_{\infty},\Phi_{\infty})$ uniformly as $R\to\infty$.
				\myitem[(v.b)] \label{itm: v.b_main_2} $\nabla_{A_{\infty}}\Phi_{\infty} = 0$.
				\myitem[(v.c)] \label{itm: v.c_main_2} $A_{\infty}$ is a reducible Yang--Mills connection on $(\Sigma^2,g_{\Sigma})$.
			\end{itemize}
	\end{itemize}
\end{theorem}
\begin{remark}
	The asymptotic decay rates of Theorem \ref{thm: main_2} are well known to be sharp; they are attained by the basic spherically symmetric monopole solution in $\mathbb{R}^3$ of Bogomolnyi--Prasad--Sommerfield \cites{bps1975exact,bogomol1976stability}, see \cite[pp. 104-105]{Jaffe1980}.
\end{remark}
\begin{remark}
	The assumption \ref{itm: b_main_2} implies, by the Gauss--Bonnet theorem, that $\Sigma$ must have genus zero, \emph{i.e.} be topologically a $2$-sphere. Nevertheless, it does not imply that $(X^3,g)$ is necessarily AE; \emph{e.g.} $(\Sigma^2,g_{\Sigma})$ could be any ellipsoid in $\mathbb{R}^3$.
\end{remark}
\begin{remark}
	 Since any AE $3$-manifold satisfies \ref{itm: b_main_2}, Theorem \ref{thm: main_2} is a direct generalization of the classical Jaffe--Taubes' \cite[Theorems IV.10.3 and IV.10.5]{Jaffe1980} critical point asymptotics for the $\rm SU(2)$ Yang--Mills--Higgs energy in $(\mathbb{R}^3,g_{\mathbb{R}^3})$. In fact, note that restricting to monopoles, that is, in case \ref{itm: a_main_2} holds, our results generalize the classical ones to \emph{any} AC $3$-manifold with only one end. 
	
	The key difficulty in the analysis for general critical points lies in deriving the sharp quadratic decay of the curvature \ref{itm: iv_main_2} (then part \ref{itm: v_main_2} follows as a consequence, see Theorem \ref{thm: limiting_configuration}). In the monopole case \ref{itm: a_main_2}, the conclusion is direct from the quadratic decay \ref{itm: iii_main_2} of $\nabla_A\Phi$, which in turn is valid for any critical point and is easier to deduce, using the exponential decay \ref{itm: i_main_2} combined with Bochner formulas and a mean value inequality (see the proof of Theorem \ref{thm: quadratic_decay_nablaPhi}, and the summary in \S\ref{subsec: organization}).  
	
	We treat the general case inspired by classical methods that use certain refined Kato inequalities to substantially improve the subelliptic estimates from the Bochner formulas. This type of argument goes back at least to the work of Bando--Kasue--Nakajima \cite{bando1989construction}; for related work in Yang--Mills theory see \emph{e.g.} \cites{rade1993decay,groisser1997sharp} and the very recent work \cite{cherkis2021instantons}. 
	
	The assumption \ref{itm: b_main_2} then appears naturally, after a scaling argument, to deal with higher order terms of the Ricci curvature in the Bochner formulas for $\nabla_A\Phi$ and $F_A$, which combined with the refined Kato inequalities yields a differential inequality along the end that allow us to improve the deducible order of decay of the curvature to the sharpest. We dedicate \S\ref{subsec: quadratic} on this matter; see also a brief summary in \S\ref{subsec: organization}. 
\end{remark}

In the conditions of Theorem \ref{thm: main_2}, it follows that $P_{\infty}\times_{\rm SU(2)}\mathbb{C}^2\cong \mathscr{L}\oplus \mathscr{L}^{-1}$, where $\mathscr{L}\to\Sigma$ is a complex line bundle with $\text{deg}(\mathscr{L})=k$. In fact, the parameters $m$ and $k$ determine the asymptotic configuration $(A_{\infty},\Phi_{\infty})$ \emph{up to gauge}:
\begin{equation}
	\Phi_{\infty} = \frac{1}{2}\begin{pmatrix}
		im & 0 \\
		0  & -im
	\end{pmatrix} \ \text{and}  \  
	F_{A_{\infty}} = \begin{pmatrix}
		F_{\mathscr{L}} & 0 \\
		0  & -F_{\mathscr{L}}
	\end{pmatrix},\text{  $F_{\mathscr{L}} \in -2 \pi i c_1(\mathscr{L}) \in H^2(\Sigma, -2 \pi i \mathbb{Z})$.}
\end{equation} If $(A,\Phi)$ is a monopole, \emph{i.e.} a solution to equation \eqref{eq: monopole}, of mass $m>0$ and charge $k$, then the energy formula \eqref{eq: Energy_Formula} reads
	\begin{equation}\label{eq: Energy_Formula_2}
	\mathcal{E}_X(A,\Phi)=\lim_{R\to\infty}\int_{\Sigma_R}\langle\Phi,F_A\rangle=\int_{\Sigma}\langle \Phi_{\infty}, F_{A_{\infty}}\rangle = 4\pi m k.
	\end{equation}

\begin{remark}\label{rmk: Oliveira_con_at_infty}
	The energy formula \eqref{eq: Energy_Formula}--\eqref{eq: Energy_Formula_2} is well known in the Euclidean case \cite[Proposition 3.7]{Jaffe1980}. For general AC $3$-manifolds this formula was established in \cite[Corollary 1.4.11]{oliveira2014thesis} under the \emph{a priori} stronger hypotheses that $(A,\Phi)$ is a finite mass irreducible monopole whose connection $A$ is asymptotic to a connection $A_{\infty}$ on a principal $\rm{SU}(2)$-bundle $P_{\infty}$ over $\Sigma$, satisfying $\varphi^{\ast}\left(P|_{X\setminus K}\right)\cong\pi^{\ast}P_{\infty}$, where $\pi:(1,\infty)\times\Sigma\to\Sigma$ is the projection onto the second factor, and such that for some $\varepsilon>0$ one has
	\begin{equation}\label{eq: strong_convergence_hyp}
	\varphi^{\ast}\nabla_A = \pi^{\ast}\nabla_{A_\infty} + a,\quad\text{where }\quad\lvert\nabla_{A_\infty}^j a\rvert = O(\rho^{-1-j-\varepsilon}),\quad\forall j\in\mathbb{N}_0.
	\end{equation} These assumptions turn out to imply finite energy by \cite[Corollary 1.4.4]{oliveira2014thesis}. Thus, our energy formula \eqref{eq: Energy_Formula} in Theorem \ref{thm: main_1}, together with the further equality \eqref{eq: Energy_Formula_2} implied by Theorem \ref{thm: main_2}, improves on \cite[Corollary 1.4.11]{oliveira2014thesis} in which we only assume finite energy in its derivation. In fact, note that it follows from Theorem \ref{thm: main_2} that, in the monopole case, the finite energy condition is equivalent to the convergence \ref{itm: iv_main_2}, which \emph{a priori} only implies the less restrictive version of \eqref{eq: strong_convergence_hyp} where $\varepsilon=0$. 
\end{remark}
Using Theorem \ref{thm: main_2} \ref{itm: i_main_2}, \ref{itm: ii_main_2} and \ref{itm: iii_main_2}, together with elliptic regularity results and the energy formula \eqref{eq: Energy_Formula}, we obtain the following sharp asymptotic expansion improving on \ref{itm: ii_main_2} (see Corollary \ref{cor: asymp_exp_Phi}):
\begin{corollary}[Asymptotic expansion of $|\Phi|$]\label{cor: main_asymp_exp_Phi}
	In the situation of Theorem \ref{thm: main_2}, there is $\mu\in (0,\nu)$ such that we have
	\begin{equation}
		|\Phi| = m - \frac{\|\nabla_A\Phi\|_{L^2(X)}^2}{m\mathrm{Vol}(\Sigma)}\frac{1}{\rho} + O(\rho^{-1-\mu})\quad\text{as }\rho\to\infty.
	\end{equation} In particular, if $(A,\Phi)$ is furthermore a monopole of mass $m>0$ and charge $k$ then
\begin{equation}
	|\Phi| = m - \frac{4\pi k}{\mathrm{Vol}(\Sigma)}\frac{1}{\rho} + O(\rho^{-1-\mu})\quad\text{as }\rho\to\infty.
\end{equation}
\end{corollary}
\begin{remark}\label{rmk: large_mass}
	In the work \cite{fadel2019limit}, co-authored with Oliveira, we considered the problem of the limiting behavior of sequences of $\rm SU(2)$ monopoles with fixed charge and arbitrarily large masses on an AC $3$-manifold with only one end. We proved that (a) the failure of compactness is entirely due to monopole bubbling; (b) monopole bubbling happens at finitely many isolated points; (c) these isolated points are exactly the asymptotic zero set of the Higgs fields; and (d) the number of points in the bubbling locus is controlled by the charge; see \cite[Theorem 1.1]{fadel2019limit}. We also prove an analogous result regarding more general
	sequences of critical points of the Yang--Mills--Higgs energy, see \cite[Theorem 1.2]{fadel2019limit}.
	
	In that paper we assume, in the definition of finite mass (see \cite[Definition 2]{fadel2019limit}), that the configuration connection is asymptotic to a connection at infinity as in Remark \ref{rmk: Oliveira_con_at_infty}. But this assumption was made only to have at our disposal the energy formula \eqref{eq: Energy_Formula} as proved in \cite[Corollary 1.4.11]{oliveira2014thesis} and also a rougher version of the asymptotic expansion of Corollary \ref{cor: main_asymp_exp_Phi}, proved in \cite[Proposition 2.2]{fadel2019limit} assuming the results in \cite[\S 1.4]{oliveira2014thesis}; in that version we had $o(\rho^{-1})$ instead of $O(\rho^{-1-\mu})$. Therefore, it follows from the present work that the theory developed in \cite{fadel2019limit} is actually valid for any finite energy monopole without further assumptions. 
\end{remark}

\begin{remark}
In \cite{kottke2015partial} Kottke and Singer constructed a partial compactification of the moduli space of finite energy $\rm SU(2)$ monopoles of charge $k$ on $\mathbb{R}^3$, studying a particular asymptotic region of the moduli space and the behavior of the $L^2$ metric in such region. The first step to their approach (see \cite[{\S}1.1]{kottke2015partial}) was to pass to the radial compactification $\overline{X}$ of $\mathbb{R}^3$ in order to conveniently deal with the noncompactness of the later; $\overline{X}$ is a compact $3$-manifold with boundary $\partial \overline{X}$ diffeomorphic to the $2$-sphere $\mathbb{S}^2$, and one can regard the Euclidean metric $g_{\mathbb{R}^3}$ as a \emph{scattering metric} on $\overline{X}$, see \cite{melrose1994spectral}. Then the classical results in Jaffe--Taubes \cite[Theorems IV.10.3 and IV.10.5]{Jaffe1980} guarantee that the finite energy condition on the monopoles imply decay properties equivalent to smoothness up to the boundary of $\overline{X}$. Kottke--Singer's approach has a natural generalization to any AC $3$-manifold $(X^3,g)$ with only one end, thought of as a general scattering manifold $(\overline{X}^3,g_{sc})$; indeed, by assuming the same decay properties of the Euclidean case (\emph{i.e.} smoothness of the monopoles up to the boundary $\partial\overline{X}$), some of the results in \cite{kottke2015partial} are proved in this generality (see e.g. \cite[proof of Proposition 3.3]{kottke2015partial}), in particular appealing to Kottke's previous work \cite{kottke2015dimension}. That said, we note that our main decay results for finite energy monopoles on general AC manifolds given by Theorem \ref{thm: main_2} provides the formal justification of the validity of their decay assumptions in the general case. 

Theorem \ref{thm: main_2} also has important physics consequences. It can be interpreted, in particular, as proving a generalized inverse square law for gauge group $\rm SU(2)$ which is of fundamental importance for example in establishing the quantization of magnetic charge, see \cite{goddard1977gauge}.
\end{remark}

\subsection{Organization}\label{subsec: organization}
This paper is divided into three parts. In a nutshell, the first part is concerned with the general geometric-analytic background of the paper, while the other two focus respectively on the proofs of the two main theorems stated in the previous section. I have also added Appendix \ref{app: A} containing a simple but important computation of the Ricci tensor on an AC manifold in an adapted frame along the end, which is used particularly in \S\ref{subsec: quadratic}.

Let us give a more detailed description of each section of the paper. We start in Section \ref{sec: analysis} reviewing important concepts in harmonic function theory on complete noncompact manifolds, and bringing attention to the important class of those satisfying a uniform parabolic Harnack inequality (PHI) or, equivalently, a two-sided Gaussian bound on the heat kernel. These include, for instance, complete manifolds with nonnegative Ricci curvature and, by the work of van Coevering \cite{van2009regularity}, it includes also AC manifolds with only one end. We review important Green's function bounds on nonparabolic manifolds satisfying (PHI), the validity of the $L^2$-Sobolev inequality on AC manifolds proven by van Coevering, as well as its relation with volume growth on this class of complete manifolds satisfying (PHI), and some fundamental results on solutions of the Poisson equation on the nonparabolic case. We include, in particular, regularity results on weighted H\"older spaces on the AC $3$-manifold case, which were also proven by van Coevering using the available Green's function bounds. Next, we give a general criteria to prove uniform decay of functions satisfying certain integrability properties (see Lemma \ref{lem: uniform_decay}), generalizing a classical Euclidean case result \cite[Proposition III.7.5]{Jaffe1980}, and we also revisit decay results for nonnegative solutions to differential inequalities of the form $\Delta u\leqslant \gamma+fu$, under certain assumptions on $u,\gamma,f$, on general geometric contexts. In particular, we recall an important classical Moser iteration decay result of Bando--Kasue--Nakajima \cite[Proposition 4.8 (1)]{bando1989construction} (stated as Proposition \ref{prop: BKN_decay}) on complete noncompact manifolds satisfying the $L^2$-Sobolev inequality and having at most Euclidean volume growth. We also prove other general decay results on manifolds with quadratically decaying Ricci curvature and polynomial volume growth (see Lemma \ref{lem: decay}), by using local mean value inequalities deduced from a parabolic mean value inequality first proved by Li--Tam \cite{li1991heat}. These results are used later particularly in \S\ref{sec: YMH}.

We end the first part with an important general result on the function theory of nonparabolic manifolds of dimension $n\geqslant 3$ satisfying (PHI), with Ricci curvature bounded from below and satisfying a uniform lower bound for the volume of balls which is independent of their center, \emph{e.g.} AC $n$-manifolds with only one end. The result is stated as Theorem \ref{thm: Liouville} and asserts that every harmonic function with finite Dirichlet energy on such manifolds must be constant. The proof makes crucial use of the strong Liouville property satisfied by such manifolds, combined with the previous results on the solutions of Poisson equations, the Bochner technique and a mean value inequality. We then use Theorem \ref{thm: Liouville} to generalize Groisser's original arguments in \cite[Lemma 1]{groisser1984integrality} to prove a functional analytic result, stated as Lemma \ref{lem: mass}, that is key to derive the main finite mass theorem proved in \S\ref{subsec: finite_mass}.

In the second part, Section \ref{sec: mass_charge}, we are concerned with the proof of Theorem \ref{thm: main_1}. We follow closely the original work of Groisser \cite{groisser1984integrality}, adding appropriate modifications to adapt the classical arguments in the Euclidean case $\mathbb{R}^3$ to our general AC setup. In particular, using our Lemma \ref{lem: mass}, we derive Theorem \ref{thm: finite_mass}, asserting that every finite energy configuration satisfying \eqref{eq:2nd_Order_Eq_1} has finite mass. The proof also uses the $L^2$-Sobolev inequality satisfied by AC manifolds (see Theorem \ref{thm: Sobolev_AC}), combined with standard elliptic techniques. Our proof also extends to any complete $3$-manifold with nonnegative Ricci curvature and maximal volume growth (see Remark \ref{rmk: nR_case_finite_mass}). Then we prove an alternative finite mass theorem, stated as Theorem \ref{thm: alternative_finite_mass}, that holds more generally for any complete nonparabolic manifold of dimension $n\geqslant 3$ and satisfying (PHI), giving a different characterization of the mass in terms of the Green's function and the Higgs field covariant derivative $\nabla_A\Phi$, under integrability assumptions of the later (although not requiring any integrability of $F_A$). The proof we give can be seen as a generalization of the Euclidean case proof of Jaffe--Taubes \cite[Theorem IV.10.3]{Jaffe1980}, using the results we collect in \S\ref{sec: analysis}.

We finish Section \ref{sec: mass_charge} with \S\ref{subsec: charge}, where we prove general results on the monopole number and, restricting to the structure group $\rm SU(2)$, we complete the proof of Theorem \ref{thm: main_1} by proving Theorem \ref{thm: integrality} and Corollary \ref{cor: energy_formula}.  

Finally, in the third part, Section \ref{sec: YMH}, we concentrate on the proof of Theorem \ref{thm: main_2} and Corollary \ref{cor: main_asymp_exp_Phi}. The most delicate parts, and probably the main contributions of this paper, lie in the proofs of the quadratic decay of $\nabla_A\Phi$ and $F_A$. We combine ideas from the original work in Jaffe--Taubes \cite[Chapter IV, Part ii]{Jaffe1980} with the regularity theory of the Laplacian on weighted H\"older spaces on AC $3$-manifolds, and we crucially explore Bochner formulas along the end combined with the decay results of Lemma \ref{lem: decay} and Proposition \ref{prop: BKN_decay}. In the case of the decay result for the curvature $F_A$, one further key ingredient are certain refined Kato inequalities with ``error terms'' that are based on \cite[(Proof of) Theorem 5]{smith2019removeability}.

We start in \S\ref{subsec: e-reg} deriving the Bochner--Weitzenb\"ock formulas for the rough Laplacian of $\nabla_A\Phi$ and $F_A$, and a consequent nonlinear estimate on the Laplacian of the energy density. This implies a well-known $\varepsilon$-regularity result that we use to derive general integrability and decay properties for any critical point of the Yang--Mills--Higgs energy on a noncompact Riemannian manifold of bounded geometry; see Corollary \ref{cor: e-reg_decay}. Combining this later result with our alternative general finite mass result given by Theorem \ref{thm: alternative_finite_mass}, we also derive a finite mass result for any critical point on any complete nonparabolic $3$-manifold of bounded geometry satisfying (PHI), see Corollary \ref{cor: finite_mass_for_crit_pts}.

From this point on, we restrict ourselves to AC manifolds with only one end and to configurations on principal bundles with structure group $\rm SU(2)$. In this context, \S\ref{subsec: exp_decay} is dedicated to the proof of Theorem \ref{thm: main_2} \ref{itm: i_main_2}, \emph{i.e.} the exponential decay of the $\Phi$-transverse components of $F_A$ and $\nabla_A\Phi$ for any irreducible critical point $(A,\Phi)$, see Theorem \ref{thm: exp_decay_transv}. The proof is based on the original Euclidean case proof due to Taubes, exploiting the decomposition of the adjoint bundle induced by the Higgs field along the end, together with the appropriate Bochner formulas and a comparison argument using the maximum principle. 

As a first consequence of the exponential decay result, we derive Theorem \ref{thm: main_2} \ref{itm: ii_main_2} by using the regularity theory for solutions of the Poisson equation on weighted H\"older spaces. Then \S\ref{subsec: quadratic_nablaPhi_asymp} focus mainly on the proof of the sharp quadratic decay rate of $\nabla_A\Phi$, \emph{i.e.} Theorem \ref{thm: main_2} \ref{itm: iii_main_2}. We reduce the proof to showing that $|\nabla_A\Phi|^2$ and its derivative are $O(\rho^{-3})$ and $O(\rho^{-4})$ respectively, which in turn is shown again by using a combination of Bochner inequalities along the end, together with the exponential decay of the transverse components and Lemma \ref{lem: decay}. We then derive Corollary \ref{cor: main_asymp_exp_Phi} from the previous results by using standard theory.

In \S\ref{subsec: quadratic} we prove part \ref{itm: iv_main_2} of Theorem \ref{thm: main_2} under condition \ref{itm: b_main_2}, \emph{i.e.} we prove the quadratic decay of the curvature for a general critical point when the asymptotic link $\Sigma$ has positive Gaussian curvature. We combine the Bochner formulas with refined Kato inequalities, and a scaling argument, to get an improved Bochner inequality along the end, see Proposition \ref{prop: refined_Kato_Bochner}. This together with the Moser iteration technique of Bando--Kasue--Nakajima and the regularity theory of the Laplace operator on weighted H\"older spaces implies the desired sharp quadratic decay \ref{itm: iv_main_2}, proved as Theorem \ref{thm: quadratic_decay}. In particular, this yields a new proof in the Euclidean case, perhaps more direct than Taubes' original proof. Finally, in \S\ref{subsec: lim_config} we use the decay estimates and standard techniques to prove the convergence to a limit configuration along the end (Theorem \ref{thm: limiting_configuration}), completing the proof of Theorem \ref{thm: main_2}.  

\subsection*{Acknowledgements} 
I would like to express my deepest gratitude to Gon\c{c}alo Oliveira, for all his teachings, support, encouragement and many important mathematical discussions that led me to write this paper. I am also thankful to Matheus Vieira, for various fruitful discussions on geometric analysis techniques, to Ákos Nagy for important comments and suggestions, and to the anonymous referee for their helpful comments. Special thanks goes to my partner, Ana Beatriz Fernandes, for her unconditional love and support, without which this work would not be possible.

This work was financed in part by the School of Mathematical Sciences of Peking University and by the Coordenação de Aperfeiçoamento de Pessoal de Nível Superior - Brasil (CAPES) - Finance Code 001, grant [88887.643728/2021-00] of the CAPES/COFECUB bilateral collaboration project Ma 898/18 [88887.143014/2017-00] (Universidade Estadual de Campinas/Université de Bretagne Occidentale).

\subsection{Notations and conventions}\label{sec: notations}
$\mathbb{N}_0:=\mathbb{N}\sqcup\{0\}$. We denote by $c>0$ a generic constant, which depends only on the dimension and geometry of the base Riemannian manifold $(X,g)$, and possibly on the (Lie algebra of the) structure group $G$ of a fixed principal bundle $P\to X$. Its value might change from one occurrence to the next. We write $x\lesssim y$ (or $y\gtrsim x$) for $x\leqslant cy$, while $x\sim y$ means both $x\lesssim y$ and $y\lesssim x$ (with possibly different constants). In case the hidden constant depends on further data, we indicate this by a subscript, \emph{e.g.} $x\lesssim_{A,\Phi} y$ or $x\sim_{A,\Phi} y$. We reserve $O(\cdot{})$ and $o(\cdot{})$ respectively for the standard big-O and little-o notations on the growth rate of functions under asymptotic regimes. 

The manifold $X$ is assumed to be connected, oriented and without boundary. Once fixed a Riemannian metric $g$ on $X$, the Riemannian distance function of $(X,g)$ is denoted by $d(\cdot{},\cdot{})$. We denote the open geodesic ball of center $x\in X$ and radius $r>0$ by $B(x,r):=\{y\in X: d(x,y)<r\}$. 
The Riemannian measure of a ball $B(x,r)$ is denoted by $V(x,r)$. All integrals of functions are with respect to the Riemannian measure, although we omit it (almost everywhere) in the notation. We let $\nabla$ denote the Levi--Civita connection of $(X,g)$, while the Riemann curvature covariant $4$-tensor is denoted by $\mathcal{R}_g$, the Ricci curvature tensor is denoted by $\mathcal{R}ic_g\in\Gamma(S^2T^{\ast}X)$ and the scalar curvature by $S_g$. We say that $(X,g)$ has \emph{bounded geometry} if the global injectivity radius is bounded away from zero, $\text{inj}(X):=\inf_{x\in X} \text{inj}(x)>0$, and the Riemann curvature tensor is bounded, $\|\mathcal{R}_g\|_{L^{\infty}(X)}\leqslant c<\infty$. In index notation, as usual, we use $g_{ij}$ and $g^{ij}$ to denote the metric $g$ and its inverse, and we use $\mathcal{R}_{ij}$, $\mathcal{R}_{ijk}^l$ and $\mathcal{R}_{ijkl}$ to denote the Ricci curvature tensor, the Riemann curvature $(1,3)$-tensor and the Riemann curvature $(0,4)$-tensor, respectively. Following the conventions in \cite{hamilton1982three,li2012geometric}, we lower the index to the third position in order to get $\mathcal{R}_{ijkl}$, \emph{i.e.} $\mathcal{R}_{ijkl}=g_{hk}\mathcal{R}_{ijl}^{h}$. We have $S_g = g^{ik}\mathcal{R}_{ik}$, $\mathcal{R}_{ik}=g^{jl}\mathcal{R}_{ijkl}$, and $\mathcal{R}_{ijkl}$ is anti-symmetric in the pairs $i,j$ and $k,l$ and symmetric in their interchange, and satisfies a Bianchi identity on the cyclic permutation of any three. 

The metric $g$ is assumed to be complete, and we emphasize this everywhere in the text. Using the metric $g$ and the Levi--Civita connection $\nabla$, we define in the usual way the Lebesgue spaces $L^p(X)$, the Sobolev spaces $W^{k,p}(X)$, the $C^k$-spaces $C^k(X)$ and the H\"older spaces $C^{k,\alpha}(X)$. 

Let $P\to X$ be a principal $G$-bundle, where $G$ is a compact Lie group. We denote by $\mathfrak{g}_P:=P\times_{\text{Ad}}\mathfrak{g}$ the associated adjoint bundle, which we assume to be equipped with a metric coming from a choice of Ad$_G$-invariant inner product on the Lie algebra $\mathfrak{g}$ of $G$. Most of the time $G=\rm SU(2)$ and in this case we fix the metric on $\mathfrak{su}(2)_P$ to be the one induced by the following normalization of the negative of the Killing form of $\mathfrak{su}(2)$: $(a,b)\mapsto -2\text{tr}(ab)$. We denote by $\mathscr{A}(P)$ the space of smooth connections on $P$. Given $A\in\mathscr{A}(P)$, we let $\nabla_A$ denote the various covariant derivatives induced by $A$, together with the Levi--Civita connection of $(X,g)$, on the vector bundles $V\otimes\mathfrak{g}_P$, where $V\to X$ denotes any tensor bundle. We write ${d}_A$ for the exterior covariant derivative induced by $\nabla_A$. Thus, \emph{e.g.} $[F_A,\xi] = d_A\nabla_A\xi$ for $\xi\in\Gamma(\mathfrak{g}_P)$. We let $d^{\ast}$, $d_A^{\ast}$ and $\nabla_A^{\ast}$ denote the formal $L^2$ adjoints of $d,d_A$ and $\nabla_A$ respectively. Finally, we use the geometer's convention for the various Laplace operators; $\Delta = d^{\ast}d$ denotes the Hodge--Laplacian operator on functions of $X$, and $\Delta_A=d_A d_A^{\ast} + d_A^{\ast}d_A$ is the covariant Hodge--Laplacian, induced by $A$, acting on the $\mathfrak{g}_P$-valued $k$-forms $\Omega^k(X,\mathfrak{g}_P)$. The notation $\nabla_A^2\xi$ means $\nabla_A(\nabla_A\xi)$ and \emph{not} the rough Laplacian of $\xi$, which we denote by $\nabla_A^{\ast}\nabla_A\xi$ instead. Also note that $\Delta_A = d_A^{\ast}d_A = \nabla_A^{\ast}\nabla_A$ on $\Omega^0(X,\mathfrak{g}_P)=\Gamma(\mathfrak{g}_P)$.

\section{Geometric analysis on complete noncompact manifolds}\label{sec: analysis}

This section sets up the general geometric analytic background of the paper. We derive and collect a number of important analytic tools while revisiting and putting together several fundamental known results, which in turn also motivate the geometric assumptions on the complete noncompact base manifolds over which we study Yang--Mills--Higgs theory in the next two sections. Our main result of this first part is a new Liouville type result, Theorem \ref{thm: Liouville}, stating that there is no nonconstant harmonic function with finite Dirichlet energy on a wide class of complete nonparabolic manifolds with only one end, including any AC manifold of dimension $n\geqslant 3$. 

\subsection{Green's functions, harmonic functions and volume growth}\label{subsec: greens_functions}
Let $(X,g)$ be a complete noncompact Riemannian manifold. Recall that a \emph{Green's function} on $(X,g)$ is a smooth function $G(x,y)$ on $X\times X\setminus\{(x,x):x\in X\}$ which is symmetric in the two variables $x$ and $y$ and satisfies $\Delta_y G(x,y) = \delta_x(y)$ as distributions, where $\delta_x(y)$ denotes the point mass delta function at $x$, \emph{i.e.}
\[
\int_X \Delta_y G(x,y) f(y)dy = f(x),\quad\text{for all }f\in C_c^{\infty}(X).
\] Any complete Riemannian manifold admits a Green's function \cite{malgrange1956existence}; a constructive proof of this fact was given by Li--Tam \cite{li1987symmetric} (see \emph{e.g.} \cite[Chapter 17]{li2012geometric} for more details). 

We say that $(X,g)$ is \emph{nonparabolic} if it admits a \emph{positive} Green's function $G(x,y)>0$; otherwise it is said to be \emph{parabolic}. It follows from \cite[Theorem 17.3]{li2012geometric} that a complete noncompact manifold is parabolic if and only if it admits no nonconstant upper bounded subharmonic functions.  

Now recall that $(X,g)$ is said to have the \emph{strong Liouville property} if it admits no nonconstant harmonic function which is bounded below (or above). We say that $(X,g)$ satisfies the (scale-invariant) \emph{elliptic Harnack inequality} (EHI) if there is a constant $C$ such that, for any ball $B(x,2r)\subset X$ and any nonnegative harmonic function $u$ in $B(x,2r)$, we have
\[
\sup_{B(x,r)} u \leqslant C \inf_{B(x,r)} u.
\] It is well known that the validity of (EHI) implies the strong Liouville property for $(X,g)$. Indeed, suppose $\Delta u = 0$ on $X$ and $u_{\ast}:=\inf_X u>-\infty$. Then applying (EHI) to the nonnegative harmonic function $u-u_{\ast}$ we get
\[
\sup_{B(x,r)} \{u-u_{\ast}\} \leqslant C \inf_{B(x,r)} \{u-u_{\ast}\},\quad\text{for any $x\in X$ and $r>0$}.
\] Since $C$ is uniform, as $r\to\infty$ the right-hand side of the above inequality tends to zero and we conclude that $u=u_{\ast}$ must be constant, as we wanted.

The parabolic version of the (EHI) is defined as follows. One says that $(X,g)$ satisfies the (scale-invariant) \emph{parabolic Harnack inequality} (PHI) if there is a constant $C$ such that for any $r,s\in\mathbb{R}$, $r>0$, any $x\in X$, and any nonnegative solution $u$ of the heat equation $(\partial_t + \Delta)u=0$ in $Q:=(s-4r^2,s)\times B(x,2r)$, we have
\[
\sup_{Q_{-}} u \leqslant C \inf_{Q_{+}} u,
\] where
\begin{align}
	Q_{+} &:= (s-r^2,s) \times B(x,r),\\
	Q_{-} &:= (s-3r^2,s-2r^2) \times B(x,r).
\end{align} It is immediate to see that (PHI) implies (EHI) (and therefore the strong Liouville property). After recalling the above, we now state a very important classical result, combining the works of various authors, which summarizes our fairly good understanding of complete noncompact manifolds satisfying (PHI). For a very good exposition and detailed proof of the following result, see \cite[Chapter 5]{saloff2002aspects}.
\begin{theorem}[Aronson, Fabes, Stroock, Grigor'yan, Saloff-Coste et al]\label{thm: PHI}
	Let $(X,g)$ be a complete, noncompact, Riemannian manifold. Then the following conditions are equivalent:
	\begin{itemize}
		\myitem[(i)]\label{itm: i_PHI} $(X,g)$ satisfies the parabolic Harnack inequality (PHI);
		\myitem[(ii)]\label{itm: ii_PHI} $(X,g)$ satisfies the following two properties:
		\begin{itemize}
			\myitem[(ii.a)]\label{itm: ii.a_PHI} Volume doubling: there exists a uniform constant $C_{\mathcal{D}}>0$, depending only on $(X,g)$, such that
			\begin{equation}\label{ineq: VD}
				V(x,2r) \leqslant C_{\mathcal{D}} V(x,r),  
			\end{equation} for all $x\in X$ and $r>0$.
			\myitem[(ii.b)]\label{itm: ii.b_PHI} Weak Neumann-type Poincaré inequality: there are uniform constants $C_{\mathcal{P}}>0$ and $\delta\in(0,1]$, depending only on $(X,g)$, such that
			\begin{equation}\label{ineq: wNP}
				r^{-2}\inf_{a\in\mathbb{R}}\int_{B(x,\delta r)} (f-a)^2\leqslant C_{\mathcal{P}} \int_{B(x,r)}|\nabla f|^2,
			\end{equation} for all $x\in X$, $r>0$ and $f\in C^{\infty}\left(\overline{B(x,r)}\right)$.
		\end{itemize}
		\myitem[(iii)]\label{itm: iii_PHI} There are uniform constants $c_i,C_i>0$, $i=1,2$, depending only on $(X,g)$, such that the heat kernel\footnote{$h(t,x,y)$ is a smooth function on $(0,\infty)\times X\times X$, symmetric in $x$ and $y$, and such that for each $x\in X$ one has that $h(t,x,y)=u(t,y)$ is the minimal positive fundamental solution of the heat equation $(\partial_t+\Delta)u=0$ with the initial condition $\lim_{t\downarrow 0}u(t,y)=\delta_x(y)$ in the sense of distributions.} $h(t,x,y)$ of $(X,g)$ satisfies the two-sided Gaussian bound
		\begin{equation}\label{ineq: Gaussian_bounds}
			\frac{c_1}{V(x,\sqrt{t})}e^{-C_1\frac{d(x,y)^2}{t}} \leqslant h(t,x,y) \leqslant \frac{c_2}{V(x,\sqrt{t})}e^{-C_2\frac{d(x,y)^2}{t}},
		\end{equation} for all $x,y\in X$ and $t\in (0,\infty)$.
	\end{itemize}
\end{theorem}
\begin{remark}
	As is known, the infimum in the left-hand side of \eqref{ineq: wNP} is achieved when the constant $a$ equals the mean value of $f$ over $B(x,\delta r)$. Moreover, due to the work of Jerison \cite{jerison1986poincare}, the volume doubling \eqref{ineq: VD} together with the weak Poincaré inequality \eqref{ineq: wNP} actually implies the strong Poincaré inequality, where the parameter $\delta=1$ in \eqref{ineq: wNP} (see \cite[Corollary 5.3.5]{saloff2002aspects}).
	
	The implication \ref{itm: i_PHI} $\Longrightarrow$ \ref{itm: ii_PHI} was proved by Saloff-Coste \cite{saloff1992poincare}, while \ref{itm: ii_PHI} $\Longrightarrow$ \ref{itm: i_PHI} was proved by both Grigor'yan {\cite{grigor1991heat}} and Saloff-Coste \cite{saloff1992poincare} independently. The proof of the equivalence \ref{itm: i_PHI} $\iff$ \ref{itm: iii_PHI} dates back to the works of Aronson \cite{aronson1967bounds}, who proves \ref{itm: i_PHI} $\Longrightarrow$ \ref{itm: iii_PHI}, and the work of Fabes--Stroock \cite{Fabes1989} where they prove the other implication \ref{itm: iii_PHI} $\Longrightarrow$ \ref{itm: i_PHI}.
\end{remark}
\begin{remark}\label{rmk: strong_Liouville}
	It follows from the discussion preceding Theorem \ref{thm: PHI} that any complete noncompact Riemannian manifold satisfying the equivalent conditions of Theorem \ref{thm: PHI} satisfies the strong Liouville property.
\end{remark}
Recalling the work of Varopoulos \cite{varopoulos1983potential}, a necessary condition for a complete noncompact manifold $(X,g)$ to be nonparabolic is that there exists $x\in X$ such that the volume $V(x,\sqrt{t})$ of a geodesic ball centered at $x$ and of radius $\sqrt{t}$ satisfies the growth condition
\begin{equation}\label{ineq: vol_growth}
	\int_{1}^{\infty}\frac{dt}{V(x,\sqrt{t})}<\infty.
\end{equation} It is easy to see that the condition \eqref{ineq: vol_growth} holds for \emph{some} $x\in X$ if and only if it holds for \emph{any} $x\in X$. Moreover, this necessary condition \eqref{ineq: vol_growth} is known to be sharp in the sense that if $f(t)$ is a smooth convex function such that $f'(t)>0$ and
\[
\int_1^{\infty} \frac{dt}{f(\sqrt{t})}<\infty,
\] then there is a complete nonparabolic manifold $(X,g)$ such that $V(x,\sqrt{t})=f(\sqrt{t})$ for large enough $t$; see e.g. \cite[\S 7.2]{grigor1999analytic}. 

Now, as a consequence of Theorem \ref{thm: PHI}, we have the following:
\begin{corollary}[{\cite[Corollary 5.4.13]{saloff2002aspects}}]\label{cor: nonparabolic}
	Suppose $(X,g)$ is a complete, noncompact, Riemannian manifold satisfying the equivalent conditions of Theorem \ref{thm: PHI}. Then $(X,g)$ is nonparabolic if and only if the volume growth condition \eqref{ineq: vol_growth} holds. Moreover, if this condition holds, the minimal positive Green's function $G(x,y)=\int_{0}^{\infty}h(t,x,y)dt$ satisfies
	\begin{equation}\label{ineq: Green_bounds}
		G(x,y) \sim \int_{d(x,y)^2}^{\infty}\frac{dt}{V(x,\sqrt{t})}
	\end{equation} and there is $\mu>0$ such that
	\begin{equation}\label{ineq: Green_Holder_bound}
		\frac{|G(x,y)-G(x,z)|}{d(y,z)^{\mu}}\lesssim \int_{d(x,y)^2}^{\infty}\frac{dt}{t^{\mu/2}V(x,\sqrt{t})}
	\end{equation} for all $x,y,z\in X$, $x\neq y$ and $d(y,z)\leqslant d(x,y)/2$.	
\end{corollary}
	The notion of (non)parabolicity of a manifold turns out to depend only on its behavior at infinity. Recall that an \emph{end} $E$ of the manifold $X$ is an unbounded connected component of the complement $X\setminus K$ of some compact subset $K\subset X$. Then $E$ is said to be (non)parabolic if it is the only end of some complete (non)parabolic manifold without boundary; equivalently, $E$ is (non)parabolic if it can be extended to a complete (non)parabolic manifold by attaching a compact set to its boundary.
	
	We note that the number of ends of a complete nonparabolic manifold is bounded by the dimension of the real vector space spanned by the set of positive harmonic functions (see \cite{li1992harmonic}). Hence, by the strong Liouville property, a nonparabolic manifold satisfying the equivalent conditions of Theorem \ref{thm: PHI} must have only one end.

The model examples of manifolds satisfying the equivalent conditions in Theorem \ref{thm: PHI} are the Euclidean spaces $\mathbb{R}^n$ for $n\geqslant 2$, which are nonparabolic for $n\geqslant 3$, while $\mathbb{R}^2$ is parabolic. In contrast, the hyperbolic spaces $\mathbb{H}^n$ do not satisfy (PHI); in fact, in these spaces the constant $C$ in both (EHI) and (PHI) does explode as the radius $r\to\infty$. There are various other examples of manifolds for which the equivalent conditions in Theorem \ref{thm: PHI} are known to be valid, most notably complete Riemannian manifolds of nonnegative Ricci curvature, and Lie groups equipped with an invariant metric having polynomial volume growth (see \cite[$\S$5.6]{saloff2002aspects}). Also, it is important to notice that the properties prescribed in item \ref{itm: ii_PHI} in Theorem \ref{thm: PHI}, \emph{i.e.} the volume doubling property \ref{itm: ii.a_PHI} and the validity of a Poincaré inequality \ref{itm: ii.b_PHI}, are invariant under quasi-isometries\footnote{Given a manifold $X$ and two Riemannian metrics $g$ and $\tilde{g}$ on $X$, we say that $g$ and $\tilde{g}$ are \emph{quasi-isometric} if there exists $c>0$ such that $c g_x\leqslant \tilde{g}_x\leqslant c^{-1}g_x$ as bilinear forms, at every point $x\in X$.}. This implies the nontrivial fact that the other two equivalent conditions \ref{itm: i_PHI} and \ref{itm: iii_PHI} in Theorem \ref{thm: PHI} are also invariant under quasi-isometries\footnote{Barlow and Murugan \cite{barlow2018stability} proved fairly recently that such an stability result also holds for (EHI), assuming a lower bound on the Ricci curvature; more precisely, if $(X,g)$ and $(X',g')$ are two Riemannian manifolds that are quasi-isometric to a Riemannian manifold with Ricci curvature bounded from below, then $(X,g)$ satisfies (EHI) if and only if $(X',g')$ satisfies (EHI).}, and also that any Riemannian manifold that is merely quasi-isometric to the previously mentioned examples satisfying the equivalent conditions of Theorem \ref{thm: PHI} are still examples. 

Asymptotically conical (AC) manifolds with only one end\footnote{The definition of an $n$-dimensional AC manifold with only one end, for $n\geqslant 2$, is entirely analogous to the $3$-dimensional definition given in \S\ref{subsec: background}, one just observes that in the general case the asymptotic link $\Sigma$ must be $(n-1)$-dimensional.} is another particularly interesting class of complete noncompact Riemannian manifolds that satisfy the equivalent conditions of Theorem \ref{thm: PHI}, and generalizes the class of (asymptotically) Euclidean spaces. Indeed, van Coevering \cite{van2009regularity} proved that condition \ref{itm: ii_PHI} of Theorem \ref{thm: PHI} holds in this case. He used the invariance under quasi-isometries to simplify the proof and showed the validity of the Poincaré inequality by using a discretization technique previously employed by Grigor'yan and Saloff-Coste \cite{grigor2005stability} and generalized by Minerbe \cite{minerbe2009weighted}.

\begin{theorem}[{\cite[Theorem 2.24]{van2009regularity}}]\label{thm: VD+P_AC}
	Suppose $(X^n,g)$ is a complete noncompact Riemannian $n$-manifold with only one end which is AC or merely quasi-isometric to an AC manifold. Then the equivalent conditions of Theorem \ref{thm: PHI} hold on $(X^n,g)$.
\end{theorem}
Moreover, using the same discretization technique, van Coevering also proved the validity of Euclidean-like $L^p$-Sobolev inequalities on AC manifolds:
\begin{theorem}[{\cite[Theorem 2.6 and Corollary 2.7]{van2009regularity}}]\label{thm: Sobolev_AC}
	Continue the hypotheses of Theorem \ref{thm: VD+P_AC}. Then for any real $p$ such that $1\leqslant p<n$, there is a constant $C_p>0$ such that $(X^n,g)$ satisfies the $L^p$-Sobolev inequality
	\begin{equation}\label{ineq: Euclidean_Lp-Sobolev}
		\|f\|_{L^{np/(n-p)}(X)}\leqslant C_p\|df\|_{L^p(X)},\quad\forall f\in C_c^{\infty}(X).
	\end{equation}
\end{theorem}
We shall pay particular attention to the $p=2$ case of the inequality \eqref{ineq: Euclidean_Lp-Sobolev}: a complete noncompact Riemannian $n$-manifold $(X^n,g)$ of dimension $n\geqslant 3$ is said to satisfy the \emph{$L^2$-Sobolev inequality} if there is a constant $C_{\mathcal{S}}>0$, depending only on $(X^n,g)$, such that
\begin{equation}\label{ineq: L2-Sobolev}
\|df\|_{L^2(X)}^2\geqslant C_{\mathcal{S}}\|f\|_{L^{2n/(n-2)}(X)}^2,\quad\forall f\in C_c^{\infty}(X).
\end{equation} 

On a given complete noncompact Riemannian manifold, there is a close relation between the property of supporting the $L^2$-Sobolev inequality \eqref{ineq: L2-Sobolev} and the property of satisfying the following uniform lower bound on the volume growth: 
\begin{equation}\label{ineq: lb_vol_growth}
V(x,r)\geqslant cr^n,\quad\text{for all }x\in X\text{ and }r>0.
\end{equation} Indeed, the first \eqref{ineq: L2-Sobolev} always implies the later \eqref{ineq: lb_vol_growth}. Moreover, it is a well-known fact that such properties are equivalent in the case of complete manifolds with nonnegative Ricci curvature (see \emph{e.g.} \cite[Chapter 14, Remark 2]{li2012geometric}). In fact, it turns out that these properties are equivalent more generally for any complete manifold satisfying the equivalent conditions of Theorem \ref{thm: PHI}. 
\begin{proposition}\label{prop: Sobolev_vol_growth}
	Let $(X^n,g)$ be a complete noncompact Riemannian $n$-manifold, $n\geqslant 3$. If $(X,g)$ satisfies the $L^2$-Sobolev inequality \eqref{ineq: L2-Sobolev}, then there is a constant $c>0$, depending only on $n$ and $C_{\mathcal{S}}$, such that $(X^n,g)$ satisfies the uniform lower bound \eqref{ineq: lb_vol_growth} on the volume growth. Conversely, whenever $(X,g)$ satisfies the equivalent conditions of Theorem \ref{thm: PHI}, the validity of \eqref{ineq: lb_vol_growth}, for some uniform constant $c>0$, implies that $(X,g)$ satisfies the $L^2$-Sobolev inequality \eqref{ineq: L2-Sobolev}, with constant $C_{\mathcal{S}}$ depending only on $n$, $c$ and the uniform constant $c_2>0$ appearing on the heat kernel upper bound in \eqref{ineq: Gaussian_bounds}.
\end{proposition}
\begin{proof}
	The first part is well known, see \cite[Theorem 3.1.5]{saloff2002aspects} or \cite[Lemma 20.11]{li2012geometric}.
	
	Now assume that $(X,g)$ satisfies the equivalent conditions of Theorem \ref{thm: PHI} and the volume growth lower bound \eqref{ineq: lb_vol_growth}. Then, combining the upper Gaussian bound in \eqref{ineq: Gaussian_bounds} with \eqref{ineq: lb_vol_growth} we deduce
	\[
	h(x,y,t)\leqslant c^{-1}c_2 t^{-n/2},\quad\text{for all }x,y\in X\text{ and for all }t\in(0,\infty).
	\] Thus, by \cite[Theorem 11.6]{li2012geometric} we get that $(X,g)$ satisfies the $L^2$-Sobolev inequality with constant $C_{\mathcal{S}}:=C_3(c^{-1}c_2)^{-2/n}$, where $C_3>0$ depends only on $n$.
\end{proof}
\begin{remark}\label{rmk: max_vol_growth}
	If $(X^n,g)$ is a complete Riemannian manifold with nonnegative Ricci curvature, it follows from the Bishop--Gromov volume comparison theorem (see \emph{e.g.} \cite[Theorem 1.1]{hebey2000nonlinear}) that the volume growth of the manifold is at most Euclidean: $V(x,r)\leqslant\omega_n r^n$, for all $x\in X$ and $r>0$, where $\omega_n$ is the volume of the unit ball in $\mathbb{R}^n$. Thus, in this case, if $(X,g)$ satisfies \eqref{ineq: lb_vol_growth} (or, equivalently -- by Proposition \ref{prop: Sobolev_vol_growth}, the $L^2$-Sobolev inequality) it is said to have \emph{maximal} volume growth; we shall use this terminology hereafter.
\end{remark}
Combining the previous results, we have in particular:
\begin{corollary}\label{cor: AC_nonparabolic}
	Let $(X^n,g)$ be an AC manifold with only one end and dimension $n\geqslant 3$. Then $(X^n,g)$ satisfies the volume growth lower bound \eqref{ineq: lb_vol_growth} and is a nonparabolic manifold whose minimal positive Green's function $G(x,y)$ satisfies the bounds \eqref{ineq: Green_bounds} and \eqref{ineq: Green_Holder_bound}.
\end{corollary}
\begin{proof}
	Theorem \ref{thm: Sobolev_AC} combined with Proposition \ref{prop: Sobolev_vol_growth} implies the volume growth lower bound \eqref{ineq: lb_vol_growth}. Then, by Theorem \ref{thm: VD+P_AC} we can use Corollary \ref{cor: nonparabolic} together with the lower bound $V(x,r)\gtrsim r^n$ and $n\geqslant 3$ to get the desired result.
\end{proof}

	We finish this section by recalling a general result of Ni \cite[Lemma 2.3]{ni2002poisson} and combining it with Corollary \ref{cor: nonparabolic}. We get the following general result on the existence of decaying nonnegative solutions of the Poisson equation on any nonparabolic manifold satisfying the equivalent conditions of Theorem \ref{thm: PHI}.
\begin{lemma}\label{lem: Poisson}
	Let $(X^n,g)$ be a complete nonparabolic $n$-manifold, $n\geqslant 3$. Let $f\in C^0(X)$ be a continuous nonnegative function on $X$. If $f\in L^1(X)$ then the Poisson equation
	\begin{equation}\label{eq: Poisson}
		\Delta u = f
	\end{equation} has a nonnegative solution $u\in W_{\text{loc}}^{2,n}(X)\cap C_{\text{loc}}^{1,\alpha}(X)$, $\alpha\in (0,1)$, given by
	\begin{equation}\label{eq: sol_Poisson}
	u(x):=\int_X G(x,\cdot{})f,
	\end{equation} where $G(x,y)>0$ denotes the minimal positive Green's function of $(X^n,g)$. If furthermore $(X^n,g)$ satisfies the equivalent conditions of Theorem \ref{thm: PHI} and $f\in L^{n-1}(X)\cap C^{\infty}(X)$, then $u$ defined by \eqref{eq: sol_Poisson} is the unique smooth solution to \eqref{eq: Poisson} which decays uniformly to zero at infinity. 
\end{lemma}
\begin{proof}
	The first part is proved in \cite[Lemma 2.3]{ni2002poisson}. As for the second part, first note that if $f$ is smooth then $u$ is smooth by standard elliptic regularity. Next, we show that the solution $u$ defined by \eqref{eq: sol_Poisson} decays uniformly to zero at infinity; such a solution is then unique by the maximum principle. 
	
	By the estimate \eqref{ineq: Green_bounds} of Corollary \ref{cor: nonparabolic}, note that the minimal positive Green's function $G(x,y)$ of $(X^n,g)$ satisfies $G(x,y)\to 0$ as $d(x,y)\to\infty$. Moreover, in general, one has $G(x,y)\sim d(x,y)^{2-n}$ as $d(x,y)\to 0$, and $\sup_{X\setminus B(x,r)} G(x,\cdot{})<\infty$ for all $r>0$, so that $G(x,\cdot{})\in L_{\text{loc}}^q(X)$ for any $q<\frac{n}{n-2}$ (see \cite{li1987symmetric}).
	
	Now fix a reference point $o\in X$, let $r,s>0$ and suppose $x\in X\setminus B(o,r+s)$. Note that for any $y\in B(x,r)$ we have $d(o,y)>s$. Thus, setting $q:=\frac{n-1}{n-2}$, using the above properties of $G(x,y)$ and the hypothesis on $f$, together with H\"older's inequality, we have
	\begin{align}
		0 &\leqslant u(x)\leqslant \left(\int_{B(x,r)} + \int_{X\setminus B(x,r)}\right)G(x,\cdot{})f\\
		&\leqslant\|G(x,\cdot{})\|_{L^q(B(x,r))}\|f\|_{L^{n-1}(X\setminus B(o,s))} + \sup_{y\in X\setminus B(x,r)}G(x,y)\|f\|_{L^1(X)}.\label{ineq: Green_sol_bound}
	\end{align} Now let $\varepsilon>0$. Then, since $G(x,y)\to 0$ as $d(x,y)\to\infty$, and $f\in L^1(X)$, we can choose $r\gg 1$ such that the last term in the right-hand side of inequality \eqref{ineq: Green_sol_bound} is smaller than $\varepsilon/2$. Then, since $G(x,\cdot{})\in L_{\text{loc}}^q(X)$ and $f\in L^{n-1}(X)$, we can choose $s\gg 1$ such that the first term in the right-hand side of inequality \eqref{ineq: Green_sol_bound} is smaller than $\varepsilon/2$. Therefore, we get $R:=r+s>0$ such that if $x\in X\setminus B(o,R)$ then we have $u(x)<\varepsilon$. This shows that $u$ decays uniformly to zero, as we wanted.
\end{proof}

\subsection{Laplacian operator on AC $3$-manifolds}\label{subsec: Laplace_AC}
Let us now restrict attention to the main class of manifolds that we shall be concerned with in the next sections of this paper. Let $(X^3,g)$ be an AC $3$-manifold with rate $\nu>0$, connected link $(\Sigma^2,g_{\Sigma})$, and radius function $\rho$. Note that $(X^3,g)$ has bounded geometry and $|\mathcal{R}_g|=O(\rho^{-2})$ as $\rho\to\infty$. Moreover, by Corollary \ref{cor: AC_nonparabolic}, $(X^3,g)$ satisfies the volume growth lower bound $V(x,r)\gtrsim r^3$ for all $x\in X$ and $r>0$, and is a nonparabolic manifold whose minimal positive Green's function $G(x,y)$ satisfies
\begin{equation}\label{eq: Green's_behavior}
	0<G(x,y)\lesssim d(x,y)^{-1},
\end{equation} for all $x,y\in X$, $x\neq y$, and there is $\mu>0$ such that
\begin{equation}\label{eq: Holder_mu}
	\frac{|G(x,y)-G(x,z)|}{d(y,z)^{\mu}}\lesssim d(x,y)^{-1-\mu},
\end{equation} for all $x,y,z\in X$, $x\neq y$ and $d(y,z)\leqslant d(x,y)/2$.

In order to deal with the Laplace operator on the noncompact AC $3$-manifold $(X^3,g)$, motivated by the above Green's function behavior, it is convenient to introduce the following \emph{weighted} H\"older spaces. 

For $\beta\in\mathbb{R}$ and $k\in\mathbb{N}_0$, we define $C_{\beta}^k(X)$ to be the real vector space of continuous functions $f:X\to\mathbb{R}$ with $k$ continuous derivatives such that
\[
\|f\|_{C_{\beta}^k}:=\sum\limits_{j=0}^k \sup_X |\rho^{j-\beta}\nabla^j f|<\infty.
\] Then $(C_{\beta}^k(X),\|\cdot{}\|_{C_{\beta}^k})$ is a Banach space.
Now let $\alpha,\gamma\in\mathbb{R}$ and $T$ be a tensor field on $X$. We define
\[
[T]_{\alpha,\gamma}:=\sup\limits_{\substack{x\neq y \\ d(x,y)<\text{inj}(x)}}	\left(\min(\rho(x),\rho(y))^{-\gamma}\frac{|T(x)-T(y)|}{d(x,y)^{\alpha}}\right),
\] where $|T(x)-T(y)|$ is understood by identifying the fibers of the tensor bundle over $x$ and $y$ via parallel translation along the unique geodesic joining $x$ and $y$.

Now for $\alpha\in (0,1)$ we define the \emph{weighted H\"older space} $C_{\beta}^{k,\alpha}(X)$ to be the real vector space of all $f\in C_{\beta}^k(X)$ for which
\[
\|f\|_{C_{\beta}^{k,\alpha}}:=\|f\|_{C_{\beta}^{k}} + [\nabla^k f]_{\alpha,\beta-k-\alpha}<\infty.
\] Then $(C_{\beta}^{k,\alpha}(X),\|\cdot{}\|_{C_{\beta}^{k,\alpha}})$ is a Banach space.

The following embedding theorem can be found in \cite[Theorems 4.17 and 4.18]{Mar02}.
\begin{theorem}\label{thm: embedding}
	Let $\alpha,\beta\in (0,1)$, $\gamma,\delta\in\mathbb{R}$, and $k,l\in\mathbb{N}_0$. Suppose that $\gamma\leqslant\delta$. If $k+\alpha\geqslant l+\beta$, then there are continuous embeddings
	\[
	C_{\gamma}^{k+1}\hookrightarrow C_{\gamma}^{k,\alpha}\hookrightarrow C_{\delta}^{l,\beta}\hookrightarrow C_{\delta}^{l},
	\] and if $k\geqslant l$ then
	\[
	C_{\gamma}^{k}\hookrightarrow C_{\delta}^{l}.
	\] The embedding $C_{\gamma}^{k,\alpha}\hookrightarrow C_{\delta}^{k}$ is compact whenever $\gamma<\delta$. 
\end{theorem}
We can now state fundamental results on the Laplace operator on AC $3$-manifolds deduced by van Coevering using the good Green's function estimates guaranteed by Corollary \ref{cor: AC_nonparabolic}.
\begin{theorem}[{\cite[Lemma 2.29 and Theorem 2.30]{van2009regularity}}]\label{thm: VC_regularity}
	Suppose that $(X^3,g)$ is an AC $3$-manifold with only one end and rate $\nu>\mu$, where $\mu>0$ is such that \eqref{eq: Holder_mu} holds. Let $\alpha\in (0,1)$ and $k\in\mathbb{N}_0$. Then the following hold:
	\begin{itemize}
		\myitem[(i)]\label{itm: VC_i} If $u\in C_{\beta}^2(X)$ and $v\in C^2_{\gamma}(X)$ where $\beta,\gamma\in\mathbb{R}$ satisfy $\beta+\gamma<-1$, then
		\[
		\int_X u \Delta v = \int_X v \Delta u.
		\]
		\myitem[(ii)]\label{itm: VC_ii} If $\rho$ is a radius function on $(X,g)$, then $\Delta(\rho^{-1})\in C_{-3-\nu}^{k,\alpha}(X)$. And if $(\Sigma,g_{\Sigma})$ is the link in the conical end, then
		\[
		\int_X\Delta(\rho^{-1}) = \mathrm{Vol}(\Sigma,g_{\Sigma}).
		\]
		\myitem[(iii)]\label{itm: VC_iii} Suppose $\beta\in (-3,-2)$. There exists $C>0$ such that for each $f\in C_{\beta}^{k,\alpha}(X)$ there is a unique $u\in C_{\beta+2}^{k+2,\alpha}(X)$ with $\Delta u =f$ which satisfies $\|u\|_{C_{\beta+2}^{k+2,\alpha}}\leqslant C\|f\|_{C_{\beta}^{k,\alpha}}$; $u$ is given by
		\begin{equation}\label{eq: Green_sol}
			u(x):=\int_X G(x,y)f(y),
		\end{equation} where $G(x,y)$ is the minimal positive Green's function of $(X^3,g)$.
		\myitem[(iv)]\label{itm: VC_iv} Suppose $\beta\in [-3-\mu,-3)$. There exist $C_1,C_2>0$ such that for each $f\in C_{\beta}^{k,\alpha}(X)$ there is a unique $u\in C_{-1}^{k+2,\alpha}(X)$ with $\Delta u =f$. Furthermore
		\[
		u = A\rho^{-1} + v,
		\] where $A:=\mathrm{Vol}(\Sigma,g_{\Sigma})^{-1}\int_X f$ satisfies $|A|\leqslant C_1\|f\|_{C_{\beta}^0}$, and $v\in C_{\beta+2}^{k+2,\alpha}(X)$ with $\|v\|_{C_{\beta+2}^{k+2,\alpha}}\leqslant C_2\|f\|_{C_{\beta}^{k,\alpha}}$. 
	\end{itemize} 
\end{theorem}
We also state here a particular instance, for the Laplace operator, of a general elliptic regularity result proved in \cite[Theorem 4.21]{Mar02}:
\begin{theorem}\label{thm: elliptic_reg}
	Let $(X^3,g)$ be an AC $3$-manifold with only one end. Suppose that $f\in L_{\text{loc}}^1(X)$ and that $u\in L_{\text{loc}}^1(X)$ is a weak solution to the equation $\Delta u = f$. If $u\in C_{\beta+2}^0(X)$ and $f\in C_{\beta}^{k,\alpha}(X)$ for some $\beta\in\mathbb{R}$, $k\in\mathbb{N}_0$ and $\alpha\in (0,1)$, then $u\in C_{\beta+2}^{k+2,\alpha}(X)$ with $\Delta u = f$ strongly and $\|u\|_{C_{\beta+2}^{k+2,\alpha}}\lesssim \|\Delta u\|_{C_{\beta}^{k,\alpha}} + \|u\|_{C_{\beta+2}^0}$.
\end{theorem}

\subsection{Decay and mean value inequalities}\label{subsec: decay_mean_value}
Back to generality, in this section we collect some useful results for the analysis of the asymptotics of functions satisfying certain integrability properties and differential inequalities. 

We start with a general criteria for uniform decay.
\begin{lemma}\label{lem: uniform_decay}
	Let $(X^n,g)$ be a complete noncompact Riemannian $n$-manifold with Ricci curvature bounded from below, and assume there is a uniform lower bound for the volume of balls which is independent of their center:
	\begin{equation}\label{ineq: non_collapsing}
	\inf_{x\in X} V(x,1)>0. 
	\end{equation} If $h\in W^{1,p}(X)$ for some $p>n$, then $h\in C^{0,\alpha}(X)$ with $\alpha:=1-n/p$, and $h$ decays uniformly to zero at infinity, \emph{i.e.} for all $x\in X$,
	\begin{equation}\label{eq: uniform_decay}
		\lim_{R\to\infty}\sup_{X\setminus B(x,R)} |h| = 0.
	\end{equation}
\end{lemma}
\begin{remark}\label{rmk: r_non_collapsing}
	Assume that $(X,g)$ is a complete Riemannian manifold with Ricci curvature bounded from below. Then it follows from Bishop--Gromov volume comparison theorem \cite[Theorem 1.1]{hebey2000nonlinear} that the assumption \eqref{ineq: non_collapsing} is actually equivalent to assuming that for any $r>0$ there is $v_r>0$ such that $\inf_{x\in X} V(x,r)\geqslant v_r$. Moreover, it is well known that \eqref{ineq: non_collapsing} is equivalent to the validity of all the standard Sobolev embeddings on $(X,g)$, see \cite[Theorems 3.2 and 3.6]{hebey2000nonlinear}. 
	
	A sufficient condition to ensure \eqref{ineq: non_collapsing} is the assumption of positivity of the injectivity radius $\mathrm{inj}(X)>0$; this follows \emph{e.g.} by \cite[Proposition 3.6 and Theorem 3.3]{hebey2000nonlinear}. Conversely, it follows from a classical result of Cheeger--Gromov--Taylor (see \cite[Theorems 4.3 and 4.7]{cheeger1982finite}) that if $(X,g)$ satisfies the stronger assumption of bounded Riemann curvature, \emph{i.e.} $\|\mathcal{R}_g\|_{L^{\infty}(X)}\leqslant c<\infty$, then the validity of \eqref{ineq: non_collapsing} implies that $(X,g)$ has positive injectivity radius; in this case we say that $(X,g)$ has \emph{bounded geometry}.
\end{remark}
\begin{proof}[Proof of Lemma \ref{lem: uniform_decay}]
	The first part follows from the H\"older--Sobolev embedding $W^{1,p}(X)\hookrightarrow C^{0,\alpha}(X)$ \cite[Theorems 3.6]{hebey2000nonlinear}. The proof of the uniform decay then proceeds in the same way as in the Euclidean case proof \cite[Proposition III.7.5]{Jaffe1980}, but for completeness we include it here. Suppose \eqref{eq: uniform_decay} does not hold. Then, there is $\varepsilon>0$ and a sequence of points $(x_j)_{j=1}^{\infty}\subset X$ such that $d(x,x_j)\to\infty$ and $|h|(x_j)>\varepsilon$. Without loss of generality, we may assume that $d(x_i,x_j)>2$ for all $i\neq j$. Now, by the H\"older--Sobolev embedding, there is $c>0$ such that setting $C:=c\|h\|_{W^{1,p}}$, we have
	\[
	|h(y)-h(x_j)|\leqslant Cd(y,x_j)^{\alpha},\quad\forall j\in\mathbb{N},\quad\forall y\in X.
	\] Now take $r:=\displaystyle\min\{1,\varepsilon/(2C)\}^{1/\alpha}$. Then, by the above, for all $y\in B(x_j,r)$ we have
	\[
	|h(y)|\geqslant |h(x_j)| - |h(y)-h(x_j)|\geqslant \varepsilon - \frac{\varepsilon}{2} = \frac{\varepsilon}{2}.
	\] From Remark \ref{rmk: r_non_collapsing}, there is $v_r>0$ such that $V(x,r)\geqslant v_r$ for every $x\in X$. Observing also that $B(x_i,r)\cap B(x_j,r)=\emptyset$ for all $i\neq j$, we conclude that
	\[
	\int_X |h|^p \geqslant \sum\limits_{j=1}^{\infty}\int_{B(x_j,r)}|h|^p \geqslant \sum\limits_{j=1}^{\infty} \left(\frac{\varepsilon}{2}\right)^p v_r = \infty,
	\] contradicting the fact that $h\in L^p(X)$.
\end{proof}

Now we recall important mean value inequality results. We start with some results from \cite[\S 4]{bando1989construction} that will play a key role in \S\ref{subsec: quadratic}. These give certain \emph{a priori} estimates for a nonnegative function $u$ on $(X,g)$ satisfying the differential inequality
\begin{equation}\label{ineq: subsolution_Schrodinger}
	\Delta u\leqslant fu,
\end{equation} for some nonnegative function $f$ on $X$, under certain integrability conditions on $u$ and $f$. Their proof rely on the so-called Moser iteration technique and therefore require the following assumptions on the underlying geometry. Suppose that $(X^n,g)$ is a complete Riemannian manifold of dimension $n\geqslant 3$, satisfying the following two properties:
\begin{itemize}
	\item The $L^2$-Sobolev inequality \eqref{ineq: L2-Sobolev}, and
	\item For some reference point $o\in X$, there is a constant $C_o>0$, depending on the point $o\in X$ and $(X,g)$, such that
	\begin{equation}\label{ineq: at_most_Euclidean_vol_growth}
	V(o,r)\leqslant C_o r^n,\quad\forall r>0.
	\end{equation}	 
\end{itemize} Examples of Riemannian manifolds satisfying these properties include the AC manifolds with only one end, as well as complete manifolds with nonnegative Ricci curvature and maximal volume growth (see Remark \ref{rmk: max_vol_growth}).
\begin{lemma}[{\cite[Lemma 4.6]{bando1989construction}}]\label{lem: moser_iteration}
	Let $(X^n,g)$ be a complete noncompact $n$-manifold, $n\geqslant 3$, satisfying \eqref{ineq: L2-Sobolev} and \eqref{ineq: at_most_Euclidean_vol_growth}. Suppose that $f$ is a nonnegative function in $L^q(X)$ for some $q>n/2$ and such that there is a constant $A\geqslant 0$ with
	\[
	\int_{X\setminus B(o,r)} f^q \leqslant A r^{-(2q-n)}.
	\] Let $p>1$ be a fixed constant. Then there exists a constant $C>0$, depending only on $p$, $A$, $n$, $C_{\mathcal{S}}$ and $C_o$, with the following significance. If $u$ is a nonnegative function satisfying \eqref{ineq: subsolution_Schrodinger} on $X\setminus B(o,\frac{r}{2})$ and such that $u\in L^p(X)$, then  
	\[
	\sup_{X\setminus B(o,2r)} u^p \leqslant C r^{-3}\int_{X\setminus B(o,r)} u^p.
	\]
\end{lemma}
Combining Lemma \ref{lem: moser_iteration} with other similar Moser iteration type results, Bando, Kasue and Nakajima proved the following important decay result.
\begin{proposition}[{\cite[Proposition 4.8]{bando1989construction}}]\label{prop: BKN_decay}
	Continue the hypothesis of Lemma \ref{lem: moser_iteration} for $(X,g)$, $f$ and $u$. Suppose furthermore that $f\in L^{n/2}(X)$. Then $u=O(d(o,\cdot{})^{-\alpha})$ as $d(o,\cdot{})\to\infty$, for any $\alpha<n-2$.
\end{proposition}

Next we cite more general local mean value inequalities, from which we shall also derive some important decay results. The following is a consequence of a parabolic mean value inequality first proved by Li--Tam \cite[Theorem 1.1]{li1991heat} via heat kernel estimates; see also \cite[Theorem 14.7]{li2012geometric}.
\begin{proposition}\label{prop: MV_ineqs}
	Let $(X,g)$ be any complete Riemannian $n$-manifold. Suppose that $x\in X$ and $r>0$ are such that the Ricci curvature of $X$ on the ball $B(x,4r)$ satisfies $\mathcal{R}ic_g\geqslant -(n-1)\kappa g$, for some constant $\kappa\geqslant 0$. Then the following hold:
	\begin{itemize}
		\myitem[(i)]\label{itm: MV_i} (cf. \cite[Corollary 14.8]{li2012geometric}) Let $p>0$ and $\lambda\geqslant 0$ be fixed constants. Then there exists a constant $C>0$, depending only on $p$, $n$, $\lambda r^2$ and $r\sqrt{\kappa}$, such that for any nonnegative function $u$ defined on $B(x,2r)$ satisfying the differential inequality
		\begin{equation}
			\Delta u \leqslant \lambda u
		\end{equation} we have
		\begin{equation}
			\sup_{B(x,\frac{r}{2})} u^p \leqslant C \frac{1}{V(x,r)}\int_{B(x,r)} u^p.
		\end{equation}
	
		\myitem[(ii)]\label{itm: MV_ii} Let $\lambda>0$ be a fixed constant. Then there exists a constant $C>0$, depending only on $n$, $\lambda r^2$ and $r\sqrt{\kappa}$, such that for any nonnegative function $u$ defined on $B(x,2r)$ satisfying the differential inequality
		\begin{equation}\label{ineq: MV_ii}
			\Delta u \leqslant \gamma + \lambda u,
		\end{equation} for some constant $\gamma\geqslant 0$, we have
		\begin{equation}
			\sup_{B(x,\frac{r}{2})} u \leqslant C \left(\frac{\gamma}{\lambda}+\frac{1}{V(x,r)}\int_{B(x,r)} u\right).
		\end{equation}
	
		\myitem[(iii)]\label{itm: MV_iii} There is a constant $C>0$, depending only on $n$ and $r\sqrt{\kappa}$, such that for any nonnegative function $u$ defined on $B(x,2r)$ satisfying the differential inequality
		\begin{equation}\label{ineq: MV_iii}
			\Delta u \leqslant \gamma,
		\end{equation} for some constant $\gamma\geqslant 0$, we have
		\begin{equation}
			\sup_{B(x,\frac{r}{2})} u \leqslant \frac{\gamma}{8}r^2 + C\frac{1}{V(x,r)}\int_{B(x,r)} u.
		\end{equation}
	\end{itemize} 
\end{proposition}
\begin{proof}
	The proof of parts \ref{itm: MV_ii} and \ref{itm: MV_iii} follows the same idea of the proof of part \ref{itm: MV_i} given in \cite[Corollary 14.8]{li2012geometric}, one just needs to define appropriate functions to apply the parabolic mean value inequality in each case. In the case \ref{itm: MV_ii}, one considers the function $g(t,x):=e^{-\lambda t}u(x) + \frac{\gamma}{\lambda}e^{-\lambda t}$, which by the hypothesis \eqref{ineq: MV_ii} satisfies $(\partial_t + \Delta) g \leqslant 0$ on $[0,\infty)\times B(x,2r)$. In the case \ref{itm: MV_iii}, one considers the function $g(t,x):= u(x) - \gamma t$, which by the hypothesis \eqref{ineq: MV_iii} satisfies $(\partial_t + \Delta) g \leqslant 0$ on $[0,\infty)\times B(x,2r)$. Then one applies the parabolic mean value inequality \cite[Theorem 14.7]{li2012geometric} to $g(t,x)$ in each case, setting the parameters in that result as follows: $q=1$, $\delta=1/4$, $\eta = 1/4$, $T=r^2/4$ and $\tau =r^2/8$. The resulting estimates imply the desired results.
\end{proof}
The following is an important general decay result that will be used often in Section \ref{sec: YMH}.
\begin{lemma}\label{lem: decay}
	Let $(X^n,g)$ be a complete noncompact Riemannian $n$-manifold. Fix a reference point $o\in X$, let $\rho_o(x):= (1 + d(o,x)^2)^{1/2}$ and suppose that $\mathcal{R}ic_g\geqslant -(n-1)K\rho_o(x)^{-2} g$ at every $x\in X$, for some constant $K\geqslant 0$. Assume also that there are constants $c>0$ and $l\geqslant 1$ such that $V(x,r)\geqslant cr^l$ for all $x\in X$ and $r>0$. Then the following hold:
	\begin{itemize}
	\myitem[(i)]\label{itm: decay_i} Let $p>0$ and $\Lambda\geqslant 0$ be fixed constants. Then there exists a constant $C>0$ depending only on $p$, $n$, $\Lambda$, $K$, $c$ and $l$, with the following significance. Suppose that $u\in L^p(X)$ is a nonnegative function satisfying
	\[
	\Delta u\leqslant \Lambda\rho_o^{-2}u\quad\text{on }X\setminus B(o,R),
	\] for some $R>1$. Then for all $x\in X\setminus B(o,R)$ we have
	\[
	u^p(x)\leqslant C (\rho_o(x)-R)^{-l}\|u\|_{L^p(X)}^p.
	\] In particular, $u$ decays uniformly to zero at infinity and $u=O(\rho_o^{-\frac{l}{p}})$.
	\myitem[(ii)]\label{itm: decay_ii} Let $q\geqslant 2$, $s\geqslant 0$ and $\Lambda>0$ be fixed constants. Then there exists a constant $C>0$ depending only on $n$, $\Lambda$, $K$, $c$, $l$, $q$ and $s$, with the following significance. Let $u$ be a nonnegative function satisfying
	\[
	\Delta u\leqslant \Gamma\rho_o^{-q} +\Lambda\rho_o^{-2}u\quad\text{on }X\setminus B(o,R),
	\] for some constants $\Gamma\geqslant 0$ and $R>1$. If $\rho_o^s u\in L^1(X)$ then for all $x\in X\setminus B(o,R)$ we have
	\[
	u(x)\leqslant C\left(\frac{\Gamma}{\Lambda}(\rho_o(x)-R)^{-(q-2)} + (\rho_o(x)-R)^{-(l+s)}\|\rho_o^s u\|_{L^1(X)}\right).
	\] In particular, if $q>2$ then $u$ decays uniformly to zero at infinity and $u=O(\rho_o^{-\alpha})$, where $\alpha:=\min\{q-2,l+s\}$. 
	\end{itemize}
\end{lemma}
\begin{proof}
	\ref{itm: decay_i}: Let $x\in X\setminus B(o,R)$ and define $r:=\frac{1}{8}(\rho_o(x)-R)$. Then, using that $R>1$, one readily verifies that for any $y\in B(x,4r)$ one has $\rho_o(y)>4r>r$. In particular, using the hypotheses, it follows that $\mathcal{R}ic\geqslant -(n-1)Kr^{-2}g$ on $B(x,4r)$, and $\Delta u\leqslant \Lambda r^{-2} u$ on $B(x,2r)$. Thus, applying Proposition \ref{prop: MV_ineqs} \ref{itm: MV_i} with $\kappa:=Kr^{-2}$ and $\lambda:=\Lambda r^{-2}$, we get that there is a constant $\tilde{C}>0$ depending only on $p$, $n$, $\Lambda$ and $\sqrt{K}$ such that
	\[
	u^p(x)\leqslant\sup_{B(x,\frac{r}{2})} u^p\leqslant \tilde{C}\frac{1}{V(x,r)}\int_{B(x,r)} u^p.
	\] Now, using the hypotheses $u\in L^p(X)$ and $V(x,r)\geqslant cr^l$, together with the definition of $r$, it follows that
	\[
	u^p(x)\leqslant \tilde{C}c^{-1}8^l(\rho_o(x)-R)^{-l}\|u\|_{L^p(X)}^p.
	\] Since $x\in X\setminus B(o,R)$ is arbitrary, the result follows by taking $C:=\tilde{C}c^{-1}8^l$. In particular, note that for $x\in X\setminus B(o,2R)$ one has $\rho_o(x)-R\geqslant\frac{1}{2}\rho_o(x)$ and therefore $u^p(x)\leqslant C' \rho_o(x)^{-l}\|u\|_{L^p(X)}^p$.\\
	
	\ref{itm: decay_ii}: Let $x$ and $r$ be as in the above proof of part \ref{itm: decay_i}. Using the hypotheses of \ref{itm: decay_ii} and applying Proposition \ref{prop: MV_ineqs} \ref{itm: MV_ii} with $\kappa:=Kr^{-2}$, $\lambda:=\Lambda r^{-2}$ and $\gamma:=\Gamma r^{-q}$, now we get a constant $\tilde{C}>0$ depending only on $n$, $\Lambda$ and $\sqrt{K}$ such that
	\[
	u(x)\leqslant\sup_{B(x,\frac{r}{2})} u\leqslant\tilde{C}\left(\frac{\Gamma}{\Lambda}r^{-(q-2)} + \frac{1}{V(x,r)}\int_{B(x,r)}u\right).  
	\] Now use that $\rho_o(y)>r$ for all $y\in B(x,r)$ to get that
	\[
	u(x)\leqslant \tilde{C}\left(\frac{\Gamma}{\Lambda}r^{-(q-2)} + \frac{1}{r^s V(x,r)}\int_{B(x,r)}\rho_o^s u\right).
	\] Finally, using the hypotheses $\rho_o^s u\in L^1(X)$ and $V(x,r)\geqslant cr^l$, and the definition of $r$, yields
	\[
	u(x)\leqslant\tilde{C}\left(8^{q-2}\frac{\Gamma}{\Lambda}(\rho_o(x)-R)^{-(q-2)} + c^{-1}8^{l+s}(\rho_o(x)-R)^{-(l+s)}\|\rho_o^s u\|_{L^1(X)}\right).
	\] Hence, the result follows by taking $C:=\tilde{C}\max\{8^{q-2},c^{-1}8^{l+s}\}$.
\end{proof}
\begin{remark}
	Let $(X^n,g)$ be an AC manifold with only one end $X\setminus K$ and dimension $n\geqslant 3$. Given any radius function $\rho$ on $X$, for any fixed reference point $o\in K$ we have $\rho(x)\sim\rho_o(x):=(1+d(o,x)^2)^{1/2}$, and by the quadratic decay of the Ricci curvature (see Appendix \ref{app: A}) there is $K\geqslant 0$ such that $\mathcal{R}ic_g\geqslant -(n-1)K\rho_o(x)^{-2}g$ at every $x\in X$. Moreover, by Corollary \ref{cor: AC_nonparabolic}, $(X^n,g)$ satisfies the volume growth lower bound \eqref{ineq: lb_vol_growth}. It follows that the assumptions of Lemma \ref{lem: decay} hold on any such $(X^n,g)$, with $l=n$. In Section \ref{sec: YMH} we will apply Lemma \ref{lem: decay} on AC $3$-manifolds; in particular, we will always have $l=n=3$ in that case. Also note that since $\rho\sim\rho_o$ we can use $\rho$ instead of $\rho_o$ in the decay estimates.
\end{remark}
 We finish this section with a general decay result for the gradient of harmonic functions with finite Dirichlet energy, \emph{i.e.} finite $L^2$-norm of the gradient, on complete noncompact manifolds with lower bounded Ricci curvature and non-collapsing volume.
\begin{lemma}\label{lem: bounded_gradient}
	Let $(X,g)$ be a complete noncompact Riemannian manifold with Ricci curvature bounded from below, namely $\mathcal{R}ic_g\geqslant -(n-1)\kappa g$ on $X$ for some constant $\kappa\geqslant 0$, and assume there is a uniform lower bound for the volume of balls which is independent of their center:
	\begin{equation}\label{ineq: non_collapsing_assumption}
		v_1:=\inf_{x\in X} V(x,1)>0.
	\end{equation}
	Then every harmonic function $h$ on $(X,g)$ with finite Dirichlet energy, \emph{i.e.} with $dh\in L^2(X)$, has bounded gradient $dh\in L^{\infty}(X)$ and in fact $|dh|$ decays uniformly to zero at infinity; in particular, $dh\in L^p(X)$ for all $p\in [2,\infty]$.
\end{lemma}
\begin{proof}
	We start noting that, since $\Delta h = 0$, by the standard Bochner formula for $1$-forms we have
	\[
	\langle\nabla^{\ast}\nabla(dh),dh\rangle = \langle\Delta(dh),dh\rangle - \mathcal{R}ic_g(dh^{\sharp},dh^{\sharp}) = - \mathcal{R}ic_g(dh^{\sharp},dh^{\sharp}). 
	\] 
	In particular, since $\mathcal{R}ic_{g}\geqslant -(n-1)\kappa g$,
	\[
	\frac{1}{2}\Delta|d h|^2 = \langle\nabla^{\ast}\nabla(dh),dh\rangle - |\nabla^2 h|^2 \leqslant (n-1)\kappa|d h|^2.
	\] Therefore, it follows from Proposition \ref{prop: MV_ineqs} \ref{itm: MV_i} that there is a constant $C>0$ depending only on $\kappa$ and $n$ such that for all $x\in X$ we have
	\begin{equation}\label{ineq: interm_step_mean_value}
	|dh|^2(x)\leqslant\sup_{B(x,\frac{1}{2})} |dh|^2 \leqslant C\frac{1}{V(x,1)}\int_{B(x,1)} |dh|^2.
	\end{equation} Inequality \eqref{ineq: interm_step_mean_value} combined with $dh\in L^2(X)$ and the assumption \eqref{ineq: non_collapsing_assumption} already gives
	\[
	\|dh\|_{L^{\infty}(X)}^2\leqslant Cv_1^{-1}\|dh\|_{L^2(X)}^2<\infty,
	\] proving that indeed $dh\in L^{\infty}(X)$. 
	
	Now, given any $\varepsilon>0$ it follows from the assumption $|dh|^2\in L^1(X)$ that there is a large enough ball $B(o,R)\subseteq X$ such that
	\begin{equation}\label{eq: smaller_than_epsilon}
	\int_{X\setminus B(o,R)}|dh|^2< C^{-1}v_1\varepsilon.
	\end{equation} Therefore, given any $x\in X\setminus B(o,R+2)$, noting that $B(x,1)\subset X\setminus B(o,R)$ and combining with \eqref{ineq: interm_step_mean_value} and \eqref{eq: smaller_than_epsilon} we get
	\[
	|dh|^2(x)\leqslant Cv_1^{-1}\int_{B(x,1)}|dh|^2\leqslant Cv_1^{-1}\int_{X\setminus B(o,R)}|dh|^2<\varepsilon.
	\] This shows that, in fact, $|dh|$ decays uniformly to zero at infinity.
	
	The last part of the statement follows from the fact that $L^2(X)\cap L^{\infty}(X)\subset L^p(X)$ for all $p\in [2,\infty]$.
\end{proof}

\subsection{A Liouville type result and a functional analytic consequence}\label{subsec: Liouville}

We state and prove the main result of this section, together with one important consequence that will be used in Section \ref{sec: mass_charge}. 

\begin{theorem}\label{thm: Liouville}
	Let $(X^n,g)$ be a complete nonparabolic Riemannian $n$-manifold, $n\geqslant 3$, satisfying the equivalent conditions of Theorem \ref{thm: PHI}. Suppose furthermore that $(X,g)$ has Ricci curvature bounded from below and satisfies a uniform lower bound for the volume of balls which is independent of their center \eqref{ineq: non_collapsing}. Then every harmonic function on $(X,g)$ with finite Dirichlet energy must be constant.
\end{theorem}
\begin{proof}
	Let $h\in C^{\infty}(X)$ satisfy $\Delta h=0$ and $|dh|^2\in L^1(X)$. By Lemma \ref{lem: bounded_gradient} we have that $|dh|^2\in L^1(X)\cap L^{n-1}(X)$. Thus, we can apply Lemma \ref{lem: Poisson} to derive the existence of a smooth nonnegative solution $u$ of the Poisson equation $\Delta u = 2|dh|^2$, which furthermore decays at infinity. On the other hand, since $h$ is harmonic, we have
	\[
	\Delta(h^2) = 2h\Delta h - 2|dh|^2 = - 2|dh|^2. 
	\] Thus, we conclude that $w:=u+h^2$ is a nonnegative harmonic function on $X$, so that by the strong Liouville property (see Remark \ref{rmk: strong_Liouville}) it follows that $w$ is a constant, say $w\equiv a^2$. Since $u\geqslant 0$, it follows that $h^2\leqslant a^2$ on $X$. 
	Hence, $h$ is a bounded harmonic function on $X$, so another application of the strong Liouville property implies that $h$ is constant, as we wanted.

\end{proof} 
\begin{remark}\label{rmk: ex_Liouville_AC}
	Any AC manifold with only one end and dimension $n\geqslant 3$ satisfies the hypotheses of Theorem \ref{thm: Liouville}; this follows from Theorem \ref{thm: VD+P_AC} together with Corollary \ref{cor: AC_nonparabolic} and the fact that the Ricci curvature decays quadratically in this case (see Appendix \ref{app: A}), so in particular it is bounded from below.
\end{remark}
\begin{remark}\label{rmk: examples_4D}
	In dimensions $n\geqslant 4$, the assumption of the uniform lower bound for the volume of balls independent of their center \eqref{ineq: non_collapsing} in Theorem \ref{thm: Liouville} is not redundant; more precisely, there are examples of manifolds $(X^n,g)$ satisfying all the assumptions of Theorem \ref{thm: Liouville} except \eqref{ineq: non_collapsing}. Indeed, Croke and Karcher \cite[Example 2]{croke1988volumes} constructed, for each positive real number $\alpha\in (2/3,1)$, a complete metric $g=g_{\alpha}$ in $\mathbb{R}^4$ (extending to higher dimensions) of positive Ricci curvature (therefore satisfying the equivalent conditions of Theorem \ref{thm: Liouville} and having Ricci curvature bounded from below), whose sectional curvatures decay to zero at infinity, and such that the volume of balls behave like $V(x,r)\sim r^3 t^{\alpha - 1}$, where $t$ denotes the Euclidean distance from $x$ to the origin $0\in\mathbb{R}^4$; in particular, $(\mathbb{R}^4,g_{\alpha})$ is nonparabolic (by Corollary \ref{cor: nonparabolic}) and the volume of balls goes uniformly to zero as the center goes off to infinity. 
\end{remark}

\begin{remark}\label{rmk: nR_case_crucial}
	Although the positive Ricci curvature examples of Remark \ref{rmk: examples_4D} fail in satisfying \eqref{ineq: non_collapsing}, the conclusion of Theorem \ref{thm: Liouville} still hold for those manifolds. In fact, more generally, for any complete noncompact manifold $(X,g)$ with nonnegative Ricci curvature it suffices to assume that $(X,g)$ has infinite volume to get the conclusion of Theorem \ref{thm: Liouville}, and the proof is very simple.  
	
	Indeed, let $h\in C^{\infty}(X)$ be such that $\Delta h =0$ and $dh\in L^2(X)$. Then the following hold strongly outside the zero locus of $|dh|$ and weakly everywhere:
	\begin{align}
		|dh|\Delta |dh| &\leqslant \langle\nabla^{\ast}\nabla(dh),dh\rangle\\
		&=\langle\Delta(dh),dh\rangle-Ric_g(dh^{\sharp},dh^{\sharp})\\
		&= -\mathcal{R}ic_g(dh^{\sharp},dh^{\sharp})\leqslant 0,  
	\end{align} where in the first line we used Kato's inequality, in the second line we used the standard Bochner--Weitzenb\"ock formula for $1$-forms, and in the third line we used the harmonicity of $h$ together with the assumption of nonnegative Ricci curvature. Therefore, $|dh|$ is subharmonic. This together with the integrability assumption $|dh|\in L^2(X)$ implies that $|dh|$ must be constant, by a classical result due to Yau \cite{yau1976some} stating that for any $p>1$ a complete manifold does not admit any nonconstant nonnegative $L^p$ subharmonic function (see \cite[Lemma 7.1]{li2012geometric}). But $X$ has infinite volume, so the integrability actually forces $|dh|\equiv 0$, as we wanted.
	
	Note that nonparabolicity implies infinite volume by \eqref{ineq: vol_growth}, but it is generally a much stronger assumption; there are of course many examples of parabolic manifolds with nonnegative Ricci curvature and infinite volume, \emph{e.g.} asymptotically cylindrical Calabi--Yau and $\rm G_2$-manifolds, and these are also of great interest in the study of higher dimensional instantons and general Yang--Mills connections (see \emph{e.g.} \cite{sa2020current}).
	
\end{remark}

\begin{remark}\label{rmk: one_end}
	The number of nonparabolic ends of a complete noncompact manifold is bounded above by the dimension of the space spanned by the bounded harmonic functions with finite Dirichlet energy \cite{li1992harmonic}. Therefore, any complete noncompact manifold satisfying the conclusion of Theorem \ref{thm: Liouville} must have only one nonparabolic end. 
\end{remark}

We now finish this section with an important consequence of Theorem \ref{thm: PHI} that will be used in Section \ref{sec: mass_charge}. In what follows, let $(X^n,g)$ be a complete noncompact Riemannian $n$-manifold, $n\geqslant 3$, satisfying the equivalent conditions of Theorem \ref{thm: PHI}, with Ricci curvature bounded from below and satisfying the $L^2$-Sobolev inequality \eqref{ineq: L2-Sobolev}. By Proposition \ref{prop: Sobolev_vol_growth} and Corollary \ref{cor: nonparabolic}, it follows that $(X,g)$ is nonparabolic and satisfies $V(x,r)\gtrsim r^n$, for all $x\in X$ and $r>0$. In particular, $(X,g)$ satisfies the hypotheses of Theorem \ref{thm: Liouville}.

Let $\mathcal{V}$ denote the space of real valued functions $f\in W_{\text{loc}}^{1,2}(X)$ for which $df\in L^2(X)$, and let $\mathcal{H}$ be the Hilbert space obtained from the completion of the space $C_c^{\infty}(X)$ of smooth compactly supported functions on $X$ with respect to the norm $\|f\|_{\mathcal{H}}:=\|df\|_{L^2(X)}$. By the $L^2$-Sobolev inequality it follows that there is a continuous embedding $\mathcal{H}\hookrightarrow L^{\frac{2n}{n-2}}(X)$; in particular, by H\"older's inequality, $\mathcal{H}\hookrightarrow\mathcal{V}$. In fact, we can prove the following key result, which is inspired by the $\mathbb{R}^3$ versions \cite[Lemma 4.12]{Taubes1982} and \cite[Lemma 1]{groisser1984integrality}.
\begin{lemma}\label{lem: mass}
	Let $(X^n,g)$ be a complete noncompact Riemannian $n$-manifold, $n\geqslant 3$, satisfying the equivalent conditions of Theorem \ref{thm: PHI}, with Ricci curvature bounded from below and satisfying the $L^2$-Sobolev inequality \eqref{ineq: L2-Sobolev}. Let $\mathcal{V}$ and $\mathcal{H}$ be defined as above. 
	
	For each $f\in\mathcal{V}$ there is a unique real number $m(f)\in\mathbb{R}$ such that $f-m(f)\in\mathcal{H}$. Thus, there is a canonical isomorphism $\mathcal{V}\cong \mathcal{H}\oplus\mathbb{R}$. Moreover, $m(f)$ is also characterized as the unique real number such that $f-m(f)\in L^{\frac{2n}{n-2}}(X)$.
\end{lemma}
\begin{proof}
	We follow the original Euclidean case proof in \cite[Lemma 1]{groisser1984integrality}, noting that the key ingredients to generalize that proof to the present case are the $L^2$-Sobolev inequality \eqref{ineq: L2-Sobolev} and Theorem \ref{thm: Liouville}. 
	
	Fix $f\in\mathcal{V}$ and define the functional $Q:\mathcal{H}\to\mathbb{R}$ by $Q(h):=\|d(f-h)\|_{L^2(X)}^2$. Then, using Young's inequality, one may readily verify that $Q$ is strictly convex and satisfies the coercive estimate $Q(h)\geqslant\frac{1}{2}\|h\|_{\mathcal{H}}^2 - \|f\|_{\mathcal{H}}^2$. Moreover, $Q$ is differentiable, with derivative at $h\in\mathcal{H}$ given by the linear functional $\delta Q(h): \mathcal{H}\ni u\mapsto 2\langle du,d(f-h)\rangle_{L^2(X)}$, whose operator norm is bounded by H\"older's inequality. In particular, $Q$ is also lower semicontinuous. From these properties, and since $\mathcal{H}$ is a reflexive Banach space, it follows that $Q$ achieves a unique minimum, say at $h\in\mathcal{H}$, and if we let $u:=f-h$ then by the vanishing of the derivative $\delta Q(h)=0$ we have that $\Delta u = 0$ holds weakly. Now, as we noted before, the $L^2$-Sobolev inequality implies $\mathcal{H}\hookrightarrow L^{\frac{2n}{n-2}}(X)$, whence $\mathcal{H}\hookrightarrow\mathcal{V}$. Thus $u=f-h\in\mathcal{V}\hookrightarrow W_{\text{loc}}^{1,2}(X)$ and, since $\Delta u=0$ holds weakly, elliptic regularity yields that $u$ is in fact a smooth harmonic function with finite Dirichlet energy.
	
	We now invoke Theorem \ref{thm: Liouville} to conclude that $u$ is constant, so $f-h=u=\text{const}=:m(f)$. Note that $f-m(f)=h\in\mathcal{H}$, and the uniqueness of $m(f)$ follows from the uniqueness of the minimum point $h$. This proves the main part of the result. The stated isomorphism is then clear. As for the last statement, we already proved that $f-m(f)\in L^{\frac{2n}{n-2}}(X)$, since $\mathcal{H}\hookrightarrow L^{\frac{2n}{n-2}}(X)$. Moreover, the uniqueness of $m(f)$ satisfying this last property follows from the fact that $(X,g)$ has infinite volume. 
\end{proof}
\begin{remark}\label{rmk: nR_case_lem_mass}
	The hypotheses of Lemma \ref{lem: mass} (and therefore its conclusion) hold, \emph{e.g.}, for any AC $n$-manifold with only one end, and any complete $n$-manifold with nonnegative Ricci curvature and maximal volume growth (see Theorem \ref{thm: Sobolev_AC} and Remarks \ref{rmk: ex_Liouville_AC} and \ref{rmk: max_vol_growth}).
\end{remark}

\section{Finite mass, charge and energy formula}\label{sec: mass_charge}
This section is dedicated to the proof of Theorem \ref{thm: main_1} and is mostly based on \cite{groisser1984integrality}, with the necessary adaptations. 

\subsection{Finite mass from finite energy}\label{subsec: finite_mass}
Let us start by introducing the concept of finite mass configurations.
\begin{definition}[Finite mass]
	Let $(X,g)$ be a complete, noncompact, Riemannian manifold with \emph{only one end}, \emph{i.e.} the complement of any ball on $X$ has only one unbounded connected component. Suppose that $(A,\Phi)$ is a smooth configuration on a principal $G$-bundle $P\to X$. Then $(A,\Phi)$ is said to have \emph{finite mass} $m\in\mathbb{R}^{+}$ if for any $x\in X$ one has
	\begin{equation}\label{eq: finite_mass}
		\lim_{R\to\infty}\sup_{X\setminus B(x,R)} |m-|\Phi|| = 0.
	\end{equation}
\end{definition}
\begin{remark}\label{rmk: finite_mass}
	Suppose that $(A,\Phi)$ has finite mass $m\geqslant 0$ and satisfies \eqref{eq:2nd_Order_Eq_1}. Then it follows from \eqref{eq: finite_mass}, \eqref{eq: subharmonic} and the maximum principle (see \cite[Proposition IV.3.3]{Jaffe1980}), that either $|\Phi|\equiv m$ or $|\Phi|<m$ on $X$. In particular, $|\Phi|^2$ is a bounded subharmonic function on $(X,g)$. Moreover, $\Phi=0$ if and only if $m=0$. In case $m>0$, by the uniform convergence \eqref{eq: finite_mass} we have $|\Phi|\geqslant\frac{m}{2}>0$ outside a sufficiently large ball on $X$.
\end{remark}
It follows from Remark \ref{rmk: finite_mass} that if we want to consider finite mass \emph{irreducible} solutions $(A,\Phi)$ of equations \eqref{eq:2nd_Order_Eq_1} and \eqref{eq:2nd_Order_Eq_2} on $(X,g)$, meaning solutions such that $\nabla_A\Phi\neq 0$, then $(X,g)$ must be a \emph{nonparabolic} manifold; otherwise it would not admit the nonconstant bounded subharmonic function $|\Phi|^2$ (see \S\ref{sec: analysis}).
 
The main class of complete nonparabolic $3$-manifolds that we shall focus our study is that of AC oriented $3$-manifolds with only one end. Nevertheless, we shall also add remarks and occasionally prove analogous results for other general classes of nonparabolic geometries (even in higher dimensions), \emph{e.g.} manifolds with nonnegative Ricci curavature and maximal volume growth. For the convenience of the reader, in Section \ref{sec: analysis} we collected all the necessary analysis background concerning these general geometries that we shall need henceforth. Note that the flat Euclidean $3$-dimensional space $\mathbb{R}^3$ is the only Ricci-flat AC $3$-manifold (up to isometry), and so is a distinguished member of the intersection of the classes of general AC $3$-manifolds with only one end and nonparabolic $3$-manifolds with nonnegative Ricci curvature. Thus, by working out the theory on AC manifolds and keeping an eye on other general geometries, such as nonparabolic $3$-manifolds with nonnegative Ricci curvature, we hope to clarify the main difficulties in generalizing the Euclidean case results, especially the role played by the geometry via the Ricci curvature and volume growth. 

\begin{remark}\label{rmk: uniqueness_R3}
	If $X^3$ is a noncompact $3$-manifold (connected and without boundary) admitting a complete metric of nonnegative Ricci curvature and maximal (Euclidean) volume growth, then $X^3$ is diffeomorphic to $\mathbb{R}^3$; this follows from combining the works of Zhu \cite{zhu1993finiteness} and Liu \cite{liu20133}. In particular, if $(X^3,g)$ is AC and has nonnegative Ricci curvature, then $X^3\cong\mathbb{R}^3$.
	
	There are interesting examples of general nonparabolic metrics in $\mathbb{R}^3$ with nonnegative Ricci curvature for which moduli spaces of monopoles have been studied. In fact, Oliveira \cite[Chapter 2]{oliveira2014thesis} studied the moduli spaces of spherically symmetric $\rm SU(2)$ monopoles on $\mathbb{R}^3$ endowed with an arbitrary spherically symmetric nonparabolic metric $g$. He proved in particular that for each nonnegative real number $m\in [0,\infty)$ there corresponds a unique gauge equivalence class of spherically symmetric monopoles with mass $m$ (see \cite[Theorem 2.2.1]{oliveira2014thesis}). Now, on $\mathbb{R}^3\setminus\{0\}\cong (0,\infty)_r\times\mathbb{S}^2$ a spherically symmetric metric $g$ can be written as $g=dr^2 + h^2(r)g_{\mathbb{S}^2}$, with $h(r)=r+h_3r^3+\ldots$ around $r=0$ to ensure smoothness and bounded curvature at $r=0$. The nonparabolicity of $g$ is then equivalent to $\int_1^{\infty}\frac{dr}{h^2(r)}<\infty$ and in this case $G(x,y)=\int_{d(x,y)}^{\infty} \frac{dr}{2h^2(r)}$ is a positive Green's function. Moreover, for such metrics we have $\mathcal{R}ic_g\geqslant -2h''(r)/h(r)$ at distance $r$ from the origin $0\in\mathbb{R}^3$. Thus we see that there are many non-trivial choices of functions $h$ for which the corresponding metric $g$ is nonparabolic and has nonnegative Ricci curvature (the trivial case $h(r)=r$ corresponding to the Euclidean metric).
\end{remark} 

From now on, unless otherwise stated, we shall assume that $(X^3,g)$ is an AC oriented $3$-manifold with only one end, and $\rho$ will denote a radius function on $X$. Now let $P$ be a principal $G$-bundle over $X$, where $G$ is a compact Lie group. We fix a metric on the associated adjoint bundle $\mathfrak{g}_P$ coming from a choice of Ad$_G$-invariant inner product on the Lie algebra $\mathfrak{g}$ of $G$. From now on we shall consider the following \emph{configuration space}:
\[
\mathscr{C}(P):=\{(A,\Phi)\in \mathscr{A}(P)\times \Gamma(\mathfrak{g}_P): \mathcal{E}_X(A,\Phi)<\infty\}.
\] Note that given $(A,\Phi)\in\mathscr{C}(P)$ then by Kato's inequality we have $d|\Phi|\in L^2(X)$ and thus $|\Phi|\in\mathcal{V}$. Therefore, by Lemma \ref{lem: mass}, $m:=m(|\Phi|)$ is the unique real number such that $m-|\Phi|\in\mathcal{H}$, and also the unique number such that $m-|\Phi|\in L^6(X)$; moreover, one has $\|m-|\Phi|\|_{L^6(X)}\lesssim \|\nabla_A\Phi\|_{L^2(X)}$.

The main result of this section is the following (cf. \cite[Lemma 3]{groisser1984integrality}).
\begin{theorem}\label{thm: finite_mass}
	Let $(A,\Phi)\in\mathscr{C}(P)$ and suppose \eqref{eq:2nd_Order_Eq_1} holds, \emph{i.e.} suppose $\Delta_A\Phi = 0$. Let $m:=m(|\Phi|)$ be given by Lemma \ref{lem: mass}. Then $m\geqslant 0$, $\nabla_A^2\Phi\in L^2(X)$, $\nabla_A\Phi\in L^p(X)$ for $2\leqslant p\leqslant 6$ and $(A,\Phi)$ has finite mass $m$, \emph{i.e.}, \eqref{eq: finite_mass} holds.
\end{theorem}
In the following we prove a series of lemmas dedicated to the proof Theorem \ref{thm: finite_mass}. 
\begin{lemma}\label{lem: Phi_bounded}
	Let $(A,\Phi)\in\mathscr{C}(P)$ and suppose $\Delta_A\Phi=0$. Then $\Phi\in L^{\infty}(X)$ with $0\leqslant |\Phi|\leqslant m$ on $X$; in particular, $m\geqslant 0$.
\end{lemma}
\begin{proof}
	Using $\Delta_A\Phi=0$ and Kato's inequality we compute that
	\[
	\Delta|\Phi| = |\Phi|^{-1}\left(|d|\Phi||^2-|\nabla_A\Phi|^2\right)\leqslant 0
	\] holds strongly outside the zero locus of $\Phi$ and weakly everywhere. By Lemma \ref{lem: mass}, $h:= m-|\Phi|\in\mathcal{H}\hookrightarrow L^6(X)$, and by the above $\Delta h \geqslant 0$ holds weakly. By the weak maximum principle (see \cite[Proposition VI.3.2]{Jaffe1980}), we get $h\geqslant 0$ everywhere, so $0\leqslant|\Phi|\leqslant m$ and $\Phi\in L^{\infty}(X)$. 
\end{proof}

\begin{notation}\label{notation: cut-off}
	Henceforth, given any $R\in (1,\infty)$, we shall denote by $\chi_R\in C_c^{\infty}(X)$ a cut-off function satisfying $0\leqslant\chi_R\leqslant 1$,
	\begin{equation}\label{eq: cut-off}
	\chi_R(x)=\begin{cases}
		1,\quad\text{if }\rho(x)\leqslant R/2\\
		0, \quad\text{if }\rho(x)\geqslant R
	\end{cases}
	\end{equation} and the bounds
	\begin{equation}\label{ineq: C2_bound_cut-off}
	\|d\chi_R\|_{L^{\infty}(X)}\lesssim R^{-1}\quad\text{and}\quad\|\Delta\chi_R\|_{L^{\infty}(X)}\lesssim R^{-2}.
	\end{equation} We note in particular that $\|d\chi_R\|_{L^3(X)}$ is bounded independently of $R$:
	\begin{equation}\label{ineq: L3_bound_cut-off}
		\|d\chi_R\|_{L^3(X)}^3=\|d\chi_R\|_{L^3(\overline{B}_{R}\setminus B_{R/2})}^3 \lesssim R^{-3}\int_{R/2}^{R} \rho^2 d\rho \lesssim 1.
	\end{equation}
\end{notation}

\begin{remark}\label{rmk: nR_case_cut-off}
	The existence of such cut-off functions $\chi_R$ as in Notation \ref{notation: cut-off} is clear when $(X^3,g)$ is an AC manifold and $\rho$ is any radius function on $X$. Now let $(X^3,g)$ be any complete $3$-manifold with nonnegative Ricci curvature. We claim that such functions $\chi_R$ also do exist in this case if we let $\rho(x):=d(o,x)$ in \eqref{eq: cut-off}, for some fixed reference point $o\in X$. Indeed, it follows from the proof of \cite[Theorem 2.2]{guneysu2016sequences} that there exists $\chi_R\in C_c^{\infty}(X)$ with $0\leqslant\chi_R\leqslant 1$ satisfying both \eqref{eq: cut-off} and \eqref{ineq: C2_bound_cut-off}. As for the property \eqref{ineq: L3_bound_cut-off}, note that in this nonnegative Ricci case by Bishop--Gromov we have $V(o,R)\lesssim R^3$, so that using the properties \eqref{eq: cut-off} and \eqref{ineq: C2_bound_cut-off} we have
	\[
	\|d\chi_R\|_{L^3(X)}^3\leqslant\int_{B(o,R)} |d\chi_R|^3 \lesssim R^{-3}V(o,R) \lesssim 1,
	\] as we wanted.
\end{remark}

\begin{lemma}\label{lem: W12_integrability}
	Let $(A,\Phi)\in\mathscr{C}(P)$ and suppose $\Delta_A\Phi = 0$. Then $\nabla_A^2\Phi\in L^2(X)$.
\end{lemma}
\begin{proof}
	Given $R\gg 1$, let $\chi=\chi_R\in C_c^{\infty}(X)$ be a cut-off function as in Notation \ref{notation: cut-off}. Integration by parts yields
	\[
	\|\nabla_A^2(\chi\Phi)\|_{L^2}^2 = \langle\nabla_A(\chi\Phi),\nabla_A^{\ast}\nabla_A(\nabla_A(\chi\Phi))\rangle_{L^2}.
	\]
	From the standard Bochner--Weitzenb\"ock formula (see \cite[Theorem 3.2]{bourguignon1981stability}), we have
	\[
	\nabla_A^{\ast}\nabla_A(\nabla_A(\chi\Phi)) = \Delta_A(\nabla_A(\chi\Phi)) - \mathcal{R}ic_g\#\nabla_A(\chi\Phi) - \ast[\ast F_A,\nabla_A(\chi\Phi)]. 
	\] Therefore
	\begin{equation}\label{eq: intermediate}
		\|\nabla_A^2(\chi\Phi)\|_{L^2}^2 \lesssim \|\nabla_A(\chi\Phi)\|_{L^2}^2 + \int_X |F_A||\nabla_A(\chi\Phi)|^2 + \langle\nabla_A(\chi\Phi),\Delta_A\nabla_A(\chi\Phi)\rangle_{L^2}.
	\end{equation} We now deal with the last two terms in the right-hand side of \eqref{eq: intermediate}. First, integrating by parts gives
	\[
		\langle\nabla_A(\chi\Phi),\Delta_A\nabla_A(\chi\Phi)\rangle_{L^2} = \|d_A^{\ast}\nabla_A(\chi\Phi)\|_{L^2}^2 + \|d_A\nabla_A(\chi\Phi)\|_{L^2}^2.
	\] Now, since $\Phi\in L^{\infty}$ by Lemma \ref{lem: Phi_bounded}, 
	\[
	\|d_A\nabla_A(\chi\Phi)\|_{L^2}^2 = \|[F_A,\chi\Phi]\|_{L^2}^2 \lesssim \|\Phi\|_{L^{\infty}}^2\|F_A\|_{L^2}^2.
	\] On the other hand, a quick computation yields $d_A^{\ast}\nabla_A(\chi\Phi)=\Delta\chi\otimes\Phi - 2\langle d\chi,\nabla_A\Phi\rangle$, so that using Young's inequality to deal with the mixed term we get
	\begin{align}
		\|d_A^{\ast}\nabla_A(\chi\Phi)\|_{L^2}^2 &\lesssim \int_X |\Delta\chi|^2|\Phi|^2 + |d\chi|^2|\nabla_A\Phi|^2\\
		&\lesssim R^{-1}\|\Phi\|_{L^{\infty}}^2 + R^{-2}\|\nabla_A\Phi\|_{L^2}^2.
	\end{align} Thus
	\begin{equation}\label{eq: last term}
		\langle\nabla_A(\chi\Phi),\Delta_A\nabla_A(\chi\Phi)\rangle_{L^2}\lesssim  R^{-1}\|\Phi\|_{L^{\infty}}^2 + R^{-2}\|\nabla_A\Phi\|_{L^2}^2 + \|\Phi\|_{L^{\infty}}^2\|F_A\|_{L^2}^2.
	\end{equation}
	As for the second term in the right-hand side of \eqref{eq: intermediate}, start noting that H\"older's inequality gives
	\[
	\int_X |F_A||\nabla_A(\chi\Phi)|^2\leqslant \|F_A\|_{L^2}\|\nabla_A(\chi\Phi)\|_{L^4}^2.
	\]
	Now, using the Gagliardo--Nirenberg interpolation inequality
	\[
	\|f\|_{L^4}\lesssim\|\nabla f\|_{L^2}^{3/4}\|f\|_{L^2}^{1/4}
	\] together with Kato's inequality and Young's inequality we have
	\begin{align}
		\|\nabla_A(\chi\Phi)\|_{L^4}^2 &\lesssim \|\nabla_A^2(\chi\Phi)\|_{L^2}^{3/2}\|\nabla_A(\chi\Phi)\|_{L^2}^{1/2} \label{eq: interpolation}\\
		&\lesssim \varepsilon\|\nabla_A^2(\chi\Phi)\|_{L^2}^2 + \varepsilon^{-3}\|\nabla_A(\chi\Phi)\|_{L^2}^2,
	\end{align} for any $\varepsilon>0$. Taking $\varepsilon=c\|F_A\|_{L^2}^{-1}$ with $c>0$ sufficiently small, the first term can be absorbed in the left-hand side of \eqref{eq: intermediate}, and by combining the above with \eqref{eq: last term} we end up with
	\[
	\|\nabla_A^2(\chi\Phi)\|_{L^2}^2 \lesssim \|\nabla_A(\chi\Phi)\|_{L^2}^2 + \|F_A\|_{L^2}^4\|\nabla_A(\chi\Phi)\|_{L^2}^2 +\|\Phi\|_{L^{\infty}}^2\|F_A\|_{L^2}^2 + R^{-1}\|\Phi\|_{L^{\infty}}^2 + R^{-2}\|\nabla_A\Phi\|_{L^2}^2.
	\] Hence, letting $R\to\infty$ yields
	\[
	\|\nabla_A^2\Phi\|_{L^2}^2 \lesssim \|\nabla_A\Phi\|_{L^2}^2 + \|F_A\|_{L^2}^4\|\nabla_A\Phi\|_{L^2}^2 +\|\Phi\|_{L^{\infty}}^2\|F_A\|_{L^2}^2<\infty.
	\] 
\end{proof}
\begin{remark}
	I was unable to find a copy of the Ph.D. thesis of Groisser, to which he refers the reader in \cite[Proof of Lemma 3]{groisser1984integrality} for the proof in $\mathbb{R}^3$ of the result proved in Lemma \ref{lem: W12_integrability}, although Jaffe--Taubes proves a slightly more general result in this case in \cite[Theorem V.8.1]{Jaffe1980}, and the same methods we employed in the above proof can also be used to prove the analogue of Jaffe--Taubes' result on general AC $3$-manifolds with only one end.
\end{remark}

\begin{proof}[Proof of Theorem \ref{thm: finite_mass}]
	By hypothesis $\nabla_A\Phi\in L^2(X)$, and by Lemma \ref{lem: W12_integrability} we further have $\nabla_A^2\Phi\in L^2(X)$. For $2\leqslant p\leqslant 6$, it follows from the Sobolev embedding $W^{1,2}(X)\hookrightarrow L^p(X)$ and Kato's inequality, that $\nabla_A\Phi\in L^p(X)$. Now let $h:=m-|\Phi|\in\mathcal{H}$ as in the proof of Lemma \ref{lem: Phi_bounded}. By the $L^2$-Sobolev inequality of Theorem \ref{thm: Sobolev_AC}, it follows that $h\in L^6(X)$, with $\|h\|_{L^6(X)}\lesssim \|\nabla_A\Phi\|_{L^2(X)}<\infty$. Now, again, by the Sobolev embedding $W^{1,2}(X)\hookrightarrow L^6(X)$ and Kato's inequality, we also have
	\[
	\|dh\|_{L^{6}(X)}\leqslant \|\nabla_A\Phi\|_{L^6(X)} \lesssim (\|\nabla_A\Phi\|_{L^2(X)}+\|\nabla_A^2\Phi\|_{L^2(X)})<\infty.
	\] Therefore $h\in W^{1,6}(X)$ and it follows from Lemma \ref{lem: uniform_decay} that $h$ decays uniformly to zero at infinity, \emph{i.e.} \eqref{eq: finite_mass} holds. The proof is complete.
\end{proof}
\begin{remark}\label{rmk: nR_case_finite_mass}
	It follows from the above proofs of Lemmas \ref{lem: Phi_bounded} and \ref{lem: W12_integrability} and of Theorem \ref{thm: finite_mass}, together with Remarks \ref{rmk: nR_case_lem_mass} and \ref{rmk: nR_case_cut-off}, that these results also hold for any complete Riemannian $3$-manifold $(X^3,g)$ with nonnegative Ricci curvature and maximal volume growth. 
\end{remark}

The above proof of Theorem \ref{thm: finite_mass} is adapted from the Euclidean case proof of \cite[Lemma 3]{groisser1984integrality}. To finish this section, we now give an alternative proof of the finite mass condition \eqref{eq: finite_mass} that does not rely on the previous results of this section but yields a different characterization of the mass (\emph{i.e.} other than that of Lemma \ref{lem: mass}, whose proof relied on the $L^2$-Sobolev inequality), although the conclusion now holds under more general assumptions on the geometry of $(X,g)$ and on the configuration $(A,\Phi)$, including also higher dimensions. The proof ot this next result is based on Lemma \ref{lem: Poisson} and the strong Liouville property, and can be seen as a generalization of the Euclidean case proof of \cite[Theorem IV.10.3]{Jaffe1980}; see also \cite[\S 4]{fadel2020asymptotic}.

\begin{theorem}\label{thm: alternative_finite_mass}
	Let $(X^n,g)$ be a complete nonparabolic Riemannian $n$-manifold, $n\geqslant 3$, satisfying the equivalent conditions of Theorem \ref{thm: PHI}. Let $P\to X$ be a principal $G$-bundle over $X$, where $G$ is a compact Lie group, and suppose that $(A,\Phi)\in\mathscr{A}(P)\times\Gamma(\mathfrak{g}_P)$ satisfies \eqref{eq:2nd_Order_Eq_1} and $\nabla_A\Phi\in L^2(X)\cap L^{2(n-1)}(X)$. Finally, let $G(x,y)>0$ be the minimal positive Green's function of $(X^n,g)$. Then
	\begin{equation}\label{eq: def_w}
	w(x):=2\int_X G(x,\cdot{})|\nabla_A\Phi|^2,\quad\forall x\in X,
	\end{equation} defines the unique nonnegative smooth solution of the Poisson equation $\Delta w = 2|\nabla_A\Phi|^2$ which decays uniformly to zero at infinity. Moreover, there is a constant $m\in [0,\infty)$ such that
	\begin{equation}\label{eq: alt_characterization_m}
		w = m^2 - |\Phi|^2.  
	\end{equation} In particular, $(A,\Phi)$ has finite mass $m$, \emph{i.e.} \eqref{eq: finite_mass} holds.
\end{theorem}	
\begin{proof}
The first part is a direct consequence of Lemma \ref{lem: Poisson}.	Now, on the one hand, by the assumption $\Delta_A\Phi=0$, we have that $\Delta |\Phi|^2 = -2|\nabla_A\Phi|^2$ (see \eqref{eq: subharmonic}), which together with the fact that $w$ is a smooth nonnegative solution to $\Delta w = 2|\nabla_A\Phi|^2$ readily implies that $h:=w+|\Phi|^2$ is a smooth nonnegative harmonic function on $(X^n,g)$. On the other hand, from the assumptions, $(X^3,g)$ satisfies the strong Liouville property (see Remark \ref{rmk: strong_Liouville}). Therefore, it follows that $h$ is constant; say $h\equiv m^2$, where $m\geqslant 0$. In particular, since $w$ decays uniformly to zero at infinity, we get that $(A,\Phi)$ has finite mass $m$, completing the proof. 
\end{proof}
\begin{remark}\label{rmk: alt_finite_mass}
	Let us restrict attention back to dimension $n=3$. By Theorem \ref{thm: VD+P_AC} and Corollary \ref{cor: AC_nonparabolic}, AC $3$-manifolds with only one end satisfy the equivalent conditions of Theorem \ref{thm: PHI} and are nonparabolic. The same is true for complete $3$-manifolds with nonnegative Ricci curvature and maximal volume growth. Moreover, on such manifolds, by Lemma \ref{lem: W12_integrability} together with the Sobolev embedding, any finite energy configuration $(A,\Phi)\in\mathscr{C}(P)$ satisfies $\nabla_A\Phi\in L^p(X)$ for all $2\leqslant p\leqslant 6$. Therefore, the above Theorem \ref{thm: alternative_finite_mass} applies to any finite energy configuration on such classes of $3$-manifolds. Furthermore, in these cases, by combining the conclusions of Theorem \ref{thm: alternative_finite_mass} with those of Theorem \ref{thm: finite_mass} we get that the mass $m$ of $(A,\Phi)$ equals the number $m(|\Phi|)$ given by Lemma \ref{lem: mass} and this number is also determined by \eqref{eq: def_w}-\eqref{eq: alt_characterization_m} (see Remark \ref{rmk: nR_case_finite_mass}).
\end{remark}

\subsection{Charge and energy formula}\label{subsec: charge}
We continue to consider the main setting of the previous section, \emph{i.e.}, unless otherwise stated, $(X^3,g)$ is an AC $3$-manifold with only one end, $P\to X$ is a principal $G$-bundle and $\mathscr{C}(P)$ is the space of smooth finite energy configurations on $P$.

\begin{definition}[Charge]
	Define\footnote{$k'$ is well defined by the finite energy condition and H\"older's inequality.} $k':\mathscr{C}(P)\to\mathbb{R}$ by
	\begin{equation}
		k'(A,\Phi):=\frac{1}{4\pi}\int_X \langle\nabla_A\Phi\wedge F_A\rangle.
	\end{equation} Then we define the \textbf{charge}, or monopole number, $k:\mathscr{C}(P)\to\mathbb{R}$ by
	\begin{equation}
		k(A,\Phi):=\begin{cases}
			m(|\Phi|)^{-1}k'(A,\Phi),\quad\text{if }m(|\Phi|)>0.\\
			k'(A,\Phi),\quad\text{if }m(|\Phi|)=0.
		\end{cases}
	\end{equation}
\end{definition}
\begin{remark}
	Let $(A,\Phi)\in\mathscr{C}(P)$. We will show shortly that if $m:=m(|\Phi|)=0$ then $k:=k(A,\Phi)=k'(A,\Phi)=0$ (see Corollary \ref{cor: zero_mass_implies_zero_charge}); this means that $F_A$ and $\ast\nabla_A\Phi$ are $L^2$-orthogonal in this case, and $\mathcal{E}_X(A,\Phi) = \frac{1}{2}\|F_A \mp \ast\nabla_A\Phi\|_{L^2}^2$. In the case $m>0$, we shall see that $k\in\mathbb{Z}$ is a topological number determined by the pair $(A,\Phi)$ (see proof of Theorem \ref{thm: integrality}); in this case, we can write
	\[
	\mathcal{E}_X(A,\Phi) = \pm 4\pi m k + \frac{1}{2}\|F_A \mp \ast\nabla_A\Phi\|_{L^2(X)}^2,
	\] and we see that for fixed mass $m>0$ and charge $k\in\mathbb{Z}$, the absolute minimizers of the energy are solutions to the \emph{monopole} or \emph{anti-monopole equations}, $F_A = \pm \ast\nabla_A\Phi$, depending on whether $k\geqslant 0$ or $k\leqslant 0$ respectively.
\end{remark}
\begin{lemma}[cf. {\cite[Lemma 2]{groisser1984integrality}}]\label{lem: charge}
	Let $A\in\mathscr{A}(P)$ and define $\mathcal{H}_A$ to be the completion of the space $C_c^{\infty}\Gamma(\mathfrak{g}_P)$ of smooth compactly supported sections of $\mathfrak{g}_P$ in the norm $\|\varphi\|_{\mathcal{H}_A}:= \|\nabla_A\varphi\|_{L^2(X)}$.
	\begin{itemize}
		\myitem[(i)]\label{itm: i_lemma} If $(A,\Phi_1),(A,\Phi_2)\in\mathscr{C}(P)$ and $\Phi_2-\Phi_1\in \mathcal{H}_A$, then $k'(A,\Phi_1)=k'(A,\Phi_2)$.
		\myitem[(ii)]\label{itm: ii_lemma} Given $(A,\Phi)\in\mathscr{C}(P)$ there exists a unique $\Phi'\in\Gamma(\mathfrak{g}_P)$ such that $\Phi'-\Phi\in \mathcal{H}_A$ and $\Delta_A\Phi' = 0$.
	\end{itemize}	
\end{lemma}
\begin{proof}
	\ref{itm: i_lemma}: Defining $\varphi:=\Phi_1-\Phi_2$, we want to show that
	\[
	\int_X \langle\nabla_A\varphi\wedge F_A\rangle = 0.
	\] For each $R\gg 1$, let $\chi_R\in C_c^{\infty}(X)$ be as in Notation \ref{notation: cut-off}. Using the finite energy condition $(A,\Phi)\in\mathscr{C}(P)$, by dominated convergence, the Bianchi identity and Stokes' theorem we have
	\begin{equation}\label{eq: aux_id}
	\int_X \langle\nabla_A\varphi\wedge F_A\rangle = \lim_{R\to\infty} \int_X\chi_R\langle\nabla_A\varphi\wedge F_A\rangle = \lim_{R\to\infty} \int_X (d\chi_R)\wedge\langle\varphi,F_A\rangle.  
	\end{equation} Since $\text{supp}(d\chi_R)\subseteq \overline{B}_{R}\setminus B_{R/2}$, H\"older's inequality yields
	\begin{equation}\label{eq: bound_last_integral}
	\left|\int_X (d\chi_R)\wedge\langle\varphi,F_A\rangle\right| \leqslant \|d\chi_R\|_{L^3(X)}\|\varphi\|_{L^6(\overline{B}_R\setminus B_{R/2})}\|F_A\|_{L^2(X)}.
	\end{equation} Now note that $\varphi\in \mathcal{H}_A$ implies $\varphi\in L^6(X)$; indeed, by Kato's inequality and the $L^2$-Sobolev inequality of Theorem \ref{thm: Sobolev_AC} we have $\|\varphi\|_{L^6}\lesssim \|d|\varphi|\|_{L^2}\lesssim \|\nabla_A\varphi\|_{L^2}<\infty$. Moreover, $\|d\chi_R\|_{L^3(X)}$ is bounded independently of $R$. Therefore, the right-hand side of \eqref{eq: bound_last_integral} goes to zero as $R\to\infty$ and so plugging this back to \eqref{eq: aux_id} we get the desired result.
	
	\ref{itm: ii_lemma}: This is similar to the proof of Lemma \ref{lem: mass}. Define the functional $Q:\mathcal{H}_A\to\mathbb{R}$ given by $Q(\varphi):=\mathcal{E}_X(A,\Phi + \varphi)$. Then $Q$ is strictly convex, coercive and lower semicontinuous. Therefore it achieves a unique absolute minimum, say at $\varphi\in \mathcal{H}_A$. Then $\Phi':=\Phi+\varphi$ solves $\Delta_A\Phi' = 0$ weakly. But again, as observed in the proof of (a) above, note that $\varphi\in \mathcal{H}_A$ implies $\varphi\in L^6$ and since $L^6\hookrightarrow L_{\text{loc}}^2$ we get $\varphi\in W_{\text{loc}}^{1,2}$. Since $(A,\Phi)$ is smooth, elliptic regularity implies that $\Phi'$ is in fact smooth and $\Delta_A \Phi' = 0$ holds strongly. This completes the proof.  
	
\end{proof}
In particular, we can reduce the problem of calculating the charge of a finite energy configuration $(A,\Phi)\in\mathscr{C}(P)$ to the case where $\Delta_A\Phi=0$.
\begin{corollary}\label{cor: mass_characterization}
	Let $(A,\Phi)\in\mathscr{C}(P)$ and let $m:=m(|\Phi|)$ be as in Lemma \ref{lem: mass}. Then $m\in [0,\infty)$; in fact, if we let $\Phi'$ be as in Lemma \ref{lem: charge} \ref{itm: ii_lemma}, then $m=m(|\Phi'|)\in [0,\infty)$.  
\end{corollary}
\begin{proof}
	Let $\Phi'\in\Gamma(\mathfrak{g}_P)$ be as in Lemma \ref{lem: charge} \ref{itm: ii_lemma}, \emph{i.e.} $\varphi:=\Phi'-\Phi\in \mathcal{H}_A$ and $\Delta_A\Phi'=0$. Note that $\|\nabla_A\Phi'\|_{L^2}\leqslant\|\nabla_A\Phi\|_{L^2} + \|\nabla_A\varphi\|_{L^2}<\infty$, so that by Kato's inequality $d|\Phi'|\in L^2$. Thus let $m':=m(|\Phi'|)$ be as in Lemma \ref{lem: mass}. Then, by the $L^2$-Sobolev inequality of Theorem \ref{thm: Sobolev_AC}, $m'-|\Phi'|\in H\hookrightarrow L^6$, and from Lemma \ref{lem: Phi_bounded} one has that $m'\in [0,\infty)$. Since $\Phi'-\Phi\in \mathcal{H}_A$, it follows from Kato's inequality together with the $L^2$-Sobolev inequality of Theorem \ref{thm: Sobolev_AC} that $||\Phi'|-|\Phi||\leqslant |\Phi'-\Phi|\in L^6$, so that
	$m'-|\Phi| = (m' - |\Phi'|) + (|\Phi'|-|\Phi|) \in L^6$. Therefore, the uniqueness part of Lemma \ref{lem: mass} implies that $m=m'\geqslant 0$.
\end{proof}
\begin{corollary}\label{cor: zero_mass_implies_zero_charge}
	If $(A,\Phi)\in\mathscr{C}(P)$ is such that $m(|\Phi|)=0$ then $k(A,\Phi) = k'(A,\Phi) = 0$. 
\end{corollary}
\begin{proof}
	Indeed, let $(A,\Phi)$ be as in the statement and let $\Phi'$ be given by Lemma \ref{lem: charge} \ref{itm: ii_lemma}. Then, by Lemma \ref{lem: charge} \ref{itm: i_lemma} we have $k'(A,\Phi)=k'(A,\Phi')$ and by Corollary \ref{cor: mass_characterization} we have $0=m(|\Phi|)=m(|\Phi'|)$. But Lemma \ref{lem: Phi_bounded} implies that $|\Phi'|\leqslant m(|\Phi'|)=0$, so that $\Phi'=0$ and therefore $k'(A,\Phi)=k'(A,\Phi')=0$.
\end{proof}

Now let us restrict ourselves to the case where the structure group $G=\rm SU(2)$. We start making some important remarks on the structure of the adjoint boundle $\mathfrak{su}(2)_P:=P\times_{\mathrm{Ad}}\mathfrak{su}(2)$ associated to the principal $\rm SU(2)$-bundle $P\to X$. 

We consider on $\mathfrak{su}(2)_P$ the metric $\langle\cdot{},\cdot{}\rangle$ induced by minus one-half the Cartan--Killing form of $\mathfrak{su}(2)$, \emph{i.e.} $\langle a,b\rangle := -2\tr(ab)$. If $\sigma_1,\sigma_2,\sigma_3$ denote the Pauli matrices, then
\[
T_1:=\frac{i\sigma_1}{2},T_2:=\frac{i\sigma_2}{2},T_3:=\frac{i\sigma_3}{2}
\] gives an orthonormal basis of $\mathfrak{su}(2)$ with respect to $\langle\cdot{},\cdot{}\rangle$, satisfying
\[
[T_1,T_2]=-T_3,\quad [T_1,T_3]=T_2,\quad [T_2,T_3]=-T_1.
\] In particular,
\[
[a,[b,c]] = b\langle a,c\rangle - c\langle a,b\rangle,\quad\forall a,b,c\in\mathfrak{su}(2).
\]
Given a Higgs field $\Phi\in\Gamma(\mathfrak{su}(2)_P)$, we shall denote by $Z(\Phi):=\{x\in X: \Phi(x)=0\}$ the (gauge invariant\footnote{Note that $Z(\Phi)=Z(Ad(\gamma)\Phi)$ for any gauge transformation $\gamma\in\mathscr{G}(P)$.}) zero locus of $\Phi$. To avoid cumbersome notation,  in what follows let us write $E:=\mathfrak{su}(2)_P$ and $V:=X\setminus Z(\Phi)$. Then we can decompose
\begin{equation}\label{eq: adjoint_decomp}
E|_{V}=E^{||} \oplus E^\perp,
\end{equation} where the \emph{longitudinal} line bundle $E^{||}$ is given by
\begin{equation}
	E^{||} = \ker \left(\ad(\Phi):E|_{V}\to E|_{V}\right) = \langle \Phi \rangle,
\end{equation}
and the \emph{transverse} rank $2$ bundle $E^\perp$ is the orthogonal complement of $E^{||}$. We note that
\begin{equation}\label{eq: decomp_Lie}
	[E^{\perp},E^{\perp}]\subset E^{||}\quad\text{and}\quad [E^{||},E^{\perp}]\subset E^{\perp}.
\end{equation}
If $(A,\Phi)$ has finite mass $m>0$ then there is $R_0\gg 1$ such that $|\Phi|\geqslant\frac{m}{2}>0$ over $X\setminus B_{R_0}$ (see Remark \ref{rmk: finite_mass}). In particular, in this case $Z(\Phi)\subset B_{R_0}$ and the above decomposition is well defined over $X\setminus B_{R_0}$. Moreover, it follows that $Z(\Phi)$ is a compact set, since it is closed ($Z=\Phi^{-1}(0)$) and bounded ($Z\subset B_{R_0}$) in the complete Riemannian manifold $(X,g)$. 

Henceforth, we split any section $\xi$ of $E=\mathfrak{su}(2)_P$, defined outside the zero locus of $\Phi$, as $\xi =\xi^{||} + \xi^\perp$; explicitly:
\begin{subequations}
	\begin{align}
		\xi^{||} &:=|\Phi|^{-2}\langle\xi,\Phi\rangle\Phi, \label{eq:long_part} \\
		\xi^{\perp} &:= |\Phi|^{-2}[\Phi,[\xi,\Phi]]. \label{eq:transv_part} \label{ineq:Taubes_outside_Z}
	\end{align} 
\end{subequations} It is clear that $\xi^{||}$ and $\xi^{\perp}$ are smooth on the complement of $Z(\Phi)$. For future use (in \S\ref{subsec: exp_decay}) we also note that
\begin{equation}\label{ineq:eigen_ad}
	|[\Phi,\xi]|\geqslant |\Phi||\xi^{\perp}|,
\end{equation} and, using \eqref{eq: decomp_Lie} and the Ad-invariance of the inner product,
\begin{equation}\label{eq: inner_prod_decomp}
	\langle [a,b],c\rangle = \langle [a^{||},b^{\perp}],c^{\perp}\rangle + \langle [a^{\perp},b^{||}],c^{\perp}\rangle + \langle [a^{\perp},b^{\perp}],c^{||}\rangle\quad\forall a,b,c\in \mathfrak{su}(2)_P.
\end{equation}

We end this section with the following main result (cf. \cite[Proposition $\S 2$]{groisser1984integrality}):
\begin{theorem}\label{thm: integrality}
	Assume $G=\rm SU(2)$. Let $(A,\Phi)\in\mathscr{C}(P)$ and suppose that $m=m(|\Phi|)>0$. Then $k=k(A,\Phi)\in\mathbb{Z}$ and if $(A,\Phi)$ has finite mass $m$ then
	\begin{equation}\label{eq: charge}
		\lim_{R\to\infty}\frac{1}{4\pi}\int_{\Sigma_R} |\Phi|^{-1}\langle\Phi,F_A\rangle = \lim_{R\to\infty}\frac{1}{4\pi m}\int_{\Sigma_R} \langle\Phi,F_A\rangle = k.
	\end{equation} Moreover, for $R\geqslant R_0\gg 1$, choosing a trivialization for $P\to X$ and restricting $\Phi/|\Phi|$ to $\Sigma_R\cong\Sigma^2$ determines a homotopy class of maps $\Sigma^2\to\mathbb{S}^2\subset\mathfrak{su}(2)$, and $k$ is the Brouwer degree of this class. Alternatively, the restrictions of the associated vector bundle $P\times_{\rm SU(2)}\mathbb{C}^2$ over $\Sigma_R$ split as $\mathscr{L}\oplus \mathscr{L}^{-1}$, where $\mathscr{L}$ is a complex line bundle over $\Sigma_{R}\cong\Sigma$, corresponding to one of the eigenspaces of $\Phi$, and the degree of any such $\mathscr{L}$ does not depend on $R$ and equals the charge $k$.
\end{theorem}
\begin{proof}
	We follow the original Euclidean case proof in \cite[Proposition $\S 2$]{groisser1984integrality} with minor additions. Start noting that as a consequence of Lemma \ref{lem: charge}, Corollary \ref{cor: mass_characterization} and Theorem \ref{thm: finite_mass}, we may assume that $(A,\Phi)$ has finite mass $m>0$.
	
	Also note that the last equality in \eqref{eq: charge} follows simply by using dominated convergence, the Bianchi identity and Stokes' theorem:
	\begin{equation}\label{eq: last_equality_k}
	4\pi m k=\int_X\langle\nabla_A\Phi\wedge F_A\rangle = \lim_{R\to\infty}\int_{\overline{B}_R}d\langle \Phi,F_A\rangle = \lim_{R\to\infty} \int_{\Sigma_R} \langle\Phi,F_A\rangle.
	\end{equation} 
	After making these initial considerations, we now proceed to the main part of the proof. Let $R_0>0$ be sufficiently large so that $|\Phi|\geqslant\frac{m}{2}>0$ on $U:=X\setminus B_{R_0}$, and write $\hat{\Phi}:=\Phi/|\Phi|$.
	
	Given $R\geqslant R_0$, let $\hat{\Phi}_R:=\hat{\Phi}|_{\Sigma_R}$.	Then the endomorphism $J:=ad(\hat{\Phi})=[\hat{\Phi},\cdot{}]$ restricted to $E_R^{\perp}:=E^{\perp}|_{\Sigma_R}$ satisfies $J^2=-1$. Therefore, for each $R\geqslant R_0$, we have the following decomposition
	\[
	E_R^{\perp}\otimes\mathbb{C} \cong L_R\oplus L_{R}^{\ast},
	\] where $L_{R}\to\Sigma_R$ is the complex line bundle defined by the $i$-eigenspace of $J$. The connection $A$ on $E|_U$ induces a connection $\tilde{A}$ on $E^{\perp}|_U$ by orthogonal projection:
	\[
	\nabla_{\tilde{A}} s := (\nabla_A s)^{\perp} = \nabla_A s - \langle \nabla_A s,\hat{\Phi}\rangle\hat{\Phi},\quad\forall s\in\Gamma(E^{\perp}|_U).
	\] The curvature $F_{\tilde{A}}\in\Omega^2(U,\text{End}(E_T))$ of $\tilde{A}$ is given by
	\[
	F_{\tilde{A}}(s) = d_{\tilde{A}} \nabla_{\tilde{A}} s = \left([F_A,s] - \langle\nabla_{A}\hat{\Phi},s\rangle\wedge\nabla_{A}\hat{\Phi}\right)^{\perp},\quad\forall s\in\Gamma(E^{\perp}|_U).
	\] Moreover, using \eqref{eq: decomp_Lie} one finds that $F_{\tilde{A}}(\cdot{})=f_{\tilde{A}}[\hat{\Phi},\cdot{}]$, where
	\[
	f_{\tilde{A}}:=\langle\hat{\Phi},F_A - \frac{1}{2}[\nabla_A\hat{\Phi},\nabla_A\hat{\Phi}]\rangle\in\Omega^2(U).
	\] Complexifying induces a connection on the line bundle $L_R$ with curvature form
	$\omega = i f_{\tilde{A}}$. In particular, $c_1(L_R)=\frac{i}{2\pi}[\omega]\in H^2(X,\mathbb{Z})$ and thus
	\begin{equation}\label{eq: degree_LR}
		\text{deg}(L_R):=\int_{\Sigma_R} c_1(L_R) = -\frac{1}{2\pi}\int_{\Sigma_R} \langle\hat{\Phi},F_A - \frac{1}{2}[\nabla_A\hat{\Phi},\nabla_A\hat{\Phi}]\rangle.
	\end{equation} Next we note that since $\rm SU(2)$ is $2$-connected, the bundle $P$ is necessarily trivializable. Let us fix a global trivialization of $P$ so that we can regard $\Phi$ as a map $X\to\mathfrak{su}(2)$. Then $\hat{\Phi}_R$ gives a well-defined homotopy class of maps from $\Sigma_R\cong\Sigma$ to the unit sphere $\mathbb{S}^2$ inside $\mathfrak{su}(2)$, which is independent of $R\geqslant R_0$, and furthermore one has $L_R\cong(\hat{\Phi}_R)^{\ast}(H^2)$, where $H$ denotes the Hopf bundle. In particular,
\begin{equation}\label{eq: deg_LR_PhiR}
\text{deg}(L_R) = -2\text{deg}(\hat{\Phi}_R)
\end{equation} and the values of $\text{deg}(L_R)$ and $\text{deg}(\hat{\Phi}_R)$ are independent of $R\geqslant R_0$ and of the choice of trivialization. Alternatively, one may also consider the bundle $\mathscr{E}:=P\times_{\rm SU(2)}\mathbb{C}^2$ associated with the standard representation, then observe that its restriction $\mathscr{E}_R$ to $\Sigma_R$ splits into eigenspaces for $\Phi|_{\Sigma_R}$ as $\mathscr{E}_R=\mathscr{L}_R\oplus\mathscr{L}_R^{\ast}$, where $\mathscr{L}_R^2\cong L_R$; in particular, using \eqref{eq: deg_LR_PhiR}, one has $\text{deg}(\mathscr{L}_R)=\text{deg}(\hat{\Phi}_R)$.
	
	Now let's analyze the limit as $R\to\infty$ of the right-hand side of \eqref{eq: degree_LR}. First, note that for $R\geqslant R_0$ we have
	\begin{align}
		\int_{\Sigma_R}\langle\hat{\Phi},F_A\rangle &= \int_{\Sigma_R}(1-\chi_{2R_0})\langle\hat{\Phi},F_A\rangle\\
		&= -\int_{\overline{B}_R}d\chi_{2R_0}\langle\hat{\Phi},F_A\rangle + \int_{\overline{B}_R}(1-\chi_{2R_0})\langle\nabla_A\hat{\Phi}\wedge F_A\rangle.\label{eq: aux_independence}
	\end{align} Since $\text{supp}(d\chi_{2R_0})\subset B_{2R_0}$, the first integral on the right-hand side of \eqref{eq: aux_independence} is independent of $R\geqslant 2R_0$. Moreover, for all $R\geqslant 2R_0$, we may compute the following on $X\setminus B_{R/2}$:
	\begin{align}
		\nabla_A\hat{\Phi} &= d\left(\frac{1}{|\Phi|}\right)\otimes\Phi + \frac{1}{|\Phi|}\nabla_A\Phi\\
		&= -\frac{1}{|\Phi|^3}\langle\nabla_A\Phi,\Phi\rangle\otimes\Phi + \frac{1}{|\Phi|}\nabla_A\Phi\\
		&= -\frac{1}{|\Phi|}(\nabla_A\Phi)^{||} + \frac{1}{|\Phi|}\nabla_A\Phi\\
		&= \frac{1}{|\Phi|}(\nabla_A\Phi)^{\perp},
	\end{align} and hence
	\begin{equation}
		|\nabla_A\hat{\Phi}| = \frac{1}{|\Phi|}|(\nabla_A\Phi)^{\perp}|\leqslant\frac{2}{m}|\nabla_A\Phi|\quad\text{on }X\setminus B_{R/2}.\label{eq: nabla_hat_Phi}
	\end{equation} Thus, using that $\text{supp}(1-\chi_{2R_0})\subset X\setminus B_{R_0}$ and H\"older's inequality we get
	\[
	\int_{\overline{B}_R}(1-\chi_{2R_0})|\langle\nabla_A\hat{\Phi}\wedge F_A\rangle|\leqslant\int_{X\setminus B_{R_0}}|\nabla_A\hat{\Phi}||F_A|\leqslant\frac{2}{m}\|\nabla_A\Phi\|_{L^2(X)}\|F_A\|_{L^2(X)}. 
	\] Hence, the last integral in \eqref{eq: aux_independence} is absolutely convergent as $R\to\infty$. In conclusion, we get that
	\begin{equation}\label{eq: existence_limit_k}
		\lim_{R\to\infty} \int_{\Sigma_R}\langle\hat{\Phi},F_A\rangle\quad\text{exists},
	\end{equation} and since $\text{deg}(L_R)$ is independent of $R\geqslant R_0$, it follows from \eqref{eq: degree_LR} that
	\begin{equation}\label{eq: limit_exists_transv}
	\lim_{R\to\infty} \int_{\Sigma_R}\langle\hat{\Phi},[\nabla_A\hat{\Phi},\nabla_A\hat{\Phi}]\rangle\quad\text{exists}.
	\end{equation} Let $\alpha\in\mathbb{R}$ be the limit in \eqref{eq: limit_exists_transv}. We claim that $\alpha=0$. Indeed, suppose on the contrary that
	\[
	0<|\alpha|=\lim_{R\to\infty}\left| \int_{\Sigma_R}\langle\hat{\Phi},[\nabla_A\hat{\Phi},\nabla_A\hat{\Phi}]\rangle\right|.
	\] Then there is $R_1\geqslant R_0$ so that for all $R\geqslant R_1$ one has
	\begin{align}
	0<\frac{|\alpha|}{2} &\leqslant\left| \int_{\Sigma_R}\langle\hat{\Phi},[\nabla_A\hat{\Phi},\nabla_A\hat{\Phi}]\rangle\right|\\
	&\leqslant\int_{\Sigma_R}|\langle\hat{\Phi},[\nabla_A\hat{\Phi},\nabla_A\hat{\Phi}]\rangle|\\
	&\leqslant\int_{\Sigma_R}|\Phi|^{-2}|(\nabla_A\Phi)^{\perp}|^2\\
	&\leqslant\frac{4}{m^2}\int_{\Sigma_R}|(\nabla_A\Phi)^{\perp}|^2.
	\end{align} Hence, using conical coordinates, we get
	\begin{align}
	\|\nabla_A\Phi\|_{L^2(X)}^2 &\geqslant\int_{X\setminus B_{R_1}}|(\nabla_A\Phi)^{\perp}|^2(x)\text{vol}_X(x) \\
	&=
	\int_{R_1}^{\infty}\int_{\Sigma}|(\nabla_A\Phi)^{\perp}|^2\rho^2\mu(\rho,\sigma)d\sigma d\rho\quad\text{ where $\mu=1+O(\rho^{-\nu})$}\\
	&\geqslant\frac{m^2|\alpha|\text{Vol}(\Sigma)}{16}\int_{R_1}^{\infty}\rho^2 d\rho = \infty,
	\end{align} contradicting the fact that $\nabla_A\Phi\in L^2(X)$. Thus $\alpha=0$ as claimed. Therefore, using \eqref{eq: degree_LR} and \eqref{eq: deg_LR_PhiR}, this yields
	\begin{equation}
		\lim_{R\to\infty}\frac{1}{4\pi}\int_{\Sigma_R}\langle\hat{\Phi},F_A\rangle = \text{deg}(\hat{\Phi}_{R_0})\in\mathbb{Z}.
	\end{equation} 

	Finally, we complete the proof of the theorem by proving the first equality in \eqref{eq: charge}. By \eqref{eq: last_equality_k} and \eqref{eq: existence_limit_k}, it suffices to prove that
	\begin{equation}\label{eq: pf_first_eq}
	\lim_{R\to\infty}\left|\int_{\Sigma_R}\left\langle\frac{\Phi}{m}-\hat{\Phi},F_A\right\rangle\right| = 0.
	\end{equation} Let $R\geqslant R_0$ and note that by the Bianchi identity and Stokes' theorem we have
	\begin{align}
	\left|\int_{\Sigma_R}\left\langle\frac{\Phi}{m}-\hat{\Phi},F_A\right\rangle\right| &= \left|\int_{\Sigma_R}(1-\chi_{R})\left\langle\frac{\Phi}{m}-\hat{\Phi},F_A\right\rangle\right|\\
	&\leqslant \frac{1}{m}\int_{\overline{B}_R}|d\chi_{R}|\left|m-|\Phi|\right||F_A| + \int_{X\setminus B_{R/2}}\left|\frac{1}{m}\nabla_A\Phi - \nabla_A\hat{\Phi}\right||F_A|.\label{ineq: bound_first_ident}
	\end{align} Arguing in the same way as in the proof of Lemma \ref{lem: charge} \ref{itm: i_lemma}, the first term in the right-hand side of \eqref{ineq: bound_first_ident} goes to zero as $R\to\infty$: indeed, by H\"older's inequality
	\[
	\int_{\overline{B}_R}|d\chi_{R}|\left|m-|\Phi|\right||F_A|\leqslant \|d\chi_{R}\|_{L^3}\|m-|\Phi|\|_{L^6(X\setminus B_{R/2})}\|F_A\|_{L^2};
	\] since $m-|\Phi|\in L^6(X)$ and $\|d\chi_{R}\|_{L^3}$ is uniformly bounded as $R\to\infty$, the claim follows. As for the second and last term in the right-hand side of \eqref{ineq: bound_first_ident}, note that for any $R\geqslant 2R_0$ we have \eqref{eq: nabla_hat_Phi} so that by H\"older's inequality
	\[
	\int_{X\setminus B_{R/2}}\left|\frac{1}{m}\nabla_A\Phi - \nabla_A\hat{\Phi}\right||F_A|\leqslant\frac{3}{m}\int_{X\setminus B_{R/2}}|\nabla_A\Phi||F_A|\leqslant\frac{3}{m}\|\nabla_A\Phi\|_{L^2(X\setminus B_{R/2})}\|F_A\|_{L^2(X\setminus B_{R/2})}, 
	\] and this bound goes to zero as $R\to\infty$, since $\nabla_A\Phi,F_A\in L^2(X)$. Putting it all together, equation \eqref{eq: pf_first_eq} holds as we wanted.

\end{proof}

\begin{corollary}\label{cor: energy_formula}
	Assume $G=\rm SU(2)$. Let $(A,\Phi)\in\mathscr{C}(P)$ be a monopole, \emph{i.e.} a solution to equation \eqref{eq: monopole}, and suppose that $m=m(|\Phi|)>0$. Then
	\[
	\mathcal{E}_X(A,\Phi)=4\pi m k,
	\] with $k=k(A,\Phi)\in\mathbb{N}_0$ and if $k>0$ then $Z(\Phi)\neq\emptyset$.
\end{corollary}
\begin{proof}
	Since $(A,\Phi)$ satisfies \eqref{eq: monopole}, we have
	\[
	4\pi mk= \int_X\langle\nabla_A\Phi\wedge F_A\rangle = \mathcal{E}_X(A,\Phi)\geqslant 0. 
	\] By Theorem \ref{thm: integrality} we know that $k\in\mathbb{Z}$, therefore the above implies $k\in\mathbb{N}_0$. Now, again by Theorem \ref{thm: integrality}, $k$ is the degree of the restrictions $(\Phi/|\Phi|)|_{\Sigma_R}$ for large enough $R\gg 1$, so that if the integer $k>0$ then $\Phi$ contains at least one zero inside $B_R$.
\end{proof}

\section{Asymptotics of critical points of the Yang--Mills--Higgs energy}\label{sec: YMH}

In this section we study analytical properties of general critical points of the Yang--Mills--Higgs energy, aiming at proving the sharp asymptotic decay rates stated in Theorem \ref{thm: main_2}.

\subsection{$\varepsilon$-regularity and consequences}\label{subsec: e-reg}

In this section, $(X^3,g)$ denotes a complete oriented Riemannian $3$-manifold of bounded geometry, and $P\to X$ is a principal $G$-bundle, where $G$ is a compact Lie group. In this general context, we study some analytical properties of solutions $(A,\Phi)$ to the second order equations \eqref{eq:2nd_Order_Eq_1} and \eqref{eq:2nd_Order_Eq_2}.

We start computing important Bochner--Weitzenb\"ock formulas for the rough Laplacian of $F_A$, $\ast F_A$ and $\nabla_A\Phi$, and a consequent nonlinear estimate for the Laplacian of the \emph{energy density} 
\[
e=e(A,\Phi):=\frac{1}{2}\left(|F_A|^2+|\nabla_A\Phi|^2\right).
\] 
\begin{lemma}\label{lem: Bochner_estimate}
	Suppose that $(A,\Phi)\in\mathscr{A}(P)\times\Gamma(\mathfrak{g}_P)$ is a solution to the second order equations  \eqref{eq:2nd_Order_Eq_1} and \eqref{eq:2nd_Order_Eq_2}. Then:
	\begin{align}
		\nabla_A^{\ast}\nabla_A(\nabla_A\Phi) &= \mathcal{R}ic_g{\#}\nabla_A\Phi - 2\ast[\ast F_A,\nabla_A\Phi] + [[\nabla_A\Phi,\Phi],\Phi],\label{eq: rough_bochner_dphi}\\
		\nabla_A^{\ast}\nabla_A(\ast F_A) &= \mathcal{R}ic_g{\#}(\ast F_A) - \ast[\ast F_A, \ast F_A] - \ast[\nabla_A\Phi,\nabla_A\Phi] + [[\ast F_A,\Phi],\Phi],\label{eq: rough_bochner_astF}\\		
		\nabla_A^{\ast}\nabla_A(F_A) &= \mathcal{R}ic_g{\#} F_A + F_A\#_{\mathfrak{g}} F_A - [\nabla_A\Phi,\nabla_A\Phi] + [[F_A,\Phi],\Phi],\label{eq: rough_bochner_curv}
	\end{align} where in an orthonormal frame we have
	\begin{align}
	(\mathcal{R}ic_g{\#}\nabla_A\Phi)_i &= -\mathcal{R}_{ik}(\nabla_A\Phi)_k,\\
	(\mathcal{R}ic_g{\#}(\ast F_A))_i &= -\mathcal{R}_{ik}(\ast F_A)_k,\\
	\left(\mathcal{R}ic_g{\#} F_A\right)_{ij} &= -\mathcal{R}_{ik}F_{jk}+\mathcal{R}_{jk}F_{ik}-S_g F_{ij},\quad\text{and}\\
	\left(F_A\#_{\mathfrak{g}} F_A\right)_{ij} &= -2[F_{ik},F_{kj}].
	\end{align}
	
	Consequently, setting $\eta:=\ast F_A - \nabla_A\Phi$, the following hold:
	\begin{align}\label{eq: bochner_dphi}
		\frac{1}{2}\Delta\lvert\nabla_A\Phi\rvert^2 = &\text{ }\langle \mathcal{R}ic_g{\#}\nabla_A\Phi,\nabla_A\Phi\rangle -2\langle\ast[\ast F_A,\nabla_A\Phi],\nabla_A\Phi\rangle\\
		&- \lvert[\nabla_A\Phi,\Phi]\rvert^2 - \lvert\nabla_A^2\Phi\rvert^2, 
	\end{align}
	\begin{align}\label{eq: bochner_xi}
		\frac{1}{2}\Delta\lvert\eta\rvert^2 = &\text{ }\langle \mathcal{R}ic_g{\#}\eta,\eta\rangle -\langle\ast[\eta,\eta],\eta\rangle - \lvert[\eta,\Phi]\rvert^2 - \lvert\nabla_A\eta\rvert^2,
	\end{align}
	\begin{align}\label{eq: bochner_curv}
		\frac{1}{2}\Delta\lvert F_A\rvert^2 = &\text{ }\langle \mathcal{R}ic_g{\#} F_A,F_A\rangle -\langle [\nabla_A\Phi,\nabla_A\Phi], F_A\rangle\\ &-2\sum_{i,j,k}\langle[F_{ik},F_{kj}],F_{ij}\rangle - \lvert[F_A,\Phi]\rvert^2 - \lvert\nabla_A F_A\rvert^2,
	\end{align} where in an orthonormal frame we have
\begin{align}
	\langle \mathcal{R}ic_g{\#}\nabla_A\Phi,\nabla_A\Phi\rangle &= -\mathcal{R}_{ik}\langle (\nabla_A\Phi)_i,(\nabla_A\Phi)_k\rangle\label{eq: inn_prod_nablaPhi}\\
	&=: -\mathcal{R}ic_g(\nabla_A\Phi,\nabla_A\Phi),\\
	\langle \mathcal{R}ic_g{\#}\eta,\eta\rangle &= -\mathcal{R}_{ik}\langle \eta_i,\eta_k\rangle\label{eq: inn_prod_xi}\\
	&=:-\mathcal{R}ic_g(\eta,\eta)\quad\text{and}\\
	\langle \mathcal{R}ic_g{\#}F_A,F_A\rangle &= -\sum_{i,j}\sum_{k}\mathcal{R}_{ik}\left\langle F_{ij}, F_{jk}\right\rangle - S_g\left|F_A\right|^{2} \label{eq: inn_prod_F}\\
	&=-\mathcal{R}_{ik}\langle (\ast F_A)_{i},(\ast F_A)_{k}\rangle\\
	&=: -\mathcal{R}ic_g(\ast F_A,\ast F_A).
\end{align}
	
 	In particular, if $\mathcal{R}ic_g\geqslant -\kappa g$ for some constant $\kappa\geqslant 0$ then the energy density $e=e(A,\Phi)$ satisfies
	\begin{equation}\label{ineq: non_linear_estimate_e}
		\Delta e\leqslant 2\kappa e + a_{\mathfrak{g}}e^{3/2},
	\end{equation} for some constant $a_{\mathfrak{g}}>0$ depending only on the structure constants of $\mathfrak{g}$.
\end{lemma}
\begin{proof}
	Using equations \eqref{eq:2nd_Order_Eq_1} and \eqref{eq:2nd_Order_Eq_2}, we compute:
	\begin{align}
		\Delta_A(\nabla_A\Phi) &= {d}_A^{\ast}{d}_A(\nabla_A\Phi)\\
		&= {d}_A^{\ast}[F_A,\Phi]\\
		&=\ast{d}_A[\ast F_A,\Phi]\\
		&=\ast\left([{d}_A\ast F_A,\Phi] - [\ast F_A,\nabla_A\Phi]\right)\\
		&=[{d}_A^{\ast} F_A,\Phi] - \ast [\ast F_A,\nabla_A\Phi]\\
		&=[[\nabla_A\Phi,\Phi],\Phi]- \ast [\ast F_A,\nabla_A\Phi].
	\end{align} Moreover, by the Bianchi identity $d_A F_A=0$ and \eqref{eq:2nd_Order_Eq_2} one has
\[
\Delta_A F_A = d_A d_A^{\ast} F_A = d_A [\nabla_A\Phi,\Phi] = [[F_A,\Phi],\Phi] - [\nabla_A\Phi,\nabla_A\Phi]
\] and
\[
\Delta_A (\ast F_A) = d_A^{\ast}d_A(\ast F_A) = \ast d_A d_A^{\ast}F_A = \ast d_A [\nabla_A\Phi,\Phi] = [[\ast F_A,\Phi],\Phi] - \ast[\nabla_A\Phi,\nabla_A\Phi].
\] Combining these with the standard Bochner--Weitzenb\"ock formula relating $\Delta_A$ with the rough Laplacian $\nabla_A^{\ast}\nabla_A$ on $\mathfrak{g}_P$-valued $1$-forms (see \cite[$\S$4.1]{oliveira2021yang}),
\[
\nabla_A^{\ast}\nabla_A\theta = \mathcal{R}ic_g\#\theta - \ast[\ast F_A,\theta] + \Delta_A\theta,\quad\forall\theta\in\Omega^1(X,\mathfrak{g}_P),
\] we already get \eqref{eq: rough_bochner_dphi} and \eqref{eq: rough_bochner_astF}. Moreover, using the standard Bochner--Weitzenb\"ock formula for $\mathfrak{g}_P$-valued $2$-forms [ibid.] we get 
\begin{equation}\label{eq: rough_bochner_curv_interm}
	\nabla_A^{\ast}\nabla_A(F_A) = \mathcal{R}_g{\#} F_A + F_A\#_{\mathfrak{g}}F_A + [[F_A,\Phi],\Phi] - [\nabla_A\Phi,\nabla_A\Phi],
\end{equation} where
\begin{align}\label{eq: Rm_sharp_F}
	\left(\mathcal{R}_g\#F_A\right)_{ij}  =-\underbrace{\mathcal{R}_{iklk}F_{lj}}_{\left(I\right)}-\underbrace{\mathcal{R}_{iklj}F_{kl}}_{\left(II\right)}+\underbrace{\mathcal{R}_{jklk}F_{li}}_{\left(III\right)}+\underbrace{\mathcal{R}_{jkli}F_{kl}}_{\left(IV\right)}.
\end{align} Thus, in order to establish \eqref{eq: rough_bochner_curv} we are left to show that $\mathcal{R}_g{\#} F_A = \mathcal{R}ic_g{\#} F_A$. We shall do so by using the very important and well-known fact that the Weyl conformal curvature tensor vanishes identically in dimension 3. This allow us to recover the full curvature tensor $\mathcal{R}_{ijkl}$ from the Ricci curvature $\mathcal{R}_{ij}$ through the following formula (see \cite[Theorem 8.1]{hamilton1982three}):
\begin{equation}\label{eq: Riem_from_Ric}
\mathcal{R}_{ijkl}=g_{ik}\mathcal{R}_{jl}-g_{il}\mathcal{R}_{jk}+\mathcal{R}_{ik}g_{jl}-\mathcal{R}_{il}g_{jk}-\frac{1}{2}S_g\left(g_{ik}g_{jl}-g_{il}g_{jk}\right).
\end{equation} Using this, we can rewrite all the terms in equation \eqref{eq: Rm_sharp_F} in terms of the Ricci curvature. Note that we just need to compute $(I)$ and $(II)$, because by simply switching the roles of $i$ and $j$ we get $(III)$ from $(I)$, and $(IV)$ from $(II)$. Now
\begin{align}
	\left(I\right) & =\left\{ g_{il}\mathcal{R}_{kk}-g_{ik}\mathcal{R}_{kl}+\mathcal{R}_{il}g_{kk}-\mathcal{R}_{ik}g_{kl}-\frac{1}{2}S_g\left(g_{il}g_{kk}-g_{ik}g_{kl}\right)\right\} F_{lj}\\
	& =S_gF_{ij}-\mathcal{R}_{il}F_{lj}+3\mathcal{R}_{il}F_{lj}-\mathcal{R}_{il}F_{lj}-\frac{1}{2}S_g\left(3F_{ij}-F_{ij}\right)\\
	& =\mathcal{R}_{il}F_{lj}.
\end{align}
Moreover,
\begin{align}
	\left(II\right) & =\left\{ g_{il}\mathcal{R}_{kj}-g_{ij}\mathcal{R}_{kl}+\mathcal{R}_{il}g_{kj}-\mathcal{R}_{ij}g_{kl}-\frac{1}{2}S_g\left(g_{il}g_{kj}-g_{ij}g_{kl}\right)\right\} F_{kl}\\
	& =\mathcal{R}_{kj}F_{ki}-\mathcal{R}_{kl}F_{kl}g_{ij}+\mathcal{R}_{il}F_{jl}-\frac{1}{2}S_g F_{ji}\\
	& =\mathcal{R}_{ik}F_{jk}+\mathcal{R}_{jk}F_{ki}-\mathcal{R}_{kl}F_{kl}g_{ij}+\frac{1}{2}S_g F_{ij}.
\end{align} Thus, we also have
\begin{align}
	\left(III\right) & = \mathcal{R}_{jl}F_{li},\quad\text{and}\\
	\left(IV\right) &=-\mathcal{R}_{ik}F_{jk}-\mathcal{R}_{jk}F_{ki}-\mathcal{R}_{kl}F_{kl}g_{ij}-\frac{1}{2}S_g F_{ij}
\end{align} Therefore,
\begin{align}
	\left(\mathcal{R}_g\#F_A\right)_{ij} & = -(I) -(II) + (III) + (IV)\\
	& =-\mathcal{R}_{il}F_{lj}+\mathcal{R}_{jl}F_{li}-2\mathcal{R}_{ik}F_{jk}+2\mathcal{R}_{jk}F_{ik}-S_gF_{ij}\\
	& = -\mathcal{R}_{ik}F_{jk}+\mathcal{R}_{jk}F_{ik} -S_g F_{ij}\\
	& =\left(\mathcal{R}ic_g\#F_A\right)_{ij},
\end{align} as we wanted.

Equations \eqref{eq: bochner_dphi}, \eqref{eq: bochner_xi} and \eqref{eq: bochner_curv} then follow form the previous formulas by using the Ad$_G$-invariance of the metric $\langle\cdot,\cdot\rangle$ and noting that
\[
\frac{1}{2}\Delta|\xi|^2=\langle\nabla_A^{\ast}\nabla_A\xi,\xi\rangle - |\nabla_A\xi|^2,\quad\forall\xi\in\Omega^k(X,\mathfrak{g}_P).
\] Equations \eqref{eq: inn_prod_nablaPhi} and \eqref{eq: inn_prod_xi} are clear. We now derive equation \eqref{eq: inn_prod_F}:
\begin{align}
	\left\langle \mathcal{R}ic_g\#F_A,F_A\right\rangle  & =\sum_{i<j}\left\langle \left(\mathcal{R}ic_g\#F_A\right)_{ij},F_{ij}\right\rangle \\
	& =-\sum_{i<j}\sum_{k}\mathcal{R}_{ik}\left\langle F_{jk},F_{ij}\right\rangle +\sum_{i<j}\sum_{k}\mathcal{R}_{jk}\left\langle F_{ik},F_{ij}\right\rangle -S_g\sum_{i<j}\left\langle F_{ij},F_{ij}\right\rangle \\
	&= -\sum_{i<j}\sum_{k}\mathcal{R}_{ik}\left\langle F_{jk},F_{ij}\right\rangle +\sum_{j<i}\sum_{k}\mathcal{R}_{ik}\left\langle F_{jk},F_{ji}\right\rangle -S_g\left|F_A\right|^{2}\\
	&=-\sum_{i,j}\sum_{k}\mathcal{R}_{ik}\left\langle F_{ij},F_{jk}\right\rangle -S_g\left|F_A\right|^{2}\\
	&=-2\sum_{i<k}\sum_{j}\mathcal{R}_{ik}\langle F_{ij}, F_{jk}\rangle +\sum_{i}\mathcal{R}_{ii}\left(\sum_{j}|F_{ij}|^2\right) - S_g|F_A|^2\\
	&=-2\sum_{i<k}\mathcal{R}_{ik}\langle (\ast F_A)_{i},(\ast F_A)_{k}\rangle - \sum_{i}\mathcal{R}_{ii}|(\ast F_A)_i|^2\\	
	&=-\sum_{i,k} \mathcal{R}_{ik}\langle (\ast F_A)_{i},(\ast F_A)_{k}\rangle,
\end{align} where the penultimate equality can be seen by computing the Hodge star at the center of normal coordinates. Finally, under the assumption $\mathcal{R}ic_g\geqslant -\kappa g$ for some $\kappa\geqslant 0$, inequality \eqref{ineq: non_linear_estimate_e} follows from the following rough estimate of the sum of equations \eqref{eq: bochner_dphi} and \eqref{eq: bochner_curv}, using also \eqref{eq: inn_prod_nablaPhi} and \eqref{eq: inn_prod_F}:
\begin{align}
\Delta e &\leqslant -\mathcal{R}ic_g(\nabla_A\Phi,\nabla_A\Phi) -2\langle\ast[\ast F_A,\nabla_A\Phi],\nabla_A\Phi\rangle \\
&\quad -\mathcal{R}ic_g(\ast F_A,\ast F_A) -\langle [\nabla_A\Phi,\nabla_A\Phi], F_A\rangle-2\sum_{i,j,k}\langle[F_{ik},F_{kj}],F_{ij}\rangle\\
&\leqslant 2\kappa e + a_{\mathfrak{g}}'|F_A||\nabla_A\Phi|^2 + a_{\mathfrak{g}}''|F_A|^3\\
&\leqslant 2\kappa e + a_{\mathfrak{g}}e^{3/2}, 
\end{align} where $a_{\mathfrak{g}},a_{\mathfrak{g}}',a_{\mathfrak{g}}''>0$ denote constants depending only on the structure constants of $\mathfrak{g}$.
\end{proof}

From the nonlinear estimate \eqref{ineq: non_linear_estimate_e} on the Laplacian of the energy density, a standard application of the so-called `Heinz trick' (\emph{cf.} \cite[Appendix B]{hohloch2009hypercontact} and \cite[Appendix A]{walpuski2017compactness}), using the mean value inequality of Proposition \ref{prop: MV_ineqs} \ref{itm: MV_iii}, implies the following $\varepsilon$-regularity result\footnote{For Yang--Mills connections, such an $\varepsilon$-regularity result goes back to the works of Uhlenbeck \cite[Theorem 3.5]{uhlenbeck1982removable} and Nakajima \cite[Lemma 3.1]{nakajima1988compactness}. As for general Yang--Mills--Higgs configurations, \emph{i.e.} solutions of \eqref{eq:2nd_Order_Eq_1} and \eqref{eq:2nd_Order_Eq_2}, see \cite[Theorem B]{afuni2019regularity}. See also \cite[Theorem 3.1]{tian2005bach} for another $\varepsilon$-regularity result outside the context of gauge theory.}. For any measurable subset $U\subset X$, we write
\[
\mathcal{E}_U(A,\Phi):=\|e\|_{L^1(U)}=\frac{1}{2}\left(\|F_A\|_{L^2(U)}^2+\|\nabla_A\Phi\|_{L^2(U)}^2\right).
\] 
\begin{theorem}\label{thm: e-regularity}
	Let $(X^3,g)$ be an oriented Riemannian $3$-manifold of bounded geometry and $P\to X$ be a principal $G$-bundle, where $G$ is a compact Lie group. Then there exist constants $\varepsilon_0,C_0>0$ and $r_0\in (0,2^{-1}\mathrm{inj}(X))$ with the following significance. Let $(A,\Phi)\in\mathscr{A}(P)\times\Gamma(\mathfrak{g}_P)$ be a solution of \eqref{eq:2nd_Order_Eq_1} and \eqref{eq:2nd_Order_Eq_2}. If $x\in X$ and $r\in (0,r_0]$ are such that
	\begin{equation}\label{ineq: smallness_condition}
		\varepsilon:= r\mathcal{E}_{B(x,r)}(A,\Phi)<\varepsilon_0,
	\end{equation} then
	\begin{equation}\label{ineq: e-regularity}
		\sup_{B(x,r/2)} e(A,\Phi) \leqslant C_0 r^{-4}\varepsilon.
	\end{equation} 
\end{theorem}
\begin{remark}
	The constants $\varepsilon_0$, $C_0$ and $r_0$ depend only on $\text{inj}(X)$, $\|\mathcal{R}_g\|_{L^{\infty}(X)}$ and the Sobolev constant\footnote{Since $(X^3,g)$ has bounded geometry, it follows from \cite[Lemma 2.24]{aubin2013some} that there is $0<r_0<\frac{1}{2}\text{inj}(X)$ small enough depending only on $\|\mathcal{R}_g\|_{L^{\infty}(X)}$, and a constant $C_{r_0}$ depending only on $r_0$ and $\|\mathcal{R}_g\|_{L^{\infty}(X)}$, such that $(X^3,g)$ satisfies the local $L^2$-Sobolev inequality
		\[
		\|df\|_{L^2(B(x,r))}^2\geqslant C_{r_0}\|f\|_{L^{6}(B(x,r))}^2, 
		\] for all $f\in C_c^{\infty}(B(x,r))$, $x\in X$ and $0<r\leqslant r_0$. By \cite[Lemma 20.11]{li2012geometric} this implies that there is a constant $c>0$, depending only on $C_{r_0}$, such that $V(x,r)\geqslant cr^3$, for all $0<r\leqslant r_0$. This lower bound estimate on the volume of balls is used combined with Proposition \ref{prop: MV_ineqs} \ref{itm: MV_iii} in the proof of the $\varepsilon$-regularity.}, and $\varepsilon_0$ depends furthermore on the structure constants of the Lie algebra $\mathfrak{g}$ of $G$ (through the constant $a_{\mathfrak{g}}$ appearing in front of the nonlinear term $e^{3/2}$ in \eqref{ineq: non_linear_estimate_e}). 
\end{remark}
\begin{corollary}\label{cor: e-reg_decay}
	Let $(X^3,g)$ be a noncompact oriented Riemannian $3$-manifold of bounded geometry, and let $P\to X$ be a principal $G$-bundle, where $G$ is a compact Lie group. Let $(A,\Phi)\in\mathscr{C}(P)$ be a solution of \eqref{eq:2nd_Order_Eq_1} and \eqref{eq:2nd_Order_Eq_2}. Then $|\nabla_A^j F_A|$ and $|\nabla_A^{j+1}\Phi|$ decay uniformly to zero at infinity for all $j\in\mathbb{N}_0$. Furthermore, the energy density $e\in L^p(X)$ for all $p\in[1,\infty]$; in fact, there is $C_{A,\Phi}>0$ depending only on $(A,\Phi)$, $(X^3,g)$ and $G$, such that $\|e\|_{L^p(X)}\leqslant C_{A,\Phi}\|e\|_{L^1(X)}$ for all $p\in (1,\infty]$.
\end{corollary}
\begin{proof}
	Let $\varepsilon_0, C_0$ and $r_0$ be the constants given by Theorem \ref{thm: e-regularity}. Since $(X,g)$ is of bounded geometry, we can cover $X$ with a countable collection of geodesic balls $\{B(x_i,{s})\}_{i=1}^{\infty}$ of radius $s:=\frac{1}{8}r_0$, with a uniform bound on the number of balls containing any point of $X$ and the half-radius balls pairwise disjoint (cf. \cite[Lemma 1.1]{hebey2000nonlinear}). Then since $e\in L^1(X)$ it follows that for each $\delta>0$ there exists $N_{\delta}\in\mathbb{N}$ so that up to removing a finite number of balls one has
	\[
	C_0 s^{-3}\|e\|_{L^1(B(x_i,4s))}<\delta,\quad\forall i>N_{\delta}.
	\] Thus, by Theorem \ref{thm: e-regularity}, we conclude that for any $\delta\in (0,4^{-1}s^{-4}C_0\varepsilon_0]$,
	\[
	\|e\|_{L^{\infty}(B(x_i,s))}<\delta,\quad\forall i>N_{\delta}.
	\] Thus $e$ decays uniformly to zero at infinity. In fact, by taking $\delta$ small enough one can make $\|F_A\|_{L^2 (B(x_i,4s))}$ to be smaller than Uhlenbeck's constant given by \cite{uhlenbeck1982connections}*{Theorem~1.3} and hence we can find a Coulomb gauge over $B(x_i,4s)$ for all sufficiently large $i$, in which the second order equations \eqref{eq:2nd_Order_Eq_1} and \eqref{eq:2nd_Order_Eq_2} become an elliptic system and standard elliptic estimates apply, implying the decay of both $|\nabla_A^j F_A|$ and $|\nabla_A^{j+1}\Phi|$ for all $j\in\mathbb{N}$.
	
	Note that, in particular, $e$ attains its maximum in $X$ and therefore $e\in L^1(X)\cap L^{\infty}(X)\subseteq L^p(X)$ for all $p\in [1,\infty)$. In fact, if $x_{\ast}\in X$ is such that $e(x_{\ast})=\|e\|_{L^{\infty}(X)}$ then by choosing $r_{\ast}\in (0,r_0]$ small enough such that $r_{\ast}\|e\|_{L^1(B(x_{\ast},{r_{\ast}}))}<\varepsilon_0$, then by \eqref{ineq: e-regularity} one has
	\[
	\|e\|_{L^{\infty}(X)}\leqslant C_0 r_{\ast}^{-3}\|e\|_{L^1(X)},
	\] and in particular, for all $p\in (1,\infty)$ one has
	\[
	\|e\|_{L^p(X)} \leqslant \|e\|_{L^{\infty}(X)}^{(p-1)/p}\|e\|_{L^1(X)}^{1/p}\leqslant \left(C_0 r_{\ast}^{-3}\right)^{(p-1)/p}\|e\|_{L^1(X)}\leqslant\max\{C_0 r_{\ast}^{-3}, 1\}\|e\|_{L^1(X)}.
	\] This completes the proof.
\end{proof}

We can also combine Corollary \ref{cor: e-reg_decay} with Theorem \ref{thm: alternative_finite_mass} to get a finite mass theorem for critical points of the energy on a wide class of complete nonparabolic $3$-manifolds:
\begin{corollary}\label{cor: finite_mass_for_crit_pts}
	Let $(X^3,g)$ be a complete nonparabolic Riemannian $3$-manifold of bounded geometry, satisfying the equivalent conditions of Theorem \ref{thm: PHI}. Let $P$ be a principal $G$-bundle over $X$, where $G$ is a compact Lie group. Suppose that $(A,\Phi)\in\mathscr{C}(P)$ is a solution to the second order equations \eqref{eq:2nd_Order_Eq_1} and \eqref{eq:2nd_Order_Eq_2}. Then $(A,\Phi)$ has finite mass $m$, where $m\in [0,\infty)$ is given by \eqref{eq: def_w}-\eqref{eq: alt_characterization_m}.
\end{corollary}
\begin{proof}
	By Corollary \ref{cor: e-reg_decay} it follows that $\nabla_A\Phi\in L^p(X)$ for all $p\geqslant 2$, and so we can apply Theorem \ref{thm: alternative_finite_mass} to get the desired result.
\end{proof}
\begin{remark}
	When restricted to AC $3$-manifolds or to complete $3$-manifolds with nonnegative Ricci curvature and maximal volume growth, Corollary \ref{cor: finite_mass_for_crit_pts} is clearly weaker than the combination of Theorems \ref{thm: finite_mass} and \ref{thm: alternative_finite_mass}, since the latter do not need to assume that equation \eqref{eq:2nd_Order_Eq_2} holds and their combined conclusion yield not only the characterization of the mass as in \eqref{eq: def_w}-\eqref{eq: alt_characterization_m} but also as the number $m(|\Phi|)$ given by Lemma \ref{lem: mass}; see Remark \ref{rmk: alt_finite_mass}.
\end{remark}
\begin{remark}
	Continue the hypotheses of Corollary \ref{cor: finite_mass_for_crit_pts}; \emph{e.g.}, one can suppose that $(X^3,g)$ is an AC oriented $3$-manifold or, alternatively, that $(X^3,g)$ is nonparabolic with nonnegative Ricci curvature and bounded geometry. We consider $(A,\Phi)\in\mathscr{C}(P)$ a critical point of the energy on a principal $G$-bundle $P\to X$.
	
	We observe that if $\nabla_A\Phi\neq 0$ then $|\nabla_A\Phi|\notin L^{p}(X)$ for any $p\in (0,1]$. By Corollary \ref{cor: e-reg_decay}, it suffices to show this claim for $p=1$. In order to do so, define the $2$-form $\gamma:=\langle\Phi,\ast\nabla_A\Phi\rangle\in\Omega^2(X)$ and observe that from \eqref{eq:2nd_Order_Eq_1} we have $d\gamma = |\nabla_A\Phi|^2\ast 1$. In particular, since $(A,\Phi)$ has finite energy, we have that $|d\gamma|\in L^1(X)$. Moreover, by Corollary \ref{cor: finite_mass_for_crit_pts} and Remark \ref{rmk: finite_mass}, we know that $(A,\Phi)$ has finite mass $m>0$ and therefore $|\gamma|\leqslant m|\nabla_A\Phi|$ on $X$. Thus if $|\nabla_A\Phi|\in L^1(X)$ then we also have $|\gamma|\in L^1(X)$ and therefore, by the generalized Stokes' theorem of Gaffney \cite{Gaffney54}, we get
	\[
	0 = \int_X d\gamma = \|\nabla_A\Phi\|_{L^2(X)}^2,
	\] contradicting our assumption $\nabla_A\Phi\neq 0$.
	
	Another related observation is that if $|\Phi|\neq 0$ then $|\Phi|\notin L^p(X)$ for any $p\in (0,\infty)$. Indeed, since $|\Phi|^2$ is a nonnegative subharmonic function (cf. \eqref{eq: subharmonic}), if $|\Phi|\in L^p(X)$ for some $p>2$ then it follows from a classical result of Yau \cite{yau1976some} that $|\Phi|$ must be constant, and since $X$ has infinite volume we get a contradiction. As for $p\in (0,2]$, note that by Corollary \ref{cor: finite_mass_for_crit_pts} and Remark \ref{rmk: finite_mass} one has $\|\Phi\|_{L^{\infty}(X)}\leqslant m<\infty$ and therefore if $|\Phi|\in L^p(X)$ then $\|\Phi\|_{L^3(X)}^3\leqslant m^{3-p}\|\Phi\|_{L^{p}(X)}^p<\infty$, contradicting the fact that $|\Phi|\notin L^3(X)$.  
\end{remark}

\subsection{Exponential decay of transverse components along the end}\label{subsec: exp_decay}
In this section, we prove the first two assertions of our second main theorem stated in the introduction, \emph{i.e.} parts \ref{itm: i_main_2} and \ref{itm: ii_main_2} of Theorem \ref{thm: main_2}.

From now on we shall restrict ourselves to principal $G$-bundles $P\to X$ with structure group $G=\rm SU(2)$. We start with some very useful Bochner--Weitzenb\"ock inequalities that one can derive by using Lemma \ref{lem: Bochner_estimate} together with the decomposition \eqref{eq: adjoint_decomp} of the adjoint bundle $\mathfrak{su}(2)_P$ outside the zero locus of a Higgs field. 
\begin{lemma}\label{lem: ref_Bochner}
	Let $(X^3,g)$ be an oriented Riemannian $3$-manifold of bounded geometry and let $P\to X$ be a principal $\rm SU(2)$-bundle. Let $(A,\Phi)\in\mathscr{A}(P)\times\Gamma(\mathfrak{su}(2)_P)$ be a solution to the second order equations \eqref{eq:2nd_Order_Eq_1} and \eqref{eq:2nd_Order_Eq_2}. Then outside the zero locus of $\Phi$ the following pointwise inequalities hold:
	\begin{align}\label{ineq: ref_Bochner_nablaPhi}
		\frac{1}{2} \Delta |\nabla_A \Phi|^2& + \lvert\nabla_A^2\Phi\rvert^2 + |\Phi|^2|(\nabla_A \Phi)^\perp|^2\\
		&\lesssim |\mathcal{R}ic_g||\nabla_A\Phi|^2 + |(F_A)^{||}| |(\nabla_A \Phi)^\perp|^2 + |(F_A)^{\perp}| |(\nabla_A \Phi)^{||}| |(\nabla_A \Phi)^\perp|,
	\end{align}
	\begin{align}\label{ineq: ref_Bochner_F}
		\frac{1}{2} \Delta |F_A|^2& + \lvert\nabla_A F_A\rvert^2 + |\Phi|^2|(F_A)^\perp|^2\\
		&\lesssim |\mathcal{R}ic_g||F_A|^2 + |(F_A)^{||}| |(\nabla_A \Phi)^\perp|^2 + |(F_A)^{\perp}| |(\nabla_A \Phi)^{||}||(\nabla_A \Phi)^\perp|\\
		&\quad\quad + |(F_A)^{\perp}|^2 |(F_A)^{||}|,
	\end{align}
	\begin{align}\label{ineq: ref_Bochner_transv_nablaPhi}
		\frac{1}{2} \Delta |[\nabla_A \Phi, &\Phi]|^2 + |\nabla_A[\nabla_A\Phi,\Phi]|^2  + |\Phi|^2|[\nabla_A \Phi, \Phi]|^2\\
		&\lesssim ( |\mathcal{R}ic_g| + |(F_A)^{||}|) |[\nabla_A \Phi, \Phi]|^2 + |(\nabla_A \Phi)^{||}||[F_A,\Phi]||[\nabla_A \Phi, \Phi]|,
	\end{align} and
	\begin{align}\label{ineq: ref_Bochner_transv_F}
		\frac{1}{2} \Delta &|[F_{A}, \Phi]|^2 + |\nabla_A[F_A,\Phi]|^2 + |\Phi|^2|[F_A, \Phi]|^2\\
		&\lesssim (|\mathcal{R}ic_g| + |F_A| + |\Phi|^{-2} |\nabla_A \Phi| ) \ |[F_A, \Phi]|^2 \\
		&\quad + \left(|(\nabla_A \Phi)^{||}| + |\Phi|^{-2} |F_A| |(\nabla_A \Phi)^{||}| + |\Phi|^{-1}|(\nabla_A F_A)^{||}| \right) |[\nabla_A \Phi, \Phi]| |[F_A, \Phi]| \\
		&\quad + |\Phi|^{-1} | \nabla_A[F_A, \Phi] | | (\nabla_A\Phi)^{||}| | [F_A, \Phi] |.
	\end{align}
\end{lemma}
\begin{proof}
	The first two inequalities \eqref{ineq: ref_Bochner_nablaPhi} and \eqref{ineq: ref_Bochner_F} follows directly from Lemma \ref{lem: Bochner_estimate} using the decomposition \eqref{eq: adjoint_decomp}, inequality \eqref{ineq:eigen_ad} and equation \eqref{eq: inner_prod_decomp}. 
	
	The proof of inequality \eqref{ineq: ref_Bochner_transv_nablaPhi} proceeds precisely in the same way as in the proof of \cite[Lemma 5.2]{fadel2020asymptotic}, except for the fact that in the present case one must of course keep the term involving the Ricci curvature whenever the Bochner formula for $\mathfrak{g}_P$-valued $1$-forms is used, as it does not necessarily vanish (note that in the case treated in \cite[Lemma 5.2]{fadel2020asymptotic} one has $\mathcal{R}ic_g=0$ because in that case the metric $g$ has $\mathrm{G}_2$-holonomy). 
	
	As for inequality \eqref{ineq: ref_Bochner_transv_F}, the proof is the same as in \cite[Lemma 5.3]{fadel2020asymptotic}, except that in the present case we can use the more refined Bochner--Weitzenb\"ock formula for $F_A$ given by \eqref{eq: rough_bochner_curv} so that in the end only the norm of the Ricci curvature appears in the final inequality, instead of the norm of the whole curvature tensor as in \cite[Lemma 5.3]{fadel2020asymptotic}. (Alternatively, one can use the standard Bochner--Weitzenb\"ock formula for $\mathfrak{g}_P$-valued $2$-forms and then in the end use the fact that in the present case $|\mathcal{R}_g|\sim|\mathcal{R}ic_g|$ by \eqref{eq: Riem_from_Ric}).
\end{proof}
As a consequence, we get the following Bochner inequalities along the end of an AC oriented $3$-manifold for critical points of the $\rm SU(2)$ Yang--Mills--Higgs energy:
\begin{corollary}\label{cor: bochner_along_the_end}
	Let $(X^3,g)$ be an AC oriented $3$-manifold with only one end and let $P\to X$ be a principal $\rm SU(2)$-bundle. Suppose that $(A,\Phi)\in\mathscr{C}(P)$ is a solution to the second order equations \eqref{eq:2nd_Order_Eq_1} and \eqref{eq:2nd_Order_Eq_2}. Denote by $m$ the mass of $(A,\Phi)$ given by Theorem \ref{thm: finite_mass} and suppose $m>0$. Then there is $R_0\gg 1$ such that for $\rho\geqslant R_0$, writing $e=e(A,\Phi)$ and $\Xi:=\left([\nabla_{A}\Phi,\Phi],[F_A,\Phi]\right)\in\Omega^1\oplus\Omega^2(X,\mathfrak{su}(2)_P)$, we have:
	\begin{equation}\label{ineq: ref_Bochner}
		\Delta e + \lvert\nabla_{A}^2\Phi\rvert^2 + \lvert\nabla_{A}F_A\rvert^2\lesssim |\mathcal{R}ic_g|e,
	\end{equation}
	and
	\begin{equation}\label{ineq: ref_Bochner_transv}
		\Delta |\Xi|^2  \leqslant -\frac{m^2}{4} |\Xi|^2.
	\end{equation}
\end{corollary}
\begin{proof}
	We start observing that, since $(A,\Phi)$ has finite mass $m>0$, it follows from Remark \ref{rmk: finite_mass} that for $R_0\gg 1$ we have
	\begin{equation}\label{ineq: lb_Phi}
	|\Phi|\geqslant\frac{m}{2}>0,\quad\text{ for $\rho\geqslant R_0$}.
	\end{equation} Next, note that for $\rho\geqslant R_0$,
	\[
	|(F_A)^{\perp}| |(\nabla_A \Phi)^{||}| |(\nabla_A \Phi)^\perp|\leqslant \frac{1}{2}|(\nabla_A \Phi)^{||}|\left(|(F_A)^{\perp}|^2 + |(\nabla_A \Phi)^\perp|^2\right).
	\] Thus, summing the inequalities \eqref{ineq: ref_Bochner_nablaPhi} and \eqref{ineq: ref_Bochner_F} given by Lemma \ref{lem: ref_Bochner}, and combining with the above observations, there is $c>0$ such that for $\rho\geqslant R_0$ we have
	\begin{align}
	\Delta e + |\nabla_A^2\Phi|^2 + |\nabla_A F_A|^2 &\leqslant \left(-\frac{m^2}{4} + c|(F_A)^{||}| + c|(\nabla_A\Phi)^{||}|\right)\left(|(F_A)^{\perp}|^2 + |(\nabla_A \Phi)^\perp|^2\right)\nonumber\\
	&\quad\quad + c|\mathcal{R}ic_g|e.\label{ineq: general_ref_Bochner}
	\end{align}
	Now, by Corollary \ref{cor: e-reg_decay}, the functions $|F_A|$ and $|\nabla_A\Phi|$ decay along the end, so that by taking $R_0\gg 1$ sufficiently large, we can arrange that for $\rho\geqslant R_0$ we have
	\[
	-\frac{m^2}{4} + c|(F_A)^{||}| + c|(\nabla_A\Phi)^{||}|\leqslant -\frac{m^2}{8}\leqslant 0,
	\] which combined with \eqref{ineq: general_ref_Bochner} immediately implies that the desired inequality \eqref{ineq: ref_Bochner} holds for $\rho\geqslant R_0$.
	
	The proof of inequality \eqref{ineq: ref_Bochner_transv} is similar, but with an additional step. We begin summing the inequalities \eqref{ineq: ref_Bochner_transv_nablaPhi} and \eqref{ineq: ref_Bochner_transv_F} from Lemma \ref{lem: ref_Bochner}, and combining this with the lower bound \eqref{ineq: lb_Phi}, and the simple fact
	\[
	|[F_A,\Phi]||[\nabla_A\Phi,\Phi]|\leqslant\frac{1}{2}\left(|[F_A,\Phi]|^2+|[\nabla_A\Phi,\Phi]|^2\right)=\frac{1}{2}|\Xi|^2,
	\] to get that there is $c>0$ such that for $\rho\geqslant R_0$
	\begin{align}
		\frac{1}{2}\Delta|\Xi|^2 &\leqslant -|\nabla_A\Xi|^2 + c\frac{2}{m}|\nabla_A[F_A,\Phi]||(\nabla_A\Phi)^{||}||[F_A,\Phi]|-\frac{m^2}{4}|\Xi|^2\nonumber\\
		&\text{ }+ c\left( |\mathcal{R}ic_g|+|F_A|+\frac{4}{m^2}|\nabla_A\Phi| + |(\nabla_A\Phi)^{||}| + \frac{4}{m^2}|F_A||(\nabla_A\Phi)^{||}| +\frac{2}{m}|(\nabla_A F_A)^{||}| \right)|\Xi|^2.\label{ineq: general_ref_Bochner_transv}
	\end{align}
	Now, on the one hand, by Corollary \ref{cor: e-reg_decay}, the functions $|F_A|$, $|\nabla_A F_A|$ and $|\nabla_A\Phi|$ decay uniformly to zero along the end, and since the metric $g$ is AC, one also has that $|\mathcal{R}ic_g|=O(\rho^{-2})$ decays (see Appendix \ref{app: A}). Thus, by taking $R_0\gg 1$ sufficiently large, we can make the whole term inside the parenthesis in the second line of the right-hand side of inequality \eqref{ineq: general_ref_Bochner_transv} to be less than, say, $\frac{1}{c}\frac{m^2}{16}$ for $\rho\geqslant R_0$. On the other hand, by Young's inequality,
	\begin{align}
	c\frac{2}{m}|\nabla_A[F_A,\Phi]||(\nabla_A\Phi)^{||}||[F_A,\Phi]| &\leqslant |\nabla_A[F_A,\Phi]|^2 + \frac{c^2}{m^2}|(\nabla_A\Phi)^{||}|^2|[F_A,\Phi]|^2\\
	&\leqslant |\nabla_A\Xi|^2 + \frac{c^2}{m^2}|(\nabla_A\Phi)^{||}|^2|\Xi|^2,
	\end{align}  and again since $|\nabla_A\Phi|$ decays, by taking $R_0\gg 1$, we have $\frac{c^2}{m^2}|(\nabla_A\Phi)^{||}|^2\leqslant\frac{m^2}{16}$ for $\rho\geqslant R_0$. In conclusion, combining these facts with inequality \eqref{ineq: general_ref_Bochner_transv}, we get the desired inequality \eqref{ineq: ref_Bochner_transv} for $\rho\geqslant R_0$.
\end{proof}
An immediate consequence of inequality \eqref{ineq: ref_Bochner} of Corollary \ref{cor: bochner_along_the_end} is the following first (non-optimal) decay rate result for the energy density.
\begin{lemma}\label{lem: rough_estimate}
	Continue the hypotheses of Corollary \ref{cor: bochner_along_the_end}. Then $e=e(A,\Phi)\in C_{-3}^{0}(X)$.
\end{lemma}
\begin{proof}
	Let $R_0$ be large enough so that inequality \eqref{ineq: ref_Bochner} holds on $\rho\geqslant R_0$. Then, since $|\mathcal{R}ic_g|=O(\rho^{-2})$, we have
	\[
	\Delta e \leqslant c\rho^{-2} e\quad\text{on } X\setminus B_{R_0}.
	\] Therefore, using that $e\in L^1(X)$ we can apply Lemma \ref{lem: decay} \ref{itm: decay_i} to get the desired result. 
\end{proof}
We now use the other inequality \eqref{ineq: ref_Bochner_transv} of Corollary \ref{cor: bochner_along_the_end} to deduce the exponential decay of the $\Phi$-transverse components of $F_A$ and $\nabla_A\Phi$. 
\begin{theorem}\label{thm: exp_decay_transv}
	Let $(X^3,g)$ be an AC oriented $3$-manifold with only one end and let $P\to X$ be a principal $\rm SU(2)$-bundle. Suppose that $(A,\Phi)\in\mathscr{C}(P)$ is a solution to the second order equations \eqref{eq:2nd_Order_Eq_1} and \eqref{eq:2nd_Order_Eq_2}. Denote by $m$ the mass of $(A,\Phi)$, given by Theorem \ref{thm: finite_mass}, and suppose that $m>0$. Then there is $R_0\gg 1$ such that for $\rho\geqslant R_0$,
	\begin{equation}
		|[\nabla_{A} \Phi,\Phi]|^2 + |[F_A,\Phi]|^2\lesssim m^2 \|e(A,\Phi)\|_{L^{\infty}(\Sigma_{R_0})}\cdot{} e^{-cm(\rho-R_0)}
	\end{equation} and, in particular,
	\begin{align}
		|(\nabla_{A} \Phi)^\perp|^2 + |F_A^\perp|^2 \lesssim \|e(A,\Phi)\|_{L^{\infty}(\Sigma_{R_0})} \cdot{} e^{-cm(\rho-R_0)}.  
	\end{align}
\end{theorem}
\begin{proof}
	Write $\Xi:=([\nabla_{A}\Phi,\Phi],[F_A,\Phi])$ and take $R_0\gg 1$ so that we have \eqref{ineq: ref_Bochner_transv}: 
	\begin{equation}
		\Delta|\Xi|^2 \leqslant -\frac{m^2}{4}|\Xi|^2\quad\text{on }X\setminus B_{R_0}.
	\end{equation}
	Now define a function $w=w(\rho)$ given by
	\begin{equation}\label{eq: def_w_M}
		w:= M e^{-cm\rho}\quad\text{with }M:=c'm^2 \|e(A,\Phi)\|_{L^{\infty}(\Sigma_{R_0})}\cdot{} e^{cm R_0}, 
	\end{equation} for constants $c,c'>0$ to be chosen later. Then
	\[
	\Delta w = (-c^2m^2|d\rho|^2-cm\Delta\rho)w.
	\] Since $(X^3,g)$ is AC, one has $|d\rho|^2 = 1+O(\rho^{-\nu})$ and $-\Delta\rho = 2\rho^{-1} + O(\rho^{-1-\nu})$ for some $\nu>0$. So taking $R_0\gg 1$, we may choose a suitable $c>0$, depending only on the geometry of $(X^3,g)$, so that
	\[
	\Delta w \geqslant -\frac{m^2}{4} w\quad\text{on }X\setminus B_{R_0}.
	\] Thus,
	\[
	\Delta(|\Xi|^2 - w) \leqslant -\frac{m^2}{4} (|\Xi|^2 - w)\quad\text{on }X\setminus B_{R_0}.
	\] Now, using Corollary \ref{cor: e-reg_decay} and the fact that $(A,\Phi)$ has finite mass $m$, we have
	\[
	|\Xi|^2\leqslant c'|\Phi|^2 e(A,\Phi)\leqslant c'm^2\|e(A,\Phi)\|_{L^{\infty}(\Sigma_{R_0})}\quad\text{on }\Sigma_{R_0},
	\] where $c'>0$ depends only on the structure constants of $\mathfrak{su}(2)$. Hence, by our choice of the constant $M$ in \eqref{eq: def_w_M} it follows that $|\Xi|^2\leqslant w$ on $\Sigma_{R_0}$. Furthermore, we know that $|\Xi|\to 0$ as $\rho\to\infty$ by Corollary \ref{cor: e-reg_decay} and the fact that $(A,\Phi)$ has finite mass $m$. Since we also have that $w\to 0$ as $\rho\to\infty$, it follows from the maximum principle that $|\Xi|^2\leqslant w$ on $\rho \geqslant R_0$, as we wanted.
\end{proof}

As a first consequence of Theorem \ref{thm: exp_decay_transv}, we get part \ref{itm: ii_main_2} of Theorem \ref{thm: main_2}.
\begin{corollary}\label{cor: decay_m-Phi}
	Continue the hypotheses of Theorem \ref{thm: exp_decay_transv}. Then there is $R_0\gg 1$ such that $m-|\Phi|\sim_{A,\Phi}\rho^{-1}$ for $\rho\geqslant R_0$.
\end{corollary}
\begin{proof}
	By Remark \ref{rmk: finite_mass} and Theorem \ref{thm: exp_decay_transv}, there is $R_0\gg 1$ such that setting $z_0:=\|e(A,\Phi)\|_{L^{\infty}(\Sigma_{R_0})}$ we have
	\[
	|\Phi|\geqslant\frac{m}{2}>0\quad\text{and}\quad |(\nabla_A\Phi)^{\perp}|^2\leqslant z_0e^{-cm(\rho-R_0)}\text{ on }X\setminus B_{R_0}.
	\] In particular,
	\[
	\Delta(m-|\Phi|) = |\Phi|^{-1}|(\nabla_A\Phi)^{\perp}|^2\leqslant 2 m^{-1}z_0e^{-cm(\rho-R_0)}\quad\text{on }X\setminus B_{R_0}.
	\] 
	Let $f:X\to\mathbb{R}$ be the smooth function given by $f(x):=2 m^{-1}z_0e^{-cm(\rho-R_0)}$, for all $x\in X$. Note that we have $f\in C_{\beta}^{0,\alpha}(X)$ for any $\beta<-3$ and $\alpha\in (0,1)$. Thus, by Theorem \ref{thm: VC_regularity} \ref{itm: VC_iv}, there is a unique solution $u\in C_{-1}^{2,\alpha}(X)$ to $\Delta u = f$, with $u=A\rho^{-1}+v$, where $v\in C_{-1-\mu}^{2,\alpha}(X)$ for some $\mu\in (0,\nu)$. Now take $M\geqslant 1$ large enough so that
	\[
	(m-|\Phi|)|_{\partial B_{R_0}}\leqslant Mu|_{\partial B_{R_0}}.
	\] Since $\Delta(m-|\Phi|)\leqslant f\leqslant Mf$ on $X\setminus B_{R_0}$, and both $m-|\Phi|$ and $u$ decay to zero at infinity, it follows from the maximum principle that $m-|\Phi|\leqslant Mu$ on $X\setminus B_{R_0}$. In particular, from the decay of $u$ we conclude that $m-|\Phi|\in C_{-1}^0(X)$. 
	
	Finally, since $m-|\Phi|$ is superharmonic and the Green's functions behave like $\sim\rho^{-1}$, another application of the maximum principle shows that $\rho^{-1}\lesssim_{A,\Phi} m-|\Phi|$ on $X\setminus B_{R_0}$. This completes the proof.
\end{proof}

\subsection{Quadratic decay of $\nabla_A\Phi$ and asymptotic expansion of $m-|\Phi|$}\label{subsec: quadratic_nablaPhi_asymp}
This section is dedicated to prove part \ref{itm: iii_main_2} of Theorem \ref{thm: main_2} and Corollary \ref{cor: main_asymp_exp_Phi}. 

We start proving the following:
\begin{theorem}\label{thm: quadratic_decay_nablaPhi}
	Let $(X^3,g)$ be an AC oriented $3$-manifold with only one end and let $P\to X$ be a principal $\rm SU(2)$-bundle. Suppose that $(A,\Phi)\in\mathscr{C}(P)$ is a solution to the second order equations \eqref{eq:2nd_Order_Eq_1} and \eqref{eq:2nd_Order_Eq_2}. Denote by $m$ the mass of $(A,\Phi)$, given by Theorem \ref{thm: finite_mass}, and suppose that $m>0$. Then $|\nabla_A\Phi|\in C_{-2}^0(X)$. Moreover, if we let $\mu\in (0,\nu)$ be as in \eqref{eq: Holder_mu}, then we can write 
	\begin{equation}\label{eq: asymp_exp_Phi2}
	|\Phi|^2 = m^2 -\frac{2\|\nabla_A\Phi\|_{L^2(X)}^2}{\mathrm{Vol}(\Sigma)}\frac{1}{\rho} + O(\rho^{-1-\mu}).
	\end{equation}
\end{theorem}
\begin{proof}
	We claim that it suffices to show that 
	\begin{equation}\label{eq: crucial_decay_nablaphi}
		|\nabla_A\Phi|^2\in C_{-3}^1(X).
	\end{equation}
	Indeed, assume that this decay property holds. On the one hand, since $C_{-3}^1(X)\hookrightarrow C_{-3}^{0,\alpha}(X)$, it follows from \eqref{eq: crucial_decay_nablaphi} that $|\nabla_A\Phi|^2\in C_{-3}^{0,\alpha}(X)$, for any given $\alpha\in (0,1)$. On the other hand, since $m^2-|\Phi|^2\leqslant 2m(m-|\Phi|)$ on $X$, it follows from Corollary \ref{cor: decay_m-Phi} that $m^2-|\Phi|^2\in C_{-1}^0(X)$. Thus, recalling from equation \eqref{eq: subharmonic} that $\Delta(m^2-|\Phi|^2) = 2|\nabla_A\Phi|^2$, it follows from Theorem \ref{thm: elliptic_reg} that we actually have $m^2-|\Phi|^2\in C_{-1}^{2,\alpha}(X)$. In particular, $|d(|\Phi|^2)|\lesssim_{A,\Phi}\rho^{-2}$ and therefore, for $\rho\geqslant R_0\gg 1$,
	\[
	|(\nabla_A\Phi)^{||}|=|d|\Phi||=2^{-1}|\Phi|^{-1}|d(|\Phi|^2)|\leqslant m^{-1}|d(|\Phi|^2)|\lesssim_{A,\Phi}\rho^{-2}.
	\] Using the exponential decay of the transverse component, $|(\nabla_A\Phi)^{\perp}|\lesssim_{A,\Phi} e^{-cm\rho}$, given by Theorem \ref{thm: exp_decay_transv}, we conclude that $|\nabla_A\Phi|\in C_{-2}^0(X)$. Furthermore, note that it follows from Theorem \ref{thm: VC_regularity} \ref{itm: VC_iv} that we can write $m^2-|\Phi|^2=A\rho^{-1}+v$, with $A:=\text{Vol}(\Sigma)^{-1}\int_X 2|\nabla_A\Phi|^2$ and $v\in C_{-1-\mu}^{2,\alpha}(X)$, thus showing \eqref{eq: asymp_exp_Phi2}.
	
	Therefore, we are left to show \eqref{eq: crucial_decay_nablaphi}. Define $\Psi:=(\nabla_A\Phi,F_A)\in\Omega^1\oplus\Omega^2(X,\mathfrak{su}(2)_P)$. By Lemma \ref{lem: rough_estimate}, we already known that $|\Psi|^2\in C_{-3}^0(X)$, so it suffices to show that $|\nabla_A\Psi|^2\in C_{-5}^0(X)$. 
	
	Let us first prove that $\rho\nabla_A\Psi\in L^2(X)$. Start noting that by the finite energy condition $\Psi\in L^2(X)$, and inequality \eqref{ineq: ref_Bochner} of Corollary \ref{cor: bochner_along_the_end}, it follows that for any $f\in W^{1,{\infty}}(X)$ we have\footnote{By density, it suffices to prove the identity \eqref{eq: Agmon} for $f\in C_c^{\infty}(X)$, which is done by using successive integration by parts, Leibniz rule and rearranging.}
	\begin{equation}\label{eq: Agmon}
	\|\nabla_A(f\Psi)\|_{L^2(X)}^2 = \|df\otimes \Psi\|_{L^2(X)}^2 + \int_X f^2\langle\nabla_A^{\ast}\nabla_A\Psi,\Psi\rangle.
	\end{equation} Now take $f=\rho_n$, where $\rho_n(x):=\min\{\rho(x),n\}$ for all $x\in X$. Since $|d\rho_n|\lesssim 1$, using inequality \eqref{ineq: ref_Bochner} and $|\mathcal{R}ic_g|=O(\rho^{-2})$ we have
	\begin{align}
		\|\nabla_A(\rho_n \Psi)\|_{L^2(X)}^2 \lesssim \|\Psi\|_{L^2(X)}^2 + \int_X \rho_n^2\rho^{-2}|\Psi|^2 \lesssim \mathcal{E}_X(A,\Phi).
	\end{align} Hence $\nabla_A(\rho \Psi)\in L^2(X)$, and since $\Psi\in L^2(X)$, we also get
	\[
	\rho\nabla_A\Psi = \nabla_A(\rho \Psi) - d\rho\otimes \Psi\in L^2(X),
	\] as we wanted. We now give an upper bound for $\Delta|\nabla_A\Psi|^2$ by bounding $\langle\nabla_A^{\ast}\nabla_A(\nabla_A\Psi),\nabla_A\Psi\rangle$. So we want to bound the right-hand side of the following two equations:
	\[
	\langle\nabla_A^{\ast}\nabla_A(\nabla_A^2\Phi),\nabla_A^2\Phi\rangle = \underbrace{\langle\nabla_A(\nabla_A^{\ast}\nabla_A(\nabla_A\Phi)),\nabla_A^2\Phi\rangle}_{(I)} + \underbrace{\langle [\nabla_A^{\ast}\nabla_A,\nabla_A](\nabla_A\Phi),\nabla_A^2\Phi\rangle}_{(II)}
	\] and
	\[
	\langle\nabla_A^{\ast}\nabla_A(\nabla_A F_A),\nabla_A F_A\rangle = \underbrace{\langle\nabla_A(\nabla_A^{\ast}\nabla_A F_A),\nabla_A F_A\rangle}_{(I')} + \underbrace{\langle [\nabla_A^{\ast}\nabla_A,\nabla_A](F_A),\nabla_A F_A\rangle}_{(II')}.
	\] We first deal with the first terms, $(I)$ and $(I')$, using the Bochner formulas from Lemma \ref{lem: Bochner_estimate}. Begin noting that taking the covariant derivative of \eqref{eq: rough_bochner_dphi} we get\footnote{Here we use $T\# Q$ to denote a generic multilinear expression involving the components of two tensors $T$ and $Q$ at most one of which is $\mathfrak{g}_P$-valued, while $T\#_{\mathfrak{g}}Q$ (note the subscript $\mathfrak{g}$) means a multilinear expression relating two $\mathfrak{g}_P$-valued tensors by combining its components using the Lie bracket $[\cdot{},\cdot{}]$.}
	\begin{align}
		\nabla_A(\nabla_A^{\ast}\nabla_A(\nabla_A\Phi)) &= \nabla\mathcal{R}ic_g\#\nabla_A\Phi + \mathcal{R}ic_g\#\nabla_A^2\Phi + \nabla_A F_A\#_{\mathfrak{g}}\nabla_A\Phi + F_A\#_{\mathfrak{g}}\nabla_A^2\Phi\\
		&\quad+ \nabla_A[[\nabla_A\Phi,\Phi],\Phi].
	\end{align} We now analyze each inner product arising from the above expression. First, since $|\nabla^j\mathcal{R}ic_g|=O(\rho^{-2-j})$, using Young's inequality and $|\Psi|^2\in C_{-3}^0(X)$, it follows that for $\rho\geqslant R_0$, we have:
	\begin{align}
		\langle\nabla\mathcal{R}ic_g\#\nabla_A\Phi + \mathcal{R}ic_g\#\nabla_A^2\Phi,\nabla_A^2\Phi\rangle &\lesssim |\nabla\mathcal{R}ic_g||\nabla_A\Phi||\nabla_A^2\Phi| + |\mathcal{R}ic_g||\nabla_A^2\Phi|^2\\
		&\lesssim_{A,\Phi} \rho^{-3}\rho^{-3/2}|\nabla_A^2\Phi| + \rho^{-2}|\nabla_A^2\Phi|^2\\
		&\lesssim_{A,\Phi} \rho^{-7} +\rho^{-2}|\nabla_A^2\Phi|^2.\label{eq: upper_estimate_nablaphi} 
	\end{align}
	Furthermore, using Young's inequality, Corollary \ref{cor: e-reg_decay} and the exponential decay of the transverse components given by Theorem \ref{thm: exp_decay_transv}, for $\rho\geqslant R_0$ we have:
	\begin{align}
		\langle \nabla_A F_A\#_{\mathfrak{g}}\nabla_A\Phi,\nabla_A^2\Phi\rangle &\lesssim |(\nabla_A F_A)^{||}||(\nabla_A\Phi)^{\perp}||(\nabla_A^2\Phi)^{\perp}| + |(\nabla_A F_A)^{\perp}||(\nabla_A\Phi)^{||}||(\nabla_A^2\Phi)^{\perp}|\\
		&\quad+|(\nabla_A F_A)^{\perp}||(\nabla_A\Phi)^{\perp}||(\nabla_A^2\Phi)^{||}|\\
		&\lesssim_{A,\Phi} e^{-cm\rho} + e^{-cm\rho}|\nabla_A^2\Phi|^2,\\
		\langle F_A\#_{\mathfrak{g}}\nabla_A^2\Phi,\nabla_A^2\Phi\rangle &\lesssim |F_A^{||}||(\nabla_A^2\Phi)^{\perp}|^2 + |F_A^{\perp}||(\nabla_A^2\Phi)^{\perp}||(\nabla_A^2\Phi)^{||}|\\
		&\lesssim_{A,\Phi} |F_A^{||}||(\nabla_A^2\Phi)^{\perp}|^2 + e^{-cm\rho} + e^{-cm\rho}|\nabla_A^2\Phi|^2,
	\end{align} and
	\begin{align}
	\langle\nabla_A[[\nabla_A\Phi,\Phi],\Phi],\nabla_A^2\Phi\rangle +|\Phi|^2|(\nabla_A^2\Phi)^{\perp}|^2 &\lesssim |(\nabla_A\Phi)^{\perp}|^2|(\nabla_A^2\Phi)^{||}|+ |(\nabla_A\Phi)^{\perp}||(\nabla_A\Phi)^{||}||(\nabla_A^2\Phi)^{\perp}|\\
	&\lesssim_{A,\Phi} e^{-cm\rho} + e^{-cm\rho}|\nabla_A^2\Phi|^2.
	\end{align} Also note that by taking $R_0\gg 1$, for $\rho\geqslant R_0$ we have
	\[
		c|F_A^{||}||(\nabla_A^2\Phi)^{\perp}|^2 - |\Phi|^2|(\nabla_A^2\Phi)^{\perp}|^2 \leqslant \left(c|F_A^{||}|-\frac{m^2}{4}\right)|(\nabla_A^2\Phi)^{\perp}|^2\leqslant 0,
	\] where the last inequality follows by Corollary \ref{cor: e-reg_decay}. In conclusion, we get that for $\rho\geqslant R_0$ we have
	\[
	(I)\lesssim_{A,\Phi} \rho^7 + \rho^{-2}|\nabla_A^2\Phi|^2.
	\]
	
	Now taking the covariant derivative of \eqref{eq: rough_bochner_curv} we get
	\begin{align}
		\nabla_A(\nabla_A^{\ast}\nabla_A F_A) &= \nabla\mathcal{R}ic_g\# F_A + \mathcal{R}ic_g\#\nabla_A F_A + \nabla_A F_A\#_{\mathfrak{g}} F_A - 2[\nabla_A^2\Phi,\nabla_A\Phi]\\
		&\quad+ \nabla_A[[F_A,\Phi],\Phi].
	\end{align} So by the same reasoning as before, note that for $\rho\geqslant R_0$ we have
\begin{equation}\label{eq: upper_estimate_F} 
\langle\nabla\mathcal{R}ic_g\# F_A + \mathcal{R}ic_g\#\nabla_A F_A,\nabla_A F_A\rangle \lesssim_{A,\Phi} \rho^{-7} + \rho^{-2}|\nabla_A F_A|^2,
\end{equation}
\begin{align}
	\langle \nabla_A F_A\#_{\mathfrak{g}} F_A, \nabla_A F_A\rangle &\lesssim |F_A^{||}||(\nabla_A F_A)^{\perp}|^2+|(\nabla_A F_A)^{||}||F_A^{\perp}||(\nabla_A F_A)^{\perp}|,\\
	&\lesssim_{A,\Phi} |F_A^{||}||(\nabla_A F_A)^{\perp}|^2 + e^{-cm\rho} + e^{-cm\rho}|\nabla_A F_A|^2,\\
	\langle [\nabla_A^2\Phi,\nabla_A\Phi],\nabla_A F_A\rangle &\lesssim |(\nabla_A^2 \Phi)^{||}||(\nabla_A\Phi)^{\perp}||(\nabla_A F_A)^{\perp}| + |(\nabla_A^2 \Phi)^{\perp}||(\nabla_A\Phi)^{\perp}||(\nabla_A F_A)^{||}|\\
	&\quad+|(\nabla_A^2 \Phi)^{\perp}||(\nabla_A\Phi)^{||}||(\nabla_A F_A)^{\perp}|,\\
	&\lesssim e^{-cm\rho} + e^{-cm\rho}|\nabla_A F_A|^2,
\end{align} and
\begin{align}
\langle\nabla_A[[F_A,\Phi],\Phi],\nabla_A F_A\rangle + |\Phi|^2|(\nabla_A F_A)^{\perp}|^2 &\lesssim  + |F_A^{\perp}||(\nabla_A\Phi)^{||}||(\nabla_A F_A)^{\perp}| + |F_A^{||}||(\nabla_A\Phi)^{\perp}||(\nabla_A F_A)^{\perp}|\\
&\quad + |F_A^{\perp}||(\nabla_A\Phi)^{\perp}||(\nabla_A F_A)^{||}|\\
&\lesssim e^{-cm\rho} + e^{-cm\rho}|\nabla_A F_A|^2.
\end{align} Moreover, again by Corollary \ref{cor: e-reg_decay}, for $\rho\geqslant R_0$ we have
\[
c|F_A^{||}||(\nabla_A F_A)^{\perp}|^2 - |\Phi|^2|(\nabla_A F_A)^{\perp}|^2 \leqslant \left(c|F_A^{||}|-\frac{m^2}{4}\right)|(\nabla_A F_A)^{\perp}|^2\leqslant 0.
\] Hence
\[
(I')\lesssim_{A,\Phi} \rho^{-7} + \rho^{-2}|\nabla_A F_A|^2.
\]

On the other hand, for any $u\in\Gamma(V\otimes\mathfrak{g}_P)$, where $V\to X$ is a tensor bundle, we have (see \cite[Lemma 20]{Floer1995config})
	\[
	[\nabla_A^{\ast}\nabla_A,\nabla](u) = -\mathcal{R}ic_g(d_A u) - \ast[\ast F(\nabla_A),d_A u] - d_A^{\ast}[u,F(\nabla_A)],
	\] where $F(\nabla_A)$ denotes the curvature of the connection $\nabla_A$ on $V\otimes\mathfrak{g}_P$. Thus, when upper estimating the terms $(II)$ and $(II')$ we will only get terms that already appear in our estimates for $(I)$ and $(I')$ respectively. Therefore, we conclude that for $\rho\geqslant R_0$ we have
	\[
	\Delta |\nabla_A\Psi|^2 \lesssim_{A,\Phi} \rho^{-7} + \rho^{-2}|\nabla_A\Psi|^2.
	\]	
	From this differential inequality along the end and the fact that $\rho\nabla_A\Psi\in L^2(X)$, we can use Lemma \ref{lem: decay} \ref{itm: decay_ii} to conclude that $|\nabla_A\Psi|^2\in C_{-5}^0(X)$. This completes the proof.
\end{proof}
As a consequence, we now prove Corollary \ref{cor: main_asymp_exp_Phi}:
\begin{corollary}\label{cor: asymp_exp_Phi}
	Continue the hypotheses of Theorem \ref{thm: exp_decay_transv}. Let $\mu\in (0,\nu)$ be as in \eqref{eq: Holder_mu}. Then we have
	\begin{equation}\label{eq: asymp_Phi}
		|\Phi| = m - \frac{\|\nabla_A\Phi\|_{L^2(X)}^2}{m\mathrm{Vol}(\Sigma)}\frac{1}{\rho} + O(\rho^{-1-\mu})\quad\text{as }\rho\to\infty.
	\end{equation} In particular, if $(A,\Phi)$ is furthermore a monopole of mass $m>0$ and charge $k$ then
	\begin{equation}\label{eq: asymp_Phi_monopole}
		|\Phi| = m - \frac{4\pi k}{\mathrm{Vol}(\Sigma)}\frac{1}{\rho} + O(\rho^{-1-\mu})\quad\text{as }\rho\to\infty.
	\end{equation}
\end{corollary}
\begin{proof}
	We prove \eqref{eq: asymp_Phi}; then \eqref{eq: asymp_Phi_monopole} follows using the energy formula \eqref{eq: Energy_Formula}. Choose $R_0$ large enough so that $|\Phi|\geqslant\frac{m}{2}>0$ on $X\setminus B_{R_0}$. Let $\phi:=1-\chi_{2R_0}$, $u:=m-|\Phi|$ and $f:=|\Phi|^{-1}|(\nabla_A\Phi)^{\perp}|^2$. Then $\Delta u=f$ on $X\setminus B_{R_0}$ and
	\[
	\Delta(\phi u)=\phi f + u\Delta\phi - 2\langle d\phi,du\rangle\quad\text{on }X.
	\] Now, on the one hand, note that the last two terms in the right hand side of the above equation are compactly supported. On the other hand, using the exponential decay of the transverse components given by Theorem \ref{thm: exp_decay_transv}, and the fact that $\nabla^2\Phi$ decays (Corollary \ref{cor: e-reg_decay}), it follows that $\phi f$ and its derivative decay exponentially. Since $\phi u$ decays like $\rho^{-1}$ (by Corollary \ref{cor: decay_m-Phi}), it follows from Theorem \ref{thm: elliptic_reg} and Theorem \ref{thm: VC_regularity} \ref{itm: VC_iv} that $\phi u\in C_{-1}^{2,\alpha}(X)$ and
	\[
	\phi u  = A\rho^{-1} + v,
	\] where $A:=\frac{1}{\text{Vol}(\Sigma,g_{\Sigma})}\int_X \Delta(\phi u)$ and $v\in C_{-1-\mu}^{2,\alpha}(X)$. Now note that $\phi\equiv 1$ for $\rho\geqslant 2R_0$, so by dominated convergence and Stokes' theorem we have
	\begin{align}
	\int_X \Delta(\phi u) &= \lim_{R\to\infty} \int_{B_R}\Delta(\phi u)\\
	&= -\lim_{R\to\infty} \int_{\partial B_R}\partial_r (\phi u)\\
	&= -\lim_{R\to\infty} \int_{\partial B_R}\partial_r u\\
	&= \lim_{R\to\infty} \int_{B_R}\Delta u = \int_X \Delta u.
	\end{align}
	
	By the definitions of $\phi$, $u$ and $A$, it follows that we can write
	\[
	m-|\Phi| = \frac{-\int_X\Delta |\Phi|}{\mathrm{Vol}(\Sigma)}\frac{1}{\rho} + O(\rho^{-1-\mu}),\quad\text{on }X\setminus B_{2R_0}.
	\] In order to finish the proof, it remains to show that
	\[
	-\int_X\Delta |\Phi| = \frac{1}{m}\int_X|\nabla_A\Phi|^2.
	\] Start noting that by the finiteness of both integrals, dominated convergence and Stokes' theorem we have
	\begin{align}
		-\int_X\Delta |\Phi| = -\lim_{R\to\infty}\int_{B_R}\Delta |\Phi| = \lim_{R\to\infty}\int_{\partial B_R}\partial_r|\Phi|
	\end{align} and, further using \eqref{eq:2nd_Order_Eq_1},
	\begin{align}
		\int_X |\nabla_A\Phi|^2 &= \lim_{R\to\infty}\int_{B_R} |\nabla_A\Phi|^2\\
		&= -\lim_{R\to\infty}\frac{1}{2}\int_{B_R}\Delta |\Phi|^2 \\
		&= \lim_{R\to\infty}\frac{1}{2}\int_{\partial B_R}\partial_r(|\Phi|^2)\\
		&= \lim_{R\to\infty}\int_{\partial B_R}|\Phi|\partial_r|\Phi|.
	\end{align} Now using the quadratic decay $|\nabla_A\Phi|\lesssim_{A,\Phi}\rho^{-2}$ given by Theorem \ref{thm: quadratic_decay_nablaPhi}, note that
	\[
	\int_{\partial B_R}|\nabla_A\Phi|\lesssim_{A,\Phi} 1.	
	\] Therefore, using Kato's inequality $|\partial_r|\Phi||\leqslant|\nabla_A\Phi|$ and that $(A,\Phi)$ has finite mass $m$ we get:
	\begin{align}
		\lim_{R\to\infty}\int_{\partial B_R}\left|\left(1-\frac{|\Phi|}{m}\right)\partial_r|\Phi|\right| \lesssim \lim_{R\to\infty} \frac{1}{m}\sup_{\partial B_R}(m-|\Phi|) = 0.
	\end{align} Putting it all together shows the desired equality and completes the proof. 
\end{proof}

\subsection{Quadratic decay of the curvature along the end}\label{subsec: quadratic}
This section is dedicated to the proof of part \ref{itm: iv_main_2} of Theorem \ref{thm: main_2}. Throughout this section, let $(X^3,g)$ be an AC oriented $3$-manifold with rate $\nu>0$, connected asymptotic link $\Sigma^2$, and fix a radius function $\rho$ on $X$. Let $P$ be a principal $\rm SU(2)$-bundle over $X$.

We start noting that in the case \ref{itm: a_main_2} of Theorem \ref{thm: main_2}, part \ref{itm: iv_main_2} follows immediately from Theorem \ref{thm: quadratic_decay_nablaPhi}. Thus in this section we focus on proving \ref{itm: iv_main_2} for general critical points under the assumption \ref{itm: b_main_2}. That is, we prove the following: 
\begin{theorem}\label{thm: quadratic_decay}
	Assume that the Gaussian curvature $K^{\Sigma}$ of $\Sigma$ is positive. Let $(A,\Phi)\in\mathscr{C}(P)$ be a solution to the second order equations \eqref{eq:2nd_Order_Eq_1} and \eqref{eq:2nd_Order_Eq_2}. Denote by $m$ the mass of $(A,\Phi)$, given by Theorem \ref{thm: finite_mass}, and suppose that $m>0$. Then $|F_A|\in C_{-2}^{0}(X)$.
\end{theorem}
\begin{remark}
	The proof of Theorem \ref{thm: quadratic_decay} that is given below does not use anything from \S\ref{subsec: quadratic_nablaPhi_asymp}, and in fact it also yield an independent proof of the quadratic decay of $\nabla_A\Phi$ in the case where the asymptotic link $\Sigma$ has positive Gaussian curvature.
\end{remark}
The remainder of this section is devoted to the proof of Theorem \ref{thm: quadratic_decay}.

We start proving the following refined Kato inequalities with ``error terms''; the proof is a minor modification of the proof of a general result appearing in \cite[Theorem 5]{smith2019removeability}.
\begin{proposition}\label{prop: refined_Kato}
	Let $A\in\mathscr{A}(P)$ and suppose that $F\in \Omega^2(X,\mathfrak{g}_P)$ is such that\footnote{\emph{e.g.} when $F=F_A$ (Bianchi identity).} $d_A F=0$, and $\theta\in\Omega^1(X,\mathfrak{g}_P)$ is such that\footnote{\emph{e.g.} when $\theta=\nabla_A\Phi$ and \eqref{eq:2nd_Order_Eq_1} holds.} $d_A^{\ast}\theta=0$. Then
	\begin{equation}\label{ineq: refined_Kato_F}
		\frac{3}{2}\left|d|F|\right|^2\leqslant |\nabla_A F|^2 + |d_A^{\ast}F|^2
	\end{equation} and
	\begin{equation}\label{ineq: refined_Kato_nablaPhi}
		\frac{3}{2}\left|d|\theta|\right|^2\leqslant |\nabla_A \theta|^2 + |d_A\theta|^2.
	\end{equation}
\end{proposition}
\begin{proof}
	We start noting that the stated inequalities do not follow directly from the result stated in \cite[Theorem 5]{smith2019removeability}. Indeed, the refined inequality appearing in \cite[Theorem 5]{smith2019removeability} for general bundle-valued $2$-forms is proved only in dimensions $n\geqslant 4$, and although the proof of the refined inequality for general bundle-valued $1$-forms given in that reference is valid in any dimension $n\geqslant 3$, the corresponding result is weaker than \eqref{ineq: refined_Kato_nablaPhi} since the constant factor appearing in the left-hand side in the general case is $(n+1)/n$, and for $n=3$ this is less than the constant $3/2$ as we stated in \eqref{ineq: refined_Kato_nablaPhi}. Nevertheless, it is explained below how the proof of the inequalities in \cite[Theorem 5]{smith2019removeability} can be easily modified to prove inequalities \eqref{ineq: refined_Kato_F} and \eqref{ineq: refined_Kato_nablaPhi} by considering our particular situation, where $n=3$ and the $2$-form $F$ and the $1$-form $\theta$ satisfy the equations $d_A F=0$ and $d_A^{\ast}\theta=0$.
	
	In order to prove inequality \eqref{ineq: refined_Kato_F}, one follows the proof of inequality (3.2) in \cite[Theorem 5]{smith2019removeability} and notes that the right-hand side of equation (3.5) in that reference has only 2 terms instead of 3 in the present situation where $d_A F=0$, so when it is stated afterwards in that reference that ``Each such replacement has either 3 or $n-1$ terms'', one should note that in our case each such replacement has $n-1=2$ terms instead, so that the Cauchy--Schwarz inequality that is used afterwards can still be used, with $n-1=2$, and the rest of the proof goes through yielding the desired inequality. 
	
	As for proving inequality \eqref{ineq: refined_Kato_nablaPhi}, one follows the proof of inequality (3.3) in \cite[Theorem 5]{smith2019removeability} and notes that the right-hand side of equation (3.14) in that reference has only $n-1=2$ terms instead of $n$ in the present case where $d_A^{\ast}\theta=0$, so that the Cauchy--Schwarz inequality that is used afterwards can be used with $n$ replaced by $n-1=2$, and following the rest of the proof one gets the desired conclusion.
\end{proof}
We now get the following refined Bochner inequality along the end under an additional assumption on the asymptotic link.
\begin{proposition}\label{prop: refined_Kato_Bochner}
	Continue the hypotheses of Theorem \ref{thm: quadratic_decay} and suppose furthermore that the Gaussian curvature $K^{\Sigma}$ of the asymptotic link $\Sigma$ satisfies $K^{\Sigma}\geqslant 1$. If we set $\Psi:=(\nabla_A\Phi,F_A)\in\Omega^1\oplus\Omega^2(X,\mathfrak{su}(2)_P)$, then there is $R_0\gg 1$ such that for $\rho\geqslant R_0$ we have
	\begin{equation}\label{ineq: refined_Bochner_Psi}
		\Delta|\Psi|^{1/2}\lesssim \rho^{-2-\nu}|\Psi|^{1/2}.
	\end{equation}
\end{proposition}
\begin{proof}
	For any $\alpha\in (0,1)$ a standard computation gives
	\begin{align}
	\Delta |\Psi|^{\alpha} &= \alpha |\Psi|^{\alpha - 2}\left((2-\alpha)|d|\Psi||^2 + \frac{1}{2}\Delta|\Psi|^2\right)\nonumber\\
	&= \alpha |\Psi|^{\alpha - 2}\left((2-\alpha)|d|\Psi||^2 + \langle\nabla_A^{\ast}\nabla_A \Psi,\Psi\rangle -|\nabla_A\Psi|^2\right).\label{eq: std_computation}
	\end{align} On the other hand, it follows from the second order equations \eqref{eq:2nd_Order_Eq_1} and \eqref{eq:2nd_Order_Eq_2}, the Bianchi identity $d_A F_A=0$ and Proposition \ref{prop: refined_Kato} that
	\begin{equation}\label{ineq: der_refined_Kato}
		-|\nabla_A\Psi|^2\leqslant -\frac{3}{2}|d|\Psi||^2 + |[F_A,\Phi]|^2 + |[\nabla_A\Phi,\Phi]|^2.
	\end{equation}
	Combining \eqref{eq: std_computation} and \eqref{ineq: der_refined_Kato} implies
	\[
	\Delta |\Psi|^{\alpha} \leqslant \alpha |\Psi|^{\alpha - 2}\left((1/2-\alpha)|d|\Psi||^2 + \langle\nabla_A^{\ast}\nabla_A \Psi,\Psi\rangle + |[F_A,\Phi]|^2 + |[\nabla_A\Phi,\Phi]|^2\right),
	\] so that by taking $\alpha=1/2$ we have
	\begin{equation}
	\Delta |\Psi|^{1/2} \leqslant \frac{1}{2}|\Psi|^{1/2 - 2}\left(\langle\nabla_A^{\ast}\nabla_A \Psi,\Psi\rangle + |[F_A,\Phi]|^2 + |[\nabla_A\Phi,\Phi]|^2\right).
	\end{equation}
	Now we use the Bochner--Weitzenb\"ock formulas of Lemma \ref{lem: Bochner_estimate} to get  
	 \begin{align}\label{eq: improved_bochner_F}
	 	\Delta |\Psi|^{1/2} \leqslant&\text{ }\frac{1}{2}|\Psi|^{1/2 - 2}(\langle \mathcal{R}ic_g{\#}\nabla_A\Phi,\nabla_A\Phi\rangle -2\langle\ast[\ast F_A,\nabla_A\Phi],\nabla_A\Phi\rangle\\
	 	&+\langle \mathcal{R}ic_g{\#} F_A,F_A\rangle -\langle [\nabla_A\Phi,\nabla_A\Phi], F_A\rangle -\sum_{i,j,k}\langle[F_{ik},F_{kj}],F_{ij}\rangle).
	 \end{align}
	 Next we observe that using \eqref{eq: inner_prod_decomp} and the exponential decay of the transverse components given by Theorem \ref{thm: exp_decay_transv}, together with Young's inequality, it follows that the three inner products on the right-hand side of \eqref{eq: improved_bochner_F} that contain Lie brackets $[\cdot{},\cdot{}]$ can all be bounded above by $c'e^{-c''m\rho}|\Psi|^2$, which in turn can be bounded by $c\rho^{-2-\nu}|\Psi|^2$. Hence, in order to conclude the proof, it remains for us to obtain upper bounds of the form $c\rho^{-2-\nu}|\Psi|^2$ on the other two remaining inner products, $\langle \mathcal{R}ic_g{\#}\nabla_A\Phi,\nabla_A\Phi\rangle$ and $\langle \mathcal{R}ic_g{\#}F_A,F_A\rangle$. 
	 
	 According to the computation \eqref{eq: Ricci_AC} of Appendix \ref{app: A}, in an adapted orthonormal frame for the cone metric, we can write the AC metric $g_{ij}$ and its Ricci tensor $\mathcal{R}_{ij}$ in such a way that, for $\rho\geqslant R_0$, we have $g_{ij} = \delta_{ij} + O(\rho^{-\nu})$ for all $i,j$, and $\mathcal{R}_{ij}=O(\rho^{-2-\nu})$ if $i\neq j$ or $i=j=1$, and moreover $\mathcal{R}_{22}$ and $\mathcal{R}_{33}$ are of the form $\rho^{-2}(K^{\Sigma}-1)+O(\rho^{-2-\nu})$. Now recall from Lemma \ref{lem: Bochner_estimate} the precise formulas \eqref{eq: inn_prod_nablaPhi} and \eqref{eq: inn_prod_F} for the inner products $\langle \mathcal{R}ic_g{\#}\nabla_A\Phi,\nabla_A\Phi\rangle$ and $\langle \mathcal{R}ic_g{\#}F_A,F_A\rangle$, respectively. Note that these formulas are given in an orthonormal frame with respect to $g$ and, as one can readily see from the proof of such formulas, in the present frame we get the extra terms $O(\rho^{-2-\nu})|\nabla_A\Phi|^2$ in \eqref{eq: inn_prod_nablaPhi} and $O(\rho^{-2-\nu})|F_A|^2$ in \eqref{eq: inn_prod_F}. Using this and the formulas for the Ricci curvature in this frame, we can compute the following:
	 \begin{align}
	 	\langle \mathcal{R}ic_g{\#}\nabla_A\Phi,\nabla_A\Phi\rangle &= O(\rho^{-2-\nu})|\nabla_A\Phi|^2-\sum_{i,k} \mathcal{R}_{ik}\langle(\nabla_A\Phi)_i,(\nabla_A\Phi)_k\rangle\\
	 	&=O(\rho^{-2-\nu})|\nabla_A\Phi|^2-\sum_{i\neq k} \mathcal{R}_{ik}\langle(\nabla_A\Phi)_i,(\nabla_A\Phi)_k\rangle - \sum_{i} \mathcal{R}_{ii}|(\nabla_A\Phi)_i|^2\\
	 	&=O(\rho^{-2-\nu})|\nabla_A\Phi|^2 + O(\rho^{-2-\nu})|(\nabla_A\Phi)_1|^2\\
	 	&\quad+(-\rho^{-2}(K^{\Sigma}-1)+O(\rho^{-2-\nu}))(|(\nabla_A\Phi)_2|^2+|(\nabla_A\Phi)_3|^2)\\
	 	&\lesssim \rho^{-2-\nu}|\nabla_A\Phi|^2, \quad\text{(since $K^{\Sigma}\geqslant 1$)}
	 \end{align} and
	 \begin{align}
	 	\langle \mathcal{R}ic_g{\#}F_A,F_A\rangle &= O(\rho^{-2-\nu})|F_A|^2-\sum_{i,j}\sum_{k}\mathcal{R}_{ik}\left\langle F_{jk},F_{ij}\right\rangle -S_g\left|F_A\right|^{2}\\ 
	 	&= O(\rho^{-2-\nu})|F_A|^2 -\sum_{j}\sum_{i\neq k}\sum_{k}\mathcal{R}_{ik}\left\langle F_{jk},F_{ij}\right\rangle  +2\sum_{i<j}\mathcal{R}_{ii}\left|F_{ij}\right|^{2}-S_g\sum_{i<j}\left|F_{ij}\right|^{2}\\
	 	& = O(\rho^{-2-\nu})|F_A|^2+ O(\rho^{-2-\nu})|F_A|^2 +\sum_{i<j}\left(2\mathcal{R}_{ii}-S_g\right)\left|F_{ij}\right|^{2}\\
	 	& = O(\rho^{-2-\nu})|F_A|^2+(-2\rho^{-2}(K^{\Sigma}-1)+O(\rho^{-2-\nu}))(|F_{12}|^2 + |F_{13}|^2) + O(\rho^{-2-\nu})|F_{23}|^2\\
	 	&\lesssim\rho^{-2-\nu}|F_A|^2. \quad\text{(again since $K^{\Sigma}\geqslant 1$)}
	 \end{align}
  This completes the proof.
\end{proof}
We are now ready to prove the main theorem of this section.
\begin{proof}[Proof of Theorem \ref{thm: quadratic_decay}]
	We start recalling the following basic scaling property of the second order equations \eqref{eq:2nd_Order_Eq_1} and \eqref{eq:2nd_Order_Eq_2}: if $\lambda\in(0,\infty)$, then $(A,\Phi)$ is a solution to \eqref{eq:2nd_Order_Eq_1} and \eqref{eq:2nd_Order_Eq_2} on $(X^3,g)$ if and only if $(A,\lambda^{-1}\Phi)$ is a solution to \eqref{eq:2nd_Order_Eq_1} and \eqref{eq:2nd_Order_Eq_2} on $(X^3,\lambda^2 g)$ (see \emph{e.g.} \cite[Proposition 2.1]{fadel2019limit}). It follows that the estimates
	\[
	\sup_X \rho^2|F_A|<\infty\quad\text{and}\quad\sup_X \rho^2|\nabla_A\Phi|<\infty
	\] that we want to prove are invariant under these rescalings. Therefore, since $K^{\Sigma}>0$, after scaling we can assume that $K^{\Sigma}\geqslant 1$. Thus, by Proposition \ref{prop: refined_Kato_Bochner}, there is $R_0\gg 1$ and $C>0$ such that the function $u:=|\Psi|^{1/2}$ satisfies
	\begin{equation}\label{ineq: crucial_quadratic}
	\Delta u\leqslant f u \quad\text{on }X\setminus B_{R_0},
	\end{equation} where $f:=C\rho^{-2-\nu}\in L^{q}(X)$ for any $q\in [3/2,\infty)$. Moreover, $u\in L^4(X)$ since $u^4=|F_A|^2+|\nabla_A\Phi|^2\in L^1(X)$. Recalling also that $(X^3,g)$ satisfies the $L^2$-Sobolev inequality of Theorem \ref{thm: Sobolev_AC} and for any fixed reference point $o\in X$ there is $C_o>0$ such that $V(o,r)\leqslant C_o r^n$ for any $r>0$, we are then able to apply the decay result of Proposition \ref{prop: BKN_decay} to conclude
	\begin{equation}\label{eq: almost_quadratic}
	u=O_{\varepsilon}(\rho^{-1+\varepsilon})\quad\text{for any }\varepsilon>0. 
	\end{equation}
	Now let $\mu\in(0,\nu)$ be as in \eqref{eq: Holder_mu} and take $\varepsilon:=\nu-\mu/2>0$. Then from \eqref{ineq: crucial_quadratic} and \eqref{eq: almost_quadratic} we get $\Delta u\leqslant C'\rho^{-3-\mu/2}$ on $X\setminus B_{R_0}$, for some $C'>0$. Using Theorem \ref{thm: embedding}, it is easy to see that $h:=C'\rho^{-3-\mu/2}\in C_{-3-\mu/2}^{k,\alpha}$ for any $\alpha\in (0,1)$ and $k\in\mathbb{N}_0$. Therefore, it follows from Theorem \ref{thm: VC_regularity} \ref{itm: VC_iv} that there is a unique $v\in C_{-1}^{k+2,\alpha}$ with $\Delta v = h$. By taking $M\geqslant 1$ large enough, we can assume $u\leqslant Mv$ on $\partial B_{R_0}$, and since $u$ decays to zero at infinity (Corollary \ref{cor: e-reg_decay}), by the maximum principle we have $u\leqslant Mv$ on $X\setminus B_{R_0}$, from where we conclude that $u\in C_{-1}^0(X)$, \emph{i.e.} $|F_A|,|\nabla_A\Phi|\in C_{-2}^0(X)$ as we wanted.
\end{proof} 
\begin{remark}[Alternative proof of the quadratic decay of $F_A$]
	Continue the hypotheses of Theorem \ref{thm: quadratic_decay}. Then, using the Bianchi identity and the second order equations \eqref{eq:2nd_Order_Eq_1} and \eqref{eq:2nd_Order_Eq_2}, it follows that $\eta:=\ast F_A - \nabla_A\Phi$ is a $\mathfrak{g}_P$-valued $1$-form satisfying $d_A^{\ast}\eta = 0$ and $d_A\eta = -[\ast\eta,\Phi]$. Thus, using the refined Kato inequality of Proposition \ref{prop: refined_Kato} for $\eta$, together with the Bochner formula \eqref{eq: bochner_xi} in Lemma \ref{lem: Bochner_estimate}, and doing the same computation as in the proof of Proposition \ref{prop: refined_Kato_Bochner}, we get that
	\[
	\Delta |\eta|^{1/2}\leqslant\frac{1}{2}|\eta|^{1/2- 2}(-\mathcal{R}ic_g(\eta,\eta)-\langle\ast[\eta,\eta],\eta\rangle).
	\] In particular, if $K^{\Sigma}\geqslant 1$ then using the same arguments as in the proof of Proposition \ref{prop: refined_Kato_Bochner} we get, for $\rho\geqslant R_0$,
	\[
	\Delta |\eta|^{1/2}\lesssim \rho^{-2-\nu}|\eta|^{1/2}.
	\] Then we can proceed just like in the above proof of Theorem \ref{thm: quadratic_decay} to prove that whenever $K^{\Sigma}>0$, then after scaling we have that $u:=|\eta|^{1/2}$ satisfies \eqref{ineq: crucial_quadratic}. So that following the same arguments of that proof, using that $u^4=|\eta|^2=|\ast F_A -\nabla_A\Phi|^2\in L^1(X)$, we can conclude that $|\ast F_A - \nabla_A\Phi|\in C_{-2}^{0}(X)$. This combined with the quadratic decay of $\nabla_A\Phi$ proved in Theorem \ref{thm: quadratic_decay_nablaPhi}, also gives $|F_A|\in C_{-2}^0(X)$.
\end{remark}

\begin{remark}[On the positive Gaussian curvature assumption]
	On the above proof of Theorem \ref{thm: quadratic_decay}, the positivity assumption on the Gaussian curvature $K^{\Sigma}$ of $\Sigma$, combined with a scaling argument and Proposition \ref{prop: refined_Kato_Bochner}, led to the strong differential inequality \eqref{ineq: crucial_quadratic} which in turn was the crucial ingredient to deduce the desired quadratic decay rate $|F_A|\in C_{-2}^0(X)$.
	
	Now, without any hypothesis on $K^{\Sigma}$ or, more generally, without any further \emph{a priori} knowledge on the Ricci curvature tensor $\mathcal{R}ic_g$ of $X$, and/or further hypothesis on $(A,\Phi)$, one can see (by following the proof of Proposition \ref{prop: refined_Kato_Bochner}) that the combination of the refined Kato inequalities with the Bochner formulas and the exponential decay of the transverse components \emph{a priori} just imply that the function $u:=|\Psi|^{1/2}$, which lies in $L^4(X)$, satisfies a differential inequality of the following form:
	\begin{equation}\label{eq: crit_dif_ineq}
	\Delta u\leqslant C\rho^{-2}u\quad\text{on}\quad X\setminus B_{R_0},
	\end{equation} for some constants $C>0$ and $R_0\gg 1$. From here, it follows from Lemma \ref{lem: decay} \ref{itm: decay_i}, or alternatively from the Moser iteration result of Lemma \ref{lem: moser_iteration}, that the function $u$ satisfies
	\begin{equation}\label{eq: apriori_best_decay}
		u=O(\rho^{-3/4}).
	\end{equation} Therefore, in principle, from the differential inequality \eqref{eq: crit_dif_ineq} we would get only that $|F_A|,|\nabla_A\Phi|\in C_{-3/2}^0(X)$, which turns out to be the non-optimal statement that we had already proven in Lemma \ref{lem: rough_estimate}. 

	We note that, in the abstract, \emph{i.e.} without further knowledge on the constant $C$ in \eqref{eq: crit_dif_ineq}, or on the integrability of $u$, the conclusion \eqref{eq: apriori_best_decay} is actually the sharpest general statement on the polynomial decay rate of a nonnegative function $u\in L^4(X)$ satisfying \eqref{eq: crit_dif_ineq}. Indeed, consider the following simple example. Suppose $X=\mathbb{R}^3$ with the standard flat metric and for each $\beta\in (3/4, 1)$ let $u_{\beta}:=\rho^{-\beta}$, where $\rho$ is a radius function; by definition $\rho$ is a smooth extension of the distance function $r(x):=|x|^{1/2}$ in $\mathbb{R}^3$ such that $\rho\geqslant 1$ throughout $\mathbb{R}^3$ and $\rho=r$ on $\mathbb{R}^3\setminus B_2$. Thus, $u_{\beta}$ is a smooth nonnegative function in $L^4(\mathbb{R}^3)$ satisfying $\Delta u_{\beta} = \beta(1-\beta)\rho^{-2}u_{\beta}$ on $\mathbb{R}^3\setminus B_2$, so that in particular $u_{\beta}$ satisfies \eqref{eq: crit_dif_ineq} for any $C\geqslant 3/16\geqslant\beta(1-\beta)>0$ and $R_0\geqslant 2$. Notice that one can take $\beta$ as close as $3/4$ as one wants.
\end{remark}


\subsection{Limiting configuration}\label{subsec: lim_config}
Finally, we prove the last part of our second main result stated in the introduction, \emph{i.e.} Theorem \ref{thm: main_2} \ref{itm: v_main_2}.

\begin{theorem}\label{thm: limiting_configuration}
	Let $(X^3,g)$ be an AC oriented $3$-manifold with connected link $(\Sigma^2,g_{\Sigma})$, and let $P\to X$ be a principal $\rm SU(2)$-bundle. Let $(A,\Phi)\in\mathscr{C}(P)$ be a solution to the second order equations \eqref{eq:2nd_Order_Eq_1} and \eqref{eq:2nd_Order_Eq_2}. Denote by $m$ the finite mass of $(A,\Phi)$, given by Theorem \ref{thm: finite_mass}, and suppose that $m>0$. Assume that $|F_A|,|\nabla_A\Phi|\in C_{-2}^0(X)$. Then there exists a principal $\rm SU(2)$-bundle $P_{\infty}\to\Sigma$ and a smooth configuration $(A_{\infty},\Phi_{\infty})\in\mathscr{A}(P_{\infty})\times\Gamma(\mathfrak{su}(2)_{P_{\infty}})$ such that the following hold:
	\begin{itemize}
		\myitem[(i)]\label{itm: i_limit} $(A,\Phi)|_{\Sigma_R}\to (A_{\infty},\Phi_{\infty})$ uniformly as $R\to\infty$.
		\myitem[(ii)]\label{itm: ii_limit} $\nabla_{A_{\infty}}\Phi_{\infty} = 0$.
		\myitem[(iii)]\label{itm: iii_limit} $A_{\infty}$ is a reducible Yang--Mills connection on $(\Sigma^2,g_{\Sigma})$.
	\end{itemize}
\end{theorem}
\begin{remark}
	According to Theorem \ref{thm: quadratic_decay_nablaPhi} and Theorem \ref{thm: quadratic_decay}, the quadratic decay assumptions $|F_A|,|\nabla_A\Phi|\in C_{-2}^0(X)$ on the above Theorem \ref{thm: limiting_configuration} holds true if we suppose at least one of the following holds:
	\begin{itemize}
		\myitem[($\dagger$)] $(A,\Phi)$ is a monopole, \emph{i.e.} a solution to equation \eqref{eq: monopole}.
		\myitem[($\dagger\dagger$)] $\Sigma$ has positive Gaussian curvature.
	\end{itemize} Therefore the result of Theorem \ref{thm: limiting_configuration} implies Theorem \ref{thm: main_2} \ref{itm: v_main_2}.
\end{remark}
\begin{proof}[Proof of Theorem \ref{thm: limiting_configuration}]
	To avoid cumbersome notation, in this proof the symbol ``$\lesssim$" actually means ``$\lesssim_{A,\Phi}$", \emph{i.e.} the implicit constant may depend on $(A,\Phi)$. By assumption, we have
	\begin{equation}
		|F_A|_g^2\lesssim \rho^{-4}\quad\text{and}\quad |\nabla_A\Phi|_g^2\lesssim \rho^{-4}.
	\end{equation}
	Now consider the cylinders $C_R = \rho^{-1}([R,R+1])$ with the conical metric $g_C$ which for large $R\gg 1$ approximates well the metric $g$. Then, we rescale $g_C$ by $r^{-2}$ to obtain the cylindrical metric 
	\begin{equation}
		h = r^{-2}g_C = d t^2 + g_\Sigma,
	\end{equation}
	where $t = \log (r)$. With respect to this translation-invariant metric we can identify all the cylinders $C_R$ with $(C = [0,1]_t \times \Sigma,h)$. Moreover, from the above we have\footnote{Note that $F_A$ is a bundle-valued $2$-form while $\nabla_A\Phi$ is a bundle-valued $1$-form, that's why we get different estimates for each of them in the cylindrical metric $h$.}
	\begin{equation}\label{ineq: bounds_cylinder}
		|F_A|_h^2 \lesssim 1\quad\text{and}\quad |\nabla_A\Phi|_h^2 \lesssim e^{-2t}.
	\end{equation}
	In particular, the restrictions $A_i = A|_{C_i}$ seen as connections over $C$ have uniformly bounded curvature with respect to $h$. Thus, Uhlenbeck's compactness results \cite{uhlenbeck1982connections} apply and by possibly passing to a subsequence, $A_i$ converges modulo gauge, as $i \to \infty$, to a connection $A_\infty$ on $C$.
	
	We now argue that such a limiting connection is unique and does not depend on the subsequence. For that consider $A_i$ on $C_i$ written in radial gauge with respect to $r$, \emph{i.e.} $A_i = a_i (r)$ with $a_i (\cdot)$ a $1$-parameter family of connections over $\Sigma$ parametrized by $r \in [R,R+1]$. Then $F_{A_i} = dr \wedge \del_r a_i(r) + F_{a_i}(r)$, where $F_{a_i}(r)$ is the curvature of $a_i(r)$ over $\lbrace r \rbrace \times \Sigma$. Using this, we find $|\del_r A_i|_g \leqslant |F_{A_i}|_r \lesssim r^{-2}$ and so
	\begin{equation}
		\int_R^{R+1} |\del_r A_i|_g dr \lesssim R^{-1},
	\end{equation}
	which decays as $R \to \infty$. This then shows that the limit 
	\begin{equation}
		A_\infty = \lim\limits_{r \to \infty} A (r),
	\end{equation}
	exists and it is independent of the coordinate $r$, so that it is a pullback of a connection over $\Sigma$. Thus, it agrees with the connection $A_\infty$ obtained as the uniform limit of the $A_i$, which is therefore pulled back from $\Sigma$.

	Now consider the restrictions $\Phi_R:=\Phi|_{C_R}$ seen as a 1-parameter family of Higgs fields in the fixed cylinder $C = [0,1]_t \times \Sigma$ with the fixed metric $h$. Then \eqref{ineq: bounds_cylinder} implies that $|\nabla_A \Phi_R|_h^2 \lesssim R^{-2}$ converges to zero as $R \to \infty$. This together with the uniform bound $|\Phi_R | \lesssim m$ (valid since $(A,\Phi)$ has finite mass $m>0$, see Remark \ref{rmk: finite_mass}) shows that $\Phi_R\to \Phi_\infty$ uniformly over $C$ with $\nabla_{A_\infty} \Phi_\infty = 0$. In particular, $\del_t \Phi_\infty = 0$ and so $\Phi_\infty$ is independent of $t$, or $r$, and so it is pulled back from $\Sigma$. This completes the proof of both \ref{itm: i_limit} and \ref{itm: ii_limit}.

	Finally, \ref{itm: i_limit} and \ref{itm: ii_limit} together with equations \eqref{eq:2nd_Order_Eq_1} and \eqref{eq:2nd_Order_Eq_2} immediately imply that $d_{A_\infty}^{\ast}F_{A_\infty}=0$, \emph{i.e.} ${A_\infty}$ is a Yang--Mills connection. Furthermore, $A_{\infty}$ is reducible since $|\Phi_{\infty}|=m>0$ and $\nabla_{A_{\infty}}\Phi_{\infty}=0$. This shows part \ref{itm: iii_limit}, thereby completing the proof the theorem. 

\end{proof}

\appendix

\section{Ricci curvature of an AC manifold in an adapted frame}\label{app: A}

Let $(X^n,g)$ be an AC manifold with only one end, dimension $n\geqslant 3$, rate $\nu>0$ and asymptotic link $(\Sigma^{n-1},g_{\Sigma})$. Fix $\rho$ a radius function on $X$ and denote by $g_C=dr^2 + r^2g_{\Sigma}$ the cone metric on $C:=(1,\infty)_r\times\Sigma^{n-1}$. 

Let $\{\tilde{e}^2,\tilde{e}^3,\ldots,\tilde{e}^n\}$ be an orthonormal coframe on $\Sigma$ with respect to $g_{\Sigma}$. Then, setting $e^1:=dr$ and $e^{\alpha}:=r\tilde{e}^{\alpha}$ for $2\leqslant \alpha \leqslant n$, the set $\{e^i\}_{i=1}^n$ forms an orthonormal coframe of the metric cone $(C,g_C)$. Denote by $\mathcal{R}_{\alpha\beta}^{\Sigma}$ and $\mathcal{R}_{ij}^C$ the Ricci tensors of $g_{\Sigma}$ and $g_C$ in the frames $\{\tilde{e}^{\alpha}\}_{\alpha=2}^n$ and $\{e^i\}_{i=1}^n$ respectively. Then a quick computation (see \cite[Equations (A.3) and (A.4) of Appendix A]{li2012geometric}) yields:
\begin{align}
	\mathcal{R}_{ij}^C =\begin{cases}\label{eq: Ricci_cone}
		0,\quad\text{if $i=1$ or $j=1$}\\
		r^{-2}\left(\mathcal{R}_{\alpha\beta}^{\Sigma}-(n-2)\delta_{\alpha\beta}\right),\quad\text{if $2\leqslant i=\alpha,j=\beta\leqslant n$}.
	\end{cases}
\end{align}
Now identify the end of $X$ with the cone $C$ and write the AC metric $g$ and its Ricci tensor along the end in the above frame $\{e^i\}_{i=1}^n$ as $g_{ij}$ and $\mathcal{R}_{ij}$ respectively. By the AC condition we have $|g_{ij}-\delta_{ij}|=O(r^{-\nu})$ and $|\mathcal{R}_{ij}-\mathcal{R}_{ij}^C|=O(r^{-2-\nu})$ as $r\to\infty$. Therefore, using \eqref{eq: Ricci_cone}, along the end we get
\begin{align}
	g_{ij} &= \delta_{ij} + O(\rho^{-\nu}),\quad\text{and}\\
	\mathcal{R}_{ij} &=\begin{cases}\label{eq: Ricci_AC_general}
		O(\rho^{-2-\nu}),\quad\text{if $i=1$ or $j=1$}\\
		r^{-2}\left(\mathcal{R}_{\alpha\beta}^{\Sigma}-(n-2)\delta_{\alpha\beta}\right) + O(\rho^{-2-\nu}),\quad\text{if $2\leqslant i=\alpha,j=\beta\leqslant n$}.
	\end{cases}
\end{align}
In particular, it follows that in general $|\mathcal{R}ic_g| = O(\rho^{-2})$. When $(X^n,g)$ is \emph{asymptotically euclidean} (AE), \emph{i.e.} when $(\Sigma,g_{\Sigma})=(\mathbb{S}^{n-1},g_{\mathbb{S}^{n-1}})$ is the round $(n-1)$-sphere, then $\mathcal{R}_{\alpha\beta}^{\Sigma} = (n-2)\delta_{\alpha\beta}$ and then we get that $|\mathcal{R}ic_g|=O(\rho^{-2-\nu})$ decays faster than quadratically.
 
Now let us restrict to the $3$-dimensional case, \emph{i.e.} when $n=3$. If $K^{\Sigma}$ denotes the Gaussian curvature of the surface $\Sigma^2$, then it is well known that
\[
\mathcal{R}_{\alpha\beta}^{\Sigma} = K^{\Sigma} \delta_{\alpha\beta}.
\] Thus, from \eqref{eq: Ricci_AC_general}, along the end we get
\begin{align}
	g_{ij} &= \delta_{ij} + O(\rho^{-\nu}),\quad\text{and}\\
	\mathcal{R}_{ij} &=\begin{cases}\label{eq: Ricci_AC}
		O(\rho^{-2-\nu}),\quad\text{if $i=1$ or $j=1$}\\
		\rho^{-2}\left(K^{\Sigma}-1\right)\delta_{ij} + O(\rho^{-2-\nu}),\quad\text{if $2\leqslant i,j\leqslant 3$}.
	\end{cases}
\end{align} 
When $(X^3,g)$ is AE, \emph{i.e.} $(\Sigma,g_{\Sigma})=(\mathbb{S}^2,g_{\mathbb{S}^2})$, then $K^{\Sigma}\equiv 1$ and we get $\mathcal{R}_{ij} = O(\rho^{-2-\nu})$.



\end{document}